\documentclass[11pt]{amsart}

\usepackage[margin=1.5in]{geometry}
\newskip\stdskip
 \usepackage{amsfonts, bbm}
\usepackage{amsmath}
\usepackage{amssymb}
\usepackage{amstext}
\usepackage{amsthm}
\usepackage{tikz}
\usepackage{hyperref,xargs,bbm}
\usepackage[all]{xy}
\usetikzlibrary{matrix}
\usetikzlibrary{shapes,calc,decorations.markings}

\usepackage{enumitem}

\newcommand{\C}{\mathbb C}
\newcommand{\R}{\mathbb R}

\newcommand{\Z}{\mathbb Z}

\newcommand{\coeffs}{\mathbbm{k}}

\newcommand{\cF}{\mathcal{F}}
\newcommand{\cG}{\mathcal{G}}
\newcommand{\cH}{\mathcal{H}}
\newcommand{\cI}{\mathcal{I}}

\newcommand{\cU}{\mathcal{U}}
\newcommand{\cV}{\mathcal{V}}

\newcommandx{\sh}[3][1=\R^2,2=\Lambda,3=\kk]{\mathbf{Sh}^\bullet_{#2}(#1,#3)}
\newcommandx{\csheaf}[2][1=\kk,2=X]{\underline{#1}_{#2}}

\newcommand{\Ext}{\operatorname{Ext}}

\newcommand{\Sh}{\mathcal{S}h}
\newcommand{\Fun}{\operatorname{Fun}}
\newcommand{\Ob}{\operatorname{Ob}}

\newcommand{\Hom}{\operatorname{Hom}}
\newcommand{\sHom}{\mathcal{H}om}
\newcommand{\Rep}{\mathcal{R}ep}
\newcommand{\Aug}{\mathcal{A}ug}

\newcommand{\Id}{\mathrm{Id}}
\newcommand{\std}{\mathrm{std}}
\newcommand{\ssm}{\smallsetminus}
\newcommand{\twomatrix}[4]{{\left(\begin{matrix} #1 & #2 \\ #3 & #4 \end{matrix}\right)}}

\newcommand{\img}{\operatorname{Im}}
\newcommand{\coker}{\operatorname{coker}}
\newcommand{\ssabut}{ \ \ \ \Longrightarrow\ \ \  }

\newcommand{\Mat}{\operatorname{Mat}}
\newcommand{\LCC}{\mathit{LCC}}
\newcommand{\LCH}{\mathit{LCH}}
\newcommand{\HW}{\mathit{HW}}

\newcommand{\kk}{\mathbbm{k}}
\newcommand{\A}{\mathcal{A}}

\newcommand{\e}{\epsilon}
\newcommand{\End}{\operatorname{End}}

\newtheorem{thm}{Theorem}[section]
\newtheorem{lemma}[thm]{Lemma}
\newtheorem{cor}[thm]{Corollary}
\newtheorem{prop}[thm]{Proposition}
\newtheorem{conj}[thm]{Conjecture}

\theoremstyle{definition}
\newtheorem{dfn}[thm]{Definition}
\newtheorem{rem}[thm]{Remark}

\begin{document}
\author[B. Chantraine]{Baptiste Chantraine}
\author[L. Ng]{Lenhard Ng}
\author[S. Sivek]{Steven Sivek}

\address{Universit\'e de Nantes}
\email{baptiste.chantraine@univ-nantes.fr}

\address{Duke University}
\email{ng@math.duke.edu}

\address{Imperial College London}
\email{s.sivek@imperial.ac.uk}

\title{Representations, sheaves, and Legendrian $(2,m)$ torus links}

\begin{abstract}
We study an $A_\infty$ category associated to Legendrian links in $\R^3$ whose objects are $n$-dimensional representations of the Chekanov--Eliashberg differential graded algebra of the link. This representation category generalizes the positive augmentation category and we conjecture that it is equivalent to a category of sheaves of microlocal rank $n$ constructed by Shende, Treumann, and Zaslow. We establish the cohomological version of this conjecture for a family of Legendrian $(2,m)$ torus links.
\end{abstract}

\maketitle

\section{Introduction}
\label{sec:introduction}

\subsection{Main results and conjectures}

Legendrian contact homology \cite{Eliashberg,EliashbergGiventalHofer} is a powerful invariant of Legendrian links which associates a differential graded algebra $(\A_\Lambda,\partial_\Lambda)$, often called the \emph{Chekanov--Eliashberg DGA}, to a link $\Lambda$ in a contact manifold $Y$.  This noncommutative algebra is generated freely by Reeb chords of $\Lambda$, and its differential counts certain holomorphic disks in $\R\times Y$ with boundary on $\R\times\Lambda$.

Chekanov \cite{Chekanov} defined this invariant combinatorially for Legendrian links in the standard contact $\R^3$, and introduced the use of \emph{augmentations} to extract more computable information from the DGA.  These are DGA morphisms $\epsilon: (\A_\Lambda,\partial_\Lambda) \to (\coeffs,0)$, where $\coeffs$ is the base ring, and they produce chain complexes $\LCC_*^\epsilon(\Lambda) = (\ker\epsilon / (\ker\epsilon)^2, \partial^\epsilon)$ for what is called \emph{linearized contact homology}.  The set of homologies $\{\LCH_*^\epsilon(\Lambda)\}$ over all $\epsilon$ is an effective invariant of $\Lambda$, which has been heavily studied because it is much easier to use in practice than the full DGA.

In recent years, it has become clear that these invariants have a much richer structure than simply a set: they can be organized into an $A_\infty$ category \cite{BC}, and even a unital $A_\infty$ category $\Aug_+(\Lambda,\coeffs)$ \cite{NRSSZ}, which is an invariant of $\Lambda$ up to $A_\infty$ equivalence.  The objects of this \emph{augmentation category} are the augmentations of $(\A_\Lambda,\partial_\Lambda)$, and the morphism spaces $\Hom(\epsilon_1,\epsilon_2)$ are an appropriate generalization of the linearized cochain complexes.

Meanwhile, Shende, Treumann, and Zaslow \cite{STZ} (building on \cite{GKS}) defined for each Legendrian link $\Lambda$ in $\R^3$ a differential graded category $\Sh(\Lambda,\coeffs)$, whose objects are constructible sheaves of $\coeffs$-modules on the plane with singular support controlled by the front projection of $\Lambda$.  Evidence suggested that in some cases, the sheaves with ``microlocal rank 1'' corresponded bijectively to augmentations of $\Lambda$, and this was no coincidence: the main theorem of \cite{NRSSZ} asserts that the full subcategory $\Sh_1(\Lambda,\coeffs)$ of such sheaves is equivalent to $\Aug_+(\Lambda,\coeffs)$ over any field $\coeffs$.  This provides a link between the disparate worlds of holomorphic curve invariants in contact geometry and homological algebra invariants. 

The goal of this paper, then, is to begin investigating natural generalizations of this theorem.
Beyond augmentations, a natural way to get a handle on the Chekanov--Eliashberg DGA is by studying its \emph{representations}, and in fact we expect that the finite-dimensional representation theory of this DGA over a field $\coeffs$ is contained in $\Sh(\Lambda,\coeffs)$.  Specifically, it is reasonable to guess that objects of the analogous category $\Sh_n(\Lambda,\coeffs)$ of sheaves with microlocal rank $n$ should correspond not to augmentations, but rather to $n$-dimensional representations of $(\A_\Lambda, \partial_\Lambda)$, cf.\ \cite{NgRutherford},
with an augmentation being a 1-dimensional representation.  Two possible ways to generalize the $A_\infty$ structure of $\Aug_+(\Lambda,\coeffs)$ to an $A_\infty$ category of $n$-dimensional representations were proposed in \cite{CDGG}.

In the present paper we pick out one of these constructions in particular and compare the resulting category to the sheaf category $\Sh_n(\Lambda,\coeffs)$ in a family of examples.

More precisely, given any integer $n \geq 1$, a field $\coeffs$, and a Legendrian link $\Lambda \subset \R^3$, we construct an $A_\infty$ category called the \textit{representation category} $\Rep_n(\Lambda,\coeffs)$. The objects of this category are $n$-dimensional representations of the Chekanov--Eliashberg differential graded algebra $(\A_\Lambda,\partial_\Lambda)$ associated to $\Lambda$, that is, algebra maps
\[
\rho :\thinspace \A_\Lambda \to \Mat_n(\coeffs),
\]
where $\Mat_n(\coeffs)$ is the $\coeffs$-algebra of $n\times n$ matrices over $\coeffs$, such that $\rho$ satisfies $\rho \circ \partial_\Lambda = 0$. When $n=1$, the category $\Rep_1(\Lambda,\coeffs)$ agrees with the positive augmentation category $\Aug_+(\Lambda,\coeffs)$ from \cite{NRSSZ}. The definition of $\Rep_n(\Lambda,\coeffs)$ is given in Section~\ref{sec:repr-categ-legendr}, though it is essentially contained in the more general discussion of ``noncommutative augmentation categories'' in \cite{CDGG}.

If the main result of \cite{NRSSZ} is that the augmentation category $\Aug_+(\Lambda,\coeffs)$ is equivalent to the sheaf category $\Sh_1(\Lambda,\coeffs)$, then we expect that even beyond just considering objects, the representation category $\Rep_n(\Lambda,\coeffs)$ should itself correspond on the sheaf side to the category $\Sh_n(\Lambda,\coeffs)$ of constructible sheaves with microsupport on $\Lambda$ and microlocal rank $n$.  We formulate the following conjecture.

\begin{conj}[``Representations are sheaves'']
Let $\Lambda$ be a Legendrian link in $\R^3$, and let $\coeffs$ be a field. Then for any $n\geq 1$, there is an $A_\infty$ equivalence of $A_\infty$ categories
\label{conj:main}
\[
\Rep_n(\Lambda,\coeffs) \stackrel{\sim}{\to} \Sh_n(\Lambda,\coeffs).
\]
\end{conj}
In particular, on the level of cohomology, Conjecture~\ref{conj:main} would immediately imply the following result.

\begin{cor}
If $\Lambda$ is a Legendrian link, 
\label{cor:main}
then the cohomology categories $H^*\Rep_n(\Lambda,\coeffs)$ and $H^*\Sh_n(\Lambda,\coeffs)$ are equivalent.
\end{cor}

In this paper, we will not attempt to prove Conjecture~\ref{conj:main} or even Corollary~\ref{cor:main} in general. However, we will 
prove the following special case of Corollary~\ref{cor:main}.

\begin{thm}
For $m \geq 1$, 
\label{thm:princ}
let $\Lambda_m$ be the Legendrian $(2,m)$ torus link given by the rainbow closure of the braid $\sigma_1^m\in B_2$, as shown in Figure~\ref{fig:2mknot}, equipped with its standard binary Maslov potential. Then the cohomology categories $H^*\Rep_n(\Lambda,\coeffs)$ and $H^*\Sh_n(\Lambda,\coeffs)$ are equivalent.
\end{thm}

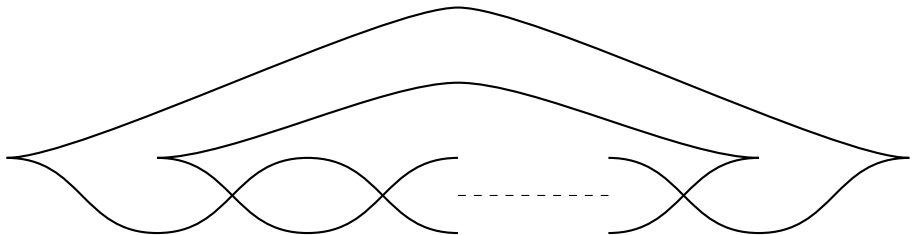
\begin{figure}[ht!]
  \centering
\begin{tikzpicture}
    \draw [thick] (0,1) .. controls +(1,0) and +(-1,0)
    .. (2,0) ..  controls +(1,0) and +(-1,0) .. (4,1); \draw [thick] (6,1)  .. controls +(1,0) and +(-1,0) .. (8,0) .. controls +(1,0) and +(-1,0) ..  (10,1)  .. controls +(-1,0) and +(1,0) ..  (4,3)  .. controls +(-1,0) and +(1,0) ..  (-2,1)  .. controls +(1,0) and +(-1,0) ..  (0,0) .. controls +(1,0) and +(-1,0) .. (2,1) ..  controls
    +(1,0) and +(-1,0) .. (4,0);
\draw [thick] (0,1) .. controls +(1,0) and +(-1,0) .. (4,2) ..  controls
    +(1,0) and +(-1,0) .. (8,1).. controls +(-1,0) and +(1,0) .. (6,0);
\path
(4,0.5) edge [dashed] (6,0.5);
  \end{tikzpicture}
  
  \caption{The front projection of the Legendrian link $\Lambda_m$, given as the rainbow closure of the $2$-braid $\sigma_1^m$.}
  \label{fig:2mknot}
\end{figure}

To prove Theorem~\ref{thm:princ}, we will explicitly compute the representation and (cohomological) sheaf categories associated to $\Lambda_m$ in a fair amount of detail, and show that they are identical. The computation on the sheaf side is likely to seem excessively detailed to a reader who is comfortable with the underlying sheaf theory. However, one of the goals of this paper is to break down the abstract sheaf calculation for $\Lambda_m$ into computations that essentially reduce to elementary linear algebra, with the hope of providing a concrete entry point into the sheaf theory for contact geometers.

\begin{rem}
Ideally, we would prove that there is an $A_\infty$ equivalence between $\Rep_n(\Lambda_m)$ and $\Sh_n(\Lambda_m)$ themselves, instead of only working at the level of cohomology.  However, our computation does not descend to this level because we compute $\Ext$ groups in $H^*(\Sh_n(\Lambda_m))$ in terms of extensions rather than working directly with complexes in $\Sh_n(\Lambda_m)$. One could hope to prove the desired equivalence by computing the appropriate Hochschild cohomology groups of these categories: if they are formal then the dg and $A_\infty$ structures are irrelevant for these examples, and if not then the deformation theory of these categories may still be simple enough to deduce their equivalence, by the sort of strategy outlined in \cite{Seidel-ICM}.
\end{rem}

\begin{rem}
It should be possible to define the representation category $\Rep_n$ for Legendrian submanifolds in arbitrary $1$-jet spaces. This would follow the definition of the positive augmentation category $\Aug_+$ for these Legendrians, which as of now has not been completely written down, though see \cite{EL} for work in this direction. We would then expect Conjecture~\ref{conj:main} to hold for Legendrians in all $1$-jet spaces.

Note also that there is no particular reason to separate sheaves by their microlocal rank, just as there is no reason to only consider $n$-dimensional representations of $(\A_\Lambda,\partial_\Lambda)$ for a single value of $n$.  One could consider the category $\Sh_{<\infty}(\Lambda,\coeffs)$ of all sheaves with finite microlocal rank, and similarly build a category $\Rep_{<\infty}(\Lambda,\coeffs)$ of finite-dimensional representations.  The computations in the present paper should extend to this more general setting with no extra difficulties other than careful bookkeeping.  However, we will stick to comparing $\Rep_n$ and $\Sh_n$ for fixed $n$, as this is already enough to motivate Conjecture~\ref{conj:main} and suggest that our construction of $\Rep_n$ is the correct one.
\end{rem}

\subsection{Geometric motivation}

Here we discuss some geometric intuition behind Conjecture~\ref{conj:main}. Our discussion is intended to provide some heuristic motivation and is consequently a sketch omitting full technical details.

First note that an exact Lagrangian filling $L$ of a Legendrian submanifold $\Lambda$ induces an augmentation of the DGA for $\Lambda$. By arguments of Seidel and Ekholm \cite{Ekholm_FloerlagCOnt}, there is an exact triangle relating:
the wrapped Floer homology of $L$, $\HW(L)$, as defined in \cite{WrappedFuk}; the linearized contact homology of $\Lambda$; and the homology of $L$, computed as the infinitesimally wrapped Floer homology of $L$.
Hence when $\HW(L)$ vanishes, we have the so-called Seidel--Ekholm--Dimitroglou-Rizell isomorphism between linearized contact homology and $H_*(L)$. This occurs for instance when $L$ is in the symplectization of the contact manifold, or more generally $L$ lies in a Weinstein filling and avoids its Lagrangian skeleton, see \cite[Proposition 7.6]{generation}. 
In \cite[Theorem 4]{EL} it is shown that when $\HW(L)=0$ this isomorphism lifts to the level of categories, leading an equivalence between some version of the augmentation category of $\Lambda$ and the infinitesimally wrapped Fukaya category.

On the sheaf side, consider the following example: 
let $N$ be a submanifold of a manifold $Q$ with smooth boundary $\partial \overline{N}$ such that $N=\{f > 0\}$ and $\partial \overline{N}=\{f=0\}$ for a smooth function $f$. The boundary at infinity of the microsupport of the constant sheaf on $N$ (i.e.  $i_*\C_N$, where $i :\thinspace N \to Q$ is inclusion) is the cosphere bundle of $\partial \overline{N}$, and its derived self hom space is the de Rham complex $\Omega^*(N)$ (see \cite[Lemma 4.4.1]{NaZa}).   The graph of $d(\ln f\vert_N)$ is a Lagrangian filling of this cosphere bundle, and it has the topology of $N$, so its Floer homology is quasi-isomorphic to the hom space in the sheaf category. This is a particularly relevant example: it follows from work of Nadler--Zaslow that such sheaves generate the full category $\Sh(Q)$, which allows them to prove that the category of sheaves corresponds to the derived infinitesimally wrapped category (see \cite{NaZa} and \cite{Nadler_brane}). 

Combining the two sides, we expect that the augmentation and sheaf categories for $\Lambda$ should be equivalent when $\HW = 0$. This last condition is why Conjecture~\ref{conj:main} is stated for $\Lambda$ in $\mathbb{R}^3$ (which might more generally be replaced by a $1$-jet space), as opposed to a unit cotangent bundle $ST^*Q$ as is more natural for the sheaf picture.

Now instead of exact Lagrangian fillings of $\Lambda$, suppose that we consider exact Lagrangian fillings $L$ equipped with rank $n$ local systems. On the sheaf side, this leads to the category $\Sh_n(\Lambda,\coeffs)$. On the DGA side, the filling $L$ produces a ``universal'' augmentation from the DGA $(\A_\Lambda,\partial_\Lambda)$ to the group ring $\Z[\pi_1(L)]$; this is the standard augmentation corresponding to $L$, which counts holomorphic disks with boundary on $L$ and positive end at a Reeb chord of $\Lambda$, but now each such disk contributes the homotopy class of its boundary. Composing this with a rank $n$ local system on $L$ leads to an $n$-dimensional representation of $(\A_\Lambda,\partial_\Lambda)$. Collecting these representations into a category yields the representation category $\Rep_n(\Lambda,\coeffs)$. As before, one could then guess that the representation and sheaf categories should be equivalent, and this is the content of Conjecture~\ref{conj:main}.

\begin{rem}
One caveat: the discussion above would produce a correspondence between derived categories, whereas Conjecture~\ref{conj:main} is for the original $A_\infty$ categories before triangulation. However, the main result of \cite{NRSSZ} establishes Conjecture~\ref{conj:main} for $n=1$ without passing to the derived categories, and we believe that this should continue to hold for arbitrary $n$. We note that the methods of \cite{NRSSZ} do not immediately generalize to the setting $n \geq 2$; one key difference is that the local computations in that paper rely on a classification of simple representations of a particular quiver algebra that does not extend to higher rank representations.
\end{rem}

One should be able to define an explicit $A_\infty$ functor from $\Rep$ to $\Sh$ yielding the equivalence in Conjecture~\ref{conj:main}. 
This builds on a construction of Viterbo for Lagrangians in cotangent bundles, 
and a map along these lines has been suggested independently by F.\ Bourgeois and V.\ Shende.
Roughly speaking, the idea is as follows: given an $n$-dimensional representation $\rho$ of the DGA of a Legendrian $\Lambda \subset J^1(Q) \subset ST^*(Q\times\R)$, for any $q_0 \in Q\times\R$ one can consider the linearized contact homology complex generated by Reeb chords from $\Lambda$ to the Legendrian cotangent fiber over $q_0$, linearized with respect to $\rho$. Then we map $\rho$ to a sheaf whose stalk over $q_0$ is this complex, which has microsupport on $\Lambda$ and microlocal rank $n$.

Note that this construction works for Legendrians in $ST^*(Q\times\R)$ and not just $J^1(Q)$. However, to extend the map on objects to a map on morphisms, one again needs some vanishing condition on $\HW$, as is guaranteed e.g.\ by restricting to $J^1(Q)$.
Without this vanishing, as an example, the conormal in $ST^*\R^2$ of a point in $\R^2$ is the microsupport of the skyscraper sheaf on the point whose hom space is $1$-dimensional (corresponding to the homology of the disk filling), whereas the corresponding hom space in the augmentation category is $2$-dimensional (the homology of $S^1$); the difference comes from the fact that the cotangent fiber has nonzero wrapped homology in the sector $T^*D^2$, see \cite{GPS}. 

Much of the preceding discussion is speculative both on the analytic side (for general Weinstein manifolds, thought of as the cotangent of arboreal singularities) and on the algebraic side. Note that recently the sheaf formulation of wrapped Floer homology (as suggested in \cite{wrappedsheaves} for an approach) has been used in \cite[Section 6.4]{GPS_Morsecat} to prove an equivalence between some sheaf category and modules over an algebra which is conjectured in \cite{EL} to be related to some version of the Chekanov--Eliashberg algebra. It is not currently fully clear to us what the precise relation is between Conjecture \ref{conj:main} and the results from \cite{GPS_Morsecat}.

\subsection{Organization}

The remainder of the paper is organized as follows. In Section~\ref{sec:repr-categ-legendr} we recall the relevant construction from \cite{CDGG}, and give details of the construction of $\Rep_n$ and its invariance under Legendrian isotopy (following \cite{NRSSZ}). Then in Section~\ref{sec:repr-categ-2} we proceed to compute this category for a family of Legendrian $(2,m)$ torus links. In Section~\ref{sec:category-shlambda} we recall the construction of the categories $\Sh(\Lambda)$ and the definition of microlocal rank. In Section~\ref{sec:sheaf-comp} we compute most of the cohomology of the sheaf category for the same family of Legendrian knots, leaving the vanishing of $\Ext^2$ and higher to Section~\ref{sec:vanishing-ext_2}. Finally, in Section~\ref{sec:constr-equiv} we gather all of these computations together and prove Theorem \ref{thm:princ}.

\subsection*{Acknowledgments}
During our exploration of the sheaf category we benefited greatly from conversations with Steven Boyer, St\'ephane Guillermou, Dan Rutherford, Vivek Shende, David Treumann, Nicolas Vichery, and Eric Zaslow; we also thank the referee for helpful feedback. The first author is partially supported by the ANR projects COSPIN (ANR-13-JS01-0008-01) and QUANTACT (ANR-16-CE40-0017). The second author is partially supported by NSF grants DMS-1406371 and DMS-1707652. Some of the early work for this paper stemmed from discussions at the 2016 workshop ``Interactions of gauge theory with contact and symplectic topology in dimensions 3 and 4'' at the Banff International Research Station.

\section{The representation category of a Legendrian link}
\label{sec:repr-categ-legendr}

In this section we describe the representation category $\Rep_n(\Lambda,\coeffs)$ associated to a Legendrian link $\Lambda$ and prove some elementary properties of this $A_\infty$ category. The construction of the representation category is due to \cite{CDGG}, who mention it briefly in the context of a more general discussion of ``noncommutative augmentation categories''. Since this particular category is of central importance to this paper, we will review its definition in a way that we hope can be read in a relatively self-contained way. However, we will assume that the reader is familiar with Chekanov--Eliashberg differential graded algebras associated to Legendrian links, about which there is now an extensive literature; a particular treatment that is especially convenient for our discussion can be found in \cite{NRSSZ}.

\begin{rem}
The portion of \cite{CDGG} that is most relevant to this paper is \cite[Appendix A]{CDGG}, which discusses noncommutative augmentation categories associated to Legendrian links. There are two categories proposed there that both generalize the positive augmentation category $\Aug_+$ to $n$-dimensional representations, corresponding to two general constructions of $A_\infty$ categories (``Case I'', ``Case II'') presented in \cite[\S 4]{CDGG}. Case II produces a category denoted in \cite{CDGG} by $\Rep_+(\Lambda,n)$, and this uses the fact that $\Mat_n(\coeffs)$ is a Hermitian algebra; however, this category is not unital (see \cite[Remark A.3]{CDGG}) and we do not consider it here. Instead we use the Case I category denoted in \cite{CDGG} by $\Rep_+(\Lambda,S)$, in the special case where $S = \Mat_n(\coeffs)$. The treatment of $\Rep_+(\Lambda,S)$ in \cite{CDGG} is rather cursory and we will expand on its definition below.
\end{rem}

\subsection{Twisting by noncommutative augmentations}
\label{ssec:Ainfty}
Before defining the representation category, we review a general technique given in \cite{CDGG} for producing an $A_\infty$ algebra by dualizing a Chekanov--Eliashberg differential graded algebra (or, more generally, a semi-free DGA in the sense of \cite{NRSSZ}) equipped with an augmentation. When the augmentation takes values in a field $\coeffs$, this story is well-studied; in the context of Legendrian knots, see e.g. \cite{CKESW,BC,NRSSZ}. In \cite[section 4.1]{CDGG}, this dualizing procedure is generalized to the case when the augmentation takes values in a possibly noncommutative algebra over $\coeffs$, and this is the version that we review here.

Let $\coeffs$ be a field, and $\{c_1,\ldots,c_p\}$ a finite set. For an algebra $A$ over $\coeffs$ (which for us will typically be $\Mat_n(\coeffs)$), we will denote by $M_{A}$ the free $A$-$A$-bimodule with generating set $\{c_1,\ldots,c_p\}$. We assume that $M_{A}$ is equipped with a grading coming from assigning degrees to each of the $c_i$'s and that elements of the coefficient ring have degree $0$. We will consider the following tensor algebra:
\[
\A_A = \mathcal{T}_A(M_{A}) = \bigoplus_{n=0}^\infty M_{A}^{\boxtimes n},
\]
where $\boxtimes = \otimes_A$ (with $M_A^{\boxtimes 0} = A$) and multiplication in $\A_A$ is given by concatenation.

The grading on $M_{A}$ induces a grading on $\A_A$. A differential $\partial$ on $\A_A$ is a degree $-1$ map satisfying the (signed) Leibniz rule $\partial(xy) = (\partial x)y+(-1)^{|x|}x(\partial y)$ along with $\partial(\alpha) = 0$ for all $\alpha\in A$. This differential is called \textit{augmented} if all constant terms in $\partial$ vanish: more precisely, $\partial$ maps $\A_A$ to $\bigoplus_{n=1}^\infty M_{A}^{\boxtimes n} \subset \A_A$.

If $\A_A$ is equipped with an augmented differential $\partial$, we can dualize this differential graded algebra to an $A_\infty$ algebra.
Let $M^\vee$ denote the dual of $M_{A}$, $M^\vee = \Hom_{A-A}(M,A)$. An element of $M^\vee$ is uniquely determined by where it sends $c_1,\ldots,c_p$ in $A$, and so we may identify $M^\vee$ with the free $A$-module generated by $c_1^\vee,\ldots,c_p^\vee$, where if $\phi \in M^\vee$ satisfies $\phi(c_i) = a_i$ then $\phi$ is identified with $\sum a_i c_i^\vee \in M^\vee$.  We assign each generator the grading $|c_i^\vee| = |c_i|+1$, where $|c_i|$ is the grading of $c_i$ as an element of $\A$.

Given an augmented differential $\partial :\thinspace \A_A \to \A_A$, we can dualize the $k$-th order part of $\partial$ to obtain a $\coeffs$-multilinear map
\[
\mu_k :\thinspace (M^\vee)^{\otimes k} \to M^\vee,
\]
where the tensor product is over $\coeffs$ and not $A$. 

A $k$-th order term in $\partial(c_i)$,
\[
\partial(c_i) = \alpha_k c_{i_k} \alpha_{k-1} c_{i_{k-1}} \cdots \alpha_1 c_{i_1} \alpha_0 + \cdots
\]
with $\alpha_j \in A$, contributes a term to $\mu_k$:
\[
\mu_k(a_{i_1}c_{i_1}^\vee,\ldots,a_{i_k}c_{i_k}^\vee) = (-1)^\sigma(\alpha_k a_{i_k} \alpha_{k-1} a_{i_{k-1}} \cdots \alpha_1 a_{i_1} \alpha_0) c_i^\vee + \cdots,
\]
where the sign is determined by
\begin{equation} \label{eq:sigma}
\sigma = \frac{k(k-1)}{2} + \left(\sum_{p<q} |c_{i_p}^\vee||c_{i_q}^\vee|\right) + |c_{i_2}^\vee| + |c_{i_4}^\vee| + \cdots.
\end{equation}
Since $\partial^2=0$, the $\mu_k$ operations satisfy the $A_\infty$ relations; here we use the $A_\infty$ sign conventions of \cite[\S 2.3]{NRSSZ} (in particular, the Koszul sign rule) rather than those of \cite{CDGG}, and hence the signs are taken from \cite[\S 3.1]{NRSSZ}.

\begin{prop}
$(M^\vee,\{\mu_k\})$ is an $A_\infty$ algebra over $\coeffs$.
\end{prop}

\begin{proof}
This follows exactly as in \cite[Theorem~4.11]{CDGG}, except that the signs have been determined using \cite[\S 3.1]{NRSSZ}.
\end{proof}

The remainder of this subsection is devoted to a description of how to construct an augmented differential $\partial$ on $\A_A$ (and thus an $A_\infty$ algebra) from a general differential graded algebra equipped with an augmentation. (When $A = \coeffs$, this story is well-known, cf.\ \cite{Chekanov}). To this end, let $R$ be a second $\coeffs$-algebra, producing a tensor algebra $\A_R$ in the same way as the definition of $\A_A$, and suppose that we have a differential $\partial$ on $\A_R$ satisfying the Leibniz rule and sending $R \subset \A_R$ to $0$. We will consider augmentations to the first $\coeffs$-algebra $A$:

\begin{dfn}
An \textit{augmentation} of $(\A_R,\partial)$ to $A$ is a DGA map
\[
\epsilon:(\A_R,\partial)\rightarrow (A,0),
\]
i.e., a graded $\coeffs$-algebra map $\epsilon:\A_R\rightarrow A$ such that $\epsilon\circ\partial=0$.
\end{dfn}

\begin{rem}
To be specific, the case that will be of interest to us is when $(\A_R,\partial)$ is the Chekanov--Eliashberg DGA of a Legendrian link $\Lambda$ in $\R^3$ (more precisely, the fully noncommutative DGA with multiple base points, tensored with $\coeffs$). In this case the generating set $\{c_1,\ldots,c_p\}$ is the set of Reeb chords of the link, and the ring $R$ is the group ring $\coeffs[F_q]$ where $F_q$ is the free group with $q$ generators.  In other words, $\A_R$ is generated as a $\coeffs$-algebra by the Reeb chords $c_1,\ldots,c_p$ of $\Lambda$ along with elements $t_1^{\pm 1},\ldots,t_q^{\pm 1}$, where $t_i$ and $t_j$ do not commute for $i\neq j$. As usual for Legendrian contact homology, the grading on $\A_R$ comes from giving Reeb chords degrees corresponding to their Conley--Zehnder indices, and the differential $\partial :\thinspace \A_R \to \A_R$ counts holomorphic disks with boundary on $\R\times\Lambda$.

Furthermore, we will be primarily interested in the case $A = \Mat_n(\coeffs)$; then we will call an augmentation of $(\A_R,\partial)$ to $A$ a \textit{representation} of $\A_R$, and we will usually use $\rho$ instead of $\epsilon$ to denote the representation. See Definition~\ref{def:rep} below.
\end{rem}

Now suppose that we have a differential graded algebra $(\A_R,\partial)$ equipped with an augmentation $\epsilon$ to $A$. Then $\partial$ induces a differential $\partial_A$ on $\A_A$ by replacing all occurrences of $\gamma \in R$ in $\partial$ by $\epsilon(\gamma) \in A$, and $\epsilon$ then induces an augmentation of $(\A_A,\partial_A)$ to $A$. One can then twist $\partial_A$ by this augmentation to produce an augmented differential on $\A_A$.

We now work out this scheme in more detail.
Define a $\coeffs$-algebra map $\phi_\epsilon:\A_R\rightarrow \A_A$ by its action on generators:
\begin{align*}
\phi_\epsilon(c_i) &= c_i+\epsilon(c_i) & i&=1,\ldots,p \\
\phi_\epsilon(\gamma) &= \epsilon(\gamma) & \gamma &\in R \subset \A_R.
\end{align*}

We can then define a map $\partial_\epsilon :\thinspace \A_A \to \A_A$ by
\begin{align*}
\partial_\epsilon(c_i) &= \phi_\epsilon(\partial c_i) & i &= 1,\ldots,p \\
\partial_\epsilon(\alpha) &= 0 & \alpha &\in A,
\end{align*}
where $\partial_\epsilon$ is extended to all of $\A_A$ by the Leibniz rule.
We now have the following proposition, which by the preceding discussion allows us to produce an $A_\infty$ algebra from $(\A_A,\partial_\e)$.

\begin{prop}
  The map $\partial_\epsilon$ is an augmented differential on $\A_A$. \label{prop:augmented}
\end{prop}

\begin{proof}
First note that $\partial_\epsilon \circ \phi_\epsilon = \phi_\epsilon \circ \partial$: the two sides agree on generators of $\A_A$ by construction, and both sides behave the same way under products. It follows that $\partial_\epsilon^2(c_i) = \partial_\epsilon^2(\phi_\epsilon(c_i)) = \phi_\epsilon \partial^2(c_i) = 0$, and so $\partial_\epsilon^2 = 0$ on all of $\A_A$ since $\partial_\epsilon$ satisfies the Leibniz rule. Furthermore, the constant term of $\partial_\epsilon c_i$ is by construction $\epsilon(\partial c_i)$, which is $0$ since $\epsilon$ is an augmentation; thus $\partial_\epsilon$ is augmented.
\end{proof}

In what follows, we will need to compare the $A_\infty$ structures arising from equivalent (``stable tame'') DGAs. To do this, we observe that the above construction of an augmented differential $\partial_\e$ from an augmentation $\e$ is natural, in the following sense. Suppose that we have a DGA map $f :\thinspace (\A_R,\partial) \to (\A'_R,\partial)$, where $\A_R,\A'_R$ are both tensor algebras but with possibly different generating sets $\{c_1,\ldots,c_p\}, \{c_1',\ldots,c_{p'}'\}$; in particular, $f$ restricts to the identity map on $R$. Further suppose that we have an augmentation $\epsilon'$ of $(\A_R',\partial)$ to $A$; this induces an augmentation $\epsilon = \epsilon' \circ f$ of $(\A_R,\partial)$ to $A$. We construct maps $\phi_\e :\thinspace \A_R \to \A_A$, $\phi_{\e'} :\thinspace \A_R' \to \A_A'$, $\partial_\e :\thinspace \A_A \to \A_A$, $\partial_{\e'} :\thinspace \A_A' \to \A_A'$ as above, where $\A_A,\A_A'$ have the same generating sets as $\A_R,\A'_R$ respectively.

Define a $\coeffs$-algebra morphism $f_A :\thinspace \A_A \to \A_A'$ by
\begin{align*}
f_A(c_i) &= \phi_{\e'}f(c_i)-\e(c_i) & i&=1,\ldots,p \\
f_A(\alpha) &= \alpha & \alpha &\in A.
\end{align*}
We claim that $f_A$ intertwines the differentials $\partial_\e$ and $\partial_{\e'}$.
Indeed, note that we have $f_A \circ \phi_\e = \phi_{\e'} \circ f$ since these two compositions agree on generators of $\A_R$:
\[
f_A \phi_\e(c_i) = f_A(c_i+\e(c_i)) = \phi_{\e'}f(c_i)-\e(c_i)+\e(c_i) = \phi_{\e'}f(c_i)
\]
and $f_A \phi_\e(\gamma) = \e(\gamma) = \phi_{\e'} (\gamma) = \phi_{\e'}f(\gamma)$ for $\gamma \in R$. As in the proof of Proposition~\ref{prop:augmented}, we have $\partial_{\e'} \circ \phi_{\e'} = \phi_{\e'} \circ \partial$, whence
\[
f_A \partial_\e(c_i) = f_A \phi_\e \partial (c_i) = \phi_{\e'} f \partial(c_i) = \phi_{\e'} \partial f(c_i) = \partial_{\e'} \phi_{\e'} f(c_i) = \partial_{\e'} f_A(c_i),
\]
and we conclude that $f_A \circ \partial_\e = \partial_{\e'} \circ f_A$.

Now the constant term in $f_A(c_i)$ is $\e'f(c_i)-\e(c_i) = 0$, and so the map
$f_A$ has no constant terms: that is, the component of $f_A\vert_{M_A}$ mapping to $(M_A')^{\boxtimes 0}$ is $0$. Thus we can write $f_A\vert_{M}
=f_1\oplus f_2\oplus\cdots\oplus f_k\oplus\cdots$ where each $f_k$ has image in $(M_A')^{\boxtimes k}$. Dualizing all the $f_i$ as in the construction of the $\mu_k$, we get a family of maps $\{F_k\}_{k\in \mathbb{N}}$ which forms an $A_\infty$ algebra morphism $((M')^\vee,\{\mu_k'\}) \to (M^\vee,\{\mu_k\})$. The whole construction is functorial, so that if $f$ is an isomorphism then it induces a isomorphism of $A_\infty$ algebras.
\subsection{The representation category}

Following \cite{BC,CDGG}, one can expand the construction from Section~\ref{ssec:Ainfty} to construct an $A_\infty$ category whose objects are augmentations $\epsilon :\thinspace \A \to A$, such that the self-homs of $\epsilon$ form the $A_\infty$ algebra given above. In the case where $A = \Mat_n(\coeffs)$, this produces an $A_\infty$ category that we might term the ``negative representation category'', following the terminology of \cite{NRSSZ}. We will instead be interested in a ``positive'' version of the representation category, which requires a bit more work to set up and which we will now describe. Throughout this subsection, we assume that $A = \Mat_n(\coeffs)$, though the discussion works for more general algebras as well. In the case $n=1$, the representation category $\Rep_1(\Lambda,\coeffs)$ will precisely be the positive augmentation category $\Aug_+(\Lambda,\coeffs)$ from \cite{NRSSZ}.

\begin{dfn}
Let $(\A,\partial)$ be the Chekanov--Eliashberg DGA of a Legendrian link $\Lambda$ equipped with base points. An $n$-dimensional \textit{representation} of $(\A,\partial)$ is an augmentation $\rho :\thinspace \A \to \Mat_n(\coeffs)$, meaning a homomorphism of graded $\coeffs$-algebras (here $\Mat_n(\coeffs)$ lives entirely in grading 0) with $\rho\circ\partial=0$.
\label{def:rep}
\end{dfn}

\begin{rem}
One can more generally consider DGA maps from $(\A,\partial)$ to $(\End V,\delta)$, where $V$ is a graded $\coeffs$-vector space and $\delta$ is induced from a differential on $V$. The representations that we consider in this paper correspond to $\delta=0$ and $V = \coeffs^n$ supported in degree $0$. The general DGA maps are of interest when studying augmentations of satellites \cite{LR}; one can similarly construct a category out of these, but to simplify the exposition we omit the details here.
\end{rem}

Given a Legendrian link $\Lambda$ with base points $\ast_1,\dots,\ast_q$, we can now define the representation category $\Rep_n(\Lambda,\coeffs)$, fleshing out the definition in \cite{CDGG}. 
Let $c_1,\ldots,c_p$ be the Reeb chords of $\Lambda$, and let $(\A,\partial)$ be the Chekanov--Eliashberg DGA of $\Lambda$.  Note again that $\A$ has invertible generators $t_1^{\pm1},\dots,t_q^{\pm1}$ for each base point which do not commute with the $c_i$ or with each other, except that $t_i t_i^{-1} = t_i^{-1} t_i = 1$.

\begin{dfn} \label{def:rep_n}
The representation category $\Rep_n(\Lambda,\coeffs)$ is the $A_\infty$ category over $\coeffs$ defined as follows.
\begin{enumerate}
\item
The objects of $\Rep_n(\Lambda,\coeffs)$ are $n$-dimensional representations of $(\A,\partial)$.
\item
Given two objects $\rho_1,\rho_2$, the morphism space $\Hom(\rho_1,\rho_2)$ is the free graded $\Mat_n(\coeffs)$-module generated by $c_1^\vee,\ldots,c_p^\vee,x_1^\vee,y_1^\vee,\dots,x_q^\vee,y_q^\vee$, the formal duals to the Reeb chords of $\Lambda$ along with $2q$ additional generators.  These have gradings $|c_i^\vee| = |c_i|+1$, $|x_j^\vee|=1$, and $|y_j^\vee|=0$.
\item
We have $\coeffs$-multilinear maps
\[
\mu_k :\thinspace \Hom(\rho_{k-1},\rho_k) \otimes \cdots \otimes \Hom(\rho_1,\rho_2) \otimes \Hom(\rho_0,\rho_1) \to \Hom(\rho_0,\rho_k)
\]
defined by dualizing the augmented differential of the $(k+1)$-copy of $\Lambda$ with respect to the pure augmentation $\epsilon=(\rho_0,\ldots,\rho_k)$, as in the previous subsection.
\end{enumerate}
\end{dfn}
To be more precise, we construct $\mu_k$ by using the \emph{Lagrangian projection $(k+1)$-copy} $\Lambda^{k+1}_f$ of $\Lambda$ as described in \cite[\S 4.2.2]{NRSSZ}. Here we choose $q$ base points $\ast_1,\ldots,\ast_q$ on $\Lambda$ with at least one base point on each component of $\Lambda$, and then choose a Morse function $f$ on $\Lambda$ with $2q$ critical points consisting of $1$ minimum and $1$ maximum in a small neighborhood of each $\ast_j$. We then use $f$ to produce $k+1$ parallel copies of $\Lambda$ by perturbing $\Lambda$ in the normal direction according to $f$.
  
Each $x_j^\vee$ and $y_j^\vee$ correspond to crossings between different components of the $(k+1)$-copy near $\ast_j$ at the maximum and minimum, respectively.  Labeling the copies $\Lambda_1,\dots,\Lambda_{k+1}$ from top to bottom and letting $z$ be any of $c_i,x_j,y_j$, we let $z^{i,j}$ denote the corresponding Reeb chord of $\Lambda_f^{k+1}$ from $\Lambda_j$ to $\Lambda_i$.  Each $\Lambda_i$ has base points $\ast_1^i,\dots,\ast_q^i$.

The DGA of $\Lambda^{k+1}_f$ is described explicitly by \cite[Proposition~4.14]{NRSSZ}, which we restate below as Proposition~\ref{prop:copydga}, and the pure augmentation $\epsilon$ satisfies $\epsilon(t_j^i) = \rho_i(t_j)$, $\epsilon(c_l^{i,j}) = \rho_i(c_l)$ if $i=j$, and $\epsilon(z)=0$ for all other generators.  Twisting the differential of this DGA by the pure augmentation $\epsilon$, each $k$-th order term of the form
\[ \partial^\epsilon(w^{1,k+1}) = \alpha_kz_k^{1,2}\alpha_{k-1}z_{k-1}^{2,3}\cdots\alpha_1z_1^{k,k+1}\alpha_0 + \dots \]
then contributes
\[ \mu_k(a_{1}z_1^\vee,\dots,a_{k}z_k^\vee) = (-1)^\sigma(\alpha_ka_k\alpha_{k-1}a_{k-1}\cdots \alpha_1 a_{1}\alpha_0)w^\vee + \cdots \]
with the sign $(-1)^\sigma$ as in equation~\eqref{eq:sigma}.

\begin{prop} \label{prop:its-a-category}
$(\Rep_n(\Lambda,\coeffs),\{\mu_k\})$ is a unital $A_\infty$ category over $\coeffs$.
\end{prop}

\begin{proof}
This category is defined as $\Rep_+(\Lambda,\Mat_n(\coeffs))$ in \cite[\S A.2.2]{CDGG} (but with sign conventions as in \cite{NRSSZ}).  There it is explained how the construction of a consistent sequence of DGAs from \cite{NRSSZ} adapts to the case where the objects are representations instead of just augmentations; then we can pass from this sequence to $\Rep_n(\Lambda,\coeffs)$ by \cite[Definition~3.16]{NRSSZ}, and it is an $A_\infty$ category by \cite[Proposition~3.17]{NRSSZ}. As observed in \cite[\S A.2.2]{CDGG}, the existence of a unit $e_\rho = -\sum_{j=1}^q y_j^\vee$ in each $\Hom(\rho,\rho)$ essentially follows from \cite[Proposition~3.28]{NRSSZ}; this unit is closed as a consequence of \cite[Theorem~5.5]{EES}.
\end{proof}

\begin{prop} \label{prop:invariance}
Fix $n$ and $\coeffs$.  Up to $A_\infty$ equivalence, the $A_\infty$ category $\Rep_n(\Lambda,\coeffs)$ does not depend on the choice of Morse function $f$.  Moreover, if $\Lambda$ and $\Lambda'$ are Legendrian isotopic, then $\Rep_n(\Lambda,\coeffs)$ and $\Rep_n(\Lambda',\coeffs)$ are $A_\infty$ equivalent.
\end{prop}

\begin{proof}
This follows exactly as in \cite[Theorem~4.20]{NRSSZ}; we outline the steps here.

First, following \cite[Proposition~4.21]{NRSSZ}, we suppose that $\Lambda$ and $\Lambda'$ are related by a Legendrian isotopy carrying the Morse function $f$ on $\Lambda$ to $f'$ on $\Lambda'$, and that $f$ (and hence $f'$) has exactly one local maximum on each component.  Then the Legendrian isotopy relating these links induces a stable tame isomorphism of DGAs, which induces a consistent family of DGA morphisms (in the sense of \cite[Definition 3.19]{NRSSZ})  as shown in \cite[Proposition 3.29]{NRSSZ}.  Using the construction of $A_\infty$ morphisms at the end of Section~\ref{ssec:Ainfty}, we produce an $A_\infty$ equivalence from this consistent family by following the proofs of \cite[Propositions~3.29~and~4.21]{NRSSZ} verbatim: tame isomorphisms of DGAs are easily seen to give $A_\infty$ equivalences, and then stabilization is treated by considering the functor associated to the inclusion $i:\mathcal{A}\rightarrow S(\mathcal{A})$ and doing the same verification by hand as in \cite[Theorem 2.14]{BC} and \cite[Proposition 4.21]{NRSSZ}.

Next, we repeat \cite[Proposition~4.22]{NRSSZ} to argue that for fixed $\Lambda$, the category $\Rep_n(\Lambda,\coeffs)$ is independent of the choice of $f$.  More precisely, each of two operations induces a consistent sequence of DGA morphisms, and hence an $A_\infty$ functor which is readily shown to be an equivalence.  These operations are replacing a single base point with multiple base points along a small arc of $\Lambda$, and pushing a base point through a crossing of $\Lambda$.  A small amount of extra bookkeeping is needed in this setting because the coefficients do not commute in general, unlike in the case of $\Aug_+(\Lambda,\coeffs)$, but the adaptation is straightforward.
\end{proof}

We conclude this section with a discussion of isomorphism in the representation category.
Here isomorphism of objects is defined as usual: $\rho_1,\rho_2$ are isomorphic if they are isomorphic in the cohomology category $H^*\Rep_n(\Lambda,\coeffs)$, in the sense that there are closed morphisms between $\rho_1$ and $\rho_2$ that compose to give the identity in cohomology.

\begin{prop} \label{prop:iso-criterion}
Let $\rho_1,\rho_2: \A \to \Mat_n(\coeffs)$ be two representations of the Chekanov--Eliashberg DGA $(\A,\partial)$.  If there is an invertible $M \in \Mat_n(\coeffs)$ such that $\rho_2 = M^{-1}\rho_1 M$, then $\rho_1$ is isomorphic to $\rho_2$.

Conversely, suppose that $\Lambda$ has a single base point and no Reeb chords in degree $-1$ or $-2$. If $\rho_1$ is isomorphic to $\rho_2$, then there is an invertible $M$ such that $\rho_2 = M^{-1}\rho_1M$.
\end{prop}

\begin{rem}
In the case $n=1$ of Proposition~\ref{prop:iso-criterion}, we see that if $\Lambda$ has one base point and no Reeb chords in degree $-1$ or $-2$, then no two distinct augmentations are isomorphic in $\Aug_+(\Lambda,\coeffs)$, cf.\ \cite[Proposition~5.17]{NRSSZ}.
\end{rem}

In the proof of Proposition~\ref{prop:iso-criterion} and in subsequent sections we will make use of a certain \emph{link grading} $r\times c$ from \cite[Definition~4.12]{NRSSZ}, which we recall here.
\begin{dfn} \label{def:link-grading}
Let $\Lambda$ be an oriented, pointed Legendrian link with base points $\ast_1,\dots,\ast_q$, at least one of which lies on each component of $\Lambda$.  We write $U_i$ for the component of $\Lambda \ssm \{\ast_1,\dots,\ast_q\}$ with initial endpoint $\ast_i$.  Then for each Reeb chord generator $c$ of the DGA $(\A,\partial)$, we define $r(c)$ and $c(c)$ to be the unique integers for which the upper and lower endpoints of $c$ lie on $U_{r(c)}$ and $U_{c(c)}$ respectively.  We also define $r(t_i)$ and $c(t_i)$ so that the base point $\ast_i$ is preceded by $U_{r(t_i)}$ and followed by $U_{c(t_i)}$.  (By definition $c(t_i)=i$.)
\end{dfn}

\begin{proof}[Proof of Proposition~\ref{prop:iso-criterion}]
In order to compute the differential $\mu_1$ on each $\Hom(\rho_i,\rho_j)$, we read off the DGA of $\Lambda_f^2$ from \cite[Proposition~4.14]{NRSSZ} (restated for convenience as Proposition~\ref{prop:copydga} below).  Letting $c_1,\dots,c_p$ denote the Reeb chords of $\Lambda_i$ and $(r,c)$ the link grading,
in $\Hom(\rho_1,\rho_2)$ we have
\begin{multline*}
\mu_1(u_iy_i^\vee) = \sum_{r(c_k)=i} u_i\rho_2(c_k)c_k^\vee - \sum_{c(c_k)=i} \rho_1(c_k)u_ic_k^\vee \\
+ \sum_{r(t_k)=i} \rho_1(t_k)^{-1} u_i \rho_2(t_k) x_k^\vee - \sum_{c(t_k)=i} u_i x_k^\vee.
\end{multline*}
(Here we have eliminated a sign $(-1)^{|a_k|}$ in front of $\rho_1(c_k)u_ic_k^\vee$ by observing that $\rho_1(c_k)=0$ if $|c_k| \neq 0$.)  Summing over all $i$, we have
\begin{multline*}
\mu_1\left(\sum_{i=1}^q u_iy_i^\vee\right) = \sum_{k=1}^p \left(u_{r(c_k)} \rho_2(c_k) - \rho_1(c_k)u_{c(c_k)}\right) c_k^\vee \\ 
+ \sum_{k=1}^q \left(\rho_1(t_k)^{-1} u_{r(t_k)} \rho_2(t_k) - u_{c(t_k)}\right) x_k^\vee,
\end{multline*}
so $f = \sum_i u_iy_i^\vee \in \Hom(\rho_1,\rho_2)$ is a degree-0 cycle if and only if $u_{r(z)}\rho_2(z) = \rho_1(z)u_{c(z)}$ for all $z = c_1,\dots,c_p, t_1,\dots,t_q$.  Similarly, $g = \sum_i u_i' y_i^\vee \in \Hom(\rho_2,\rho_1)$ is a degree-0 cycle if and only if $u'_{r(z)}\rho_1(z) = \rho_2(z)u'_{c(z)}$ for all $z$.

For the composition $\mu_2$, we examine the DGA of $\Lambda_f^3$: the only terms with degree 2 in the $y_k$ variables have the form $\partial y^{13}_k = y^{12}_k y^{23}_k$, from which we have $\mu_2(uy_i^\vee,u'y_j^\vee) = 0$ if $i\neq j$ and $\mu_2(uy_k^\vee,u'y_k^\vee) = -u'uy_k^\vee$.  Thus any two elements $f \in \Hom^0(\rho_1,\rho_2)$ and $g \in \Hom^0(\rho_2,\rho_1)$ of the form $f=\sum u_iy_i^\vee$ and $g=\sum u_i'y_i^\vee$ have compositions
\begin{align*}
\mu_2(g,f) &= -\sum_{i=1}^q u_iu_i'y_i^\vee \in \Hom^0(\rho_1,\rho_1), \\
\mu_2(f,g) &= -\sum_{i=1}^q u_i'u_iy_i^\vee \in \Hom^0(\rho_2,\rho_2),
\end{align*}
which are equal to the units $e_{\rho_j} = -\sum_{i=1}^q y_i^\vee$ if and only if $u_i' = u_i^{-1}$ for all $i=1,\dots,q$.

Putting these facts together, we see that the pair $(f,g)$ as above gives an isomorphism $\rho_1 \cong \rho_2$ if and only if there are invertible matrices $u_1,\dots,u_q$ such that
\begin{equation} \label{eq:isomorphism-criterion}
\rho_2(z) = u_{r(z)}^{-1} \rho_1(z) u_{c(z)}
\end{equation}
for all generators $z$ of $(\A,\partial)$.  If $\rho_1$ and $\rho_2$ are conjugate, with $\rho_2 = M^{-1}\rho_1M$, then we take $u_i=M$ for all $i$ to see that indeed $\rho_1 \cong \rho_2$.

For the converse, we suppose that $\Lambda$ has no Reeb chords in degree $-1$ or $-2$ and that the cycles $f \in \Hom^0(\rho_1,\rho_2)$ and $g \in \Hom^0(\rho_2,\rho_1)$ satisfy $[\mu_2(g,f)]=[e_{\rho_1}]$ and $[\mu_2(f,g)]=[e_{\rho_2}]$ in cohomology.  By assumption, each $\Hom^0(\rho_i,\rho_j)$ is spanned by the $y_i^\vee$, so we must have $f = \sum u_i y_i^\vee$ and $g = \sum u_i' y_i^\vee$ for some matrices $u_i$ and $u_i'$.  Moreover, we have $\Hom^{-1}(\rho_j,\rho_j)=0$ for all $j$, so any cycle in $\Hom^0(\rho_j,\rho_j)$ which is homologous to $e_{\rho_j}$ must in fact be equal to it.  In particular, we have $\mu_2(g,f) = e_{\rho_1}$ and $\mu_2(f,g) = e_{\rho_2}$, and it follows exactly as above that equation~\eqref{eq:isomorphism-criterion} holds.  The assumption that $\Lambda$ has exactly one base point implies that the link grading collapses to $r=c=1$; setting $M = u_1$, we conclude that $\rho_2 = M^{-1} \rho_1 M$ as claimed.
\end{proof}

When $\Lambda$ has multiple base points $\ast_1,\dots,\ast_q$, we can think of each representation $\rho: \A \to \Mat_n(\coeffs)$ as an action of $(\A,\partial)$ on $V^{\oplus q} = V_1 \oplus \dots \oplus V_q$, where $V = \coeffs^n$ and each $\rho(z)$ is an element of $\Hom(V_{c(z)}, V_{r(z)})$.  In other words, for each generator $z$ of $\A$ the endomorphism $\rho(z)$ is a $qn \times qn$ matrix, viewed as a $q\times q$ matrix of $n\times n$ blocks whose only nonzero block is in position $(r(z),c(z))$. For more on this perspective, see \cite[\S 3.2]{NRSSZ}.

In this case, if $\Lambda$ has no Reeb chords in degree $-1$ or $-2$, equation~\eqref{eq:isomorphism-criterion} says that two representations $\rho_1,\rho_2$ are isomorphic in $H^*\Rep_n(\Lambda,\coeffs)$ if and only if they are identified by some simultaneous change of bases for $V_1,\dots,V_q$, or equivalently iff they are conjugate by some invertible $qn\times qn$ matrix such that all off-diagonal $n\times n$ blocks are zero.  When $\Lambda$ is a knot with several base points, this seems like a much weaker notion of equivalence than in Proposition~\ref{prop:iso-criterion}, but this is explained by the fact that the additional base points introduce extra $t_i^{\pm 1}$ generators into the DGA and hence there are generally many more representations.

\section{The representation category of Legendrian \texorpdfstring{$(2,m)$}{(2,m)} torus links}
\label{sec:repr-categ-2}

In this section, we will use $\Lambda_m$ to denote the Legendrian $(2,m)$ torus link in $\R^3$ whose Lagrangian ($xy$) projection is shown in Figure~\ref{fig:2mknot-lagrangian}. This is Legendrian isotopic to the $\Lambda_m$ shown in Figure~\ref{fig:2mknot} in the Introduction, with the advantage that the Chekanov--Eliashberg DGA for this version of $\Lambda_m$ is rather simple. For the purposes of computing this DGA,
place base points $\ast_1,\ast_2$ along the loops to the right of $b_1,b_2$ (note that if $m$ is odd, this is one more base point than necessary).
\begin{figure}[ht!]
\centering
\begin{tikzpicture}[every path/.style=thick]
\draw (-4.4,0.75) .. controls +(0,1.75) and +(-2.5,2.5) .. (5,1) .. controls +(1,-1) and +(0,-0.5) .. (7,1) .. controls +(0,0.5) and +(0.9,0.9) .. (5.1,1.1);
\draw (-4.4,0.75) .. controls +(0,-0.5) and +(-0.7,0) .. (-3,-0.5) .. controls +(0.5,0) and +(-0.1,-0.1) .. (-2.1,-0.1);
\draw (-1.9,0.1) .. controls +(0.1,0.1) and +(-0.5,0) .. (-1,0.5) .. controls +(0.75,0) and +(-0.75,0) .. (1,-0.5);
\draw (3,-0.5) .. controls +(0.5,0) and +(-0.1,-0.1) .. (3.9,-0.1);
\draw (4.1,0.1) .. controls +(0.1,0.1) and +(-0.1,-0.1) .. (4.9,0.9);
\path (1.2,0) edge [dashed] (2.8,0);
\draw (-4.4,-0.75) .. controls +(0,-1.5) and +(-2.4,-2.4) .. (4.9,-1.1);
\draw (5.1,-0.9) .. controls +(0.9,0.9) and +(0,0.5) .. (7,-1) .. controls +(0,-0.5) and +(1,-1) .. (5,-1) .. controls +(-1,1) and +(0.6,0) .. (3,0.5);
\draw (-4.4,-0.75) .. controls +(0,0.4) and +(-0.1,-0.1) .. (-4.1,-0.1);
\draw (-3.9,0.1) .. controls +(0.1,0.1) and +(-0.5,0) .. (-3,0.5) .. controls +(0.75,0) and +(-0.75,0) .. (-1,-0.5) .. controls +(0.5,0) and +(-0.1,-0.1) .. (-0.1,-0.1);
\draw (0.1,0.1) .. controls +(0.1,0.1) and +(-0.5,0) .. (1,0.5);
\draw (-4,-0.4) node {$a_1$};
\draw (-2,-0.4) node {$a_2$};
\draw (0,-0.4) node {$a_3$};
\draw (4,-0.4) node {$a_m$};
\draw (5,0.6) node {$b_1$};
\draw (5,-1.4) node {$b_2$};
\draw (7,1) node {\huge$*$};
\draw (7.14,0.9) node {$\phantom{*}_1$};
\draw (7,-1) node {\huge$*$};
\draw (7.14,-1.1) node {$\phantom{*}_2$};
\end{tikzpicture}  
\caption{A Lagrangian projection of the Legendrian $(2,m)$ torus link $\Lambda_m$.}
\label{fig:2mknot-lagrangian}
\end{figure}
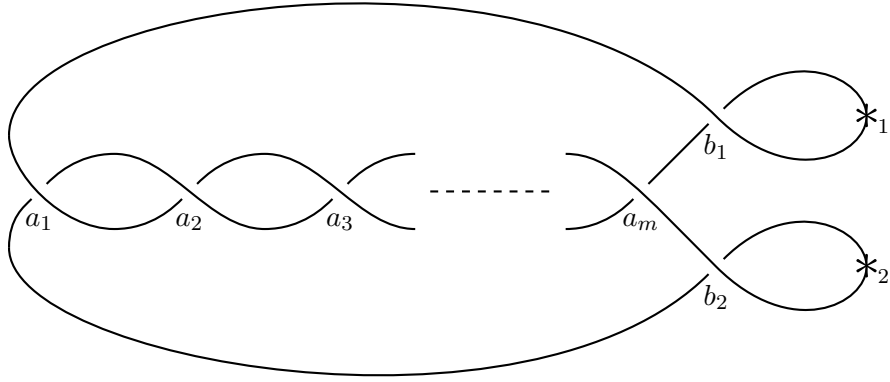

Let $a_1,a_2,\ldots$ denote a collection of noncommuting variables. Define two sequences of polynomials $\{P_m\}_{m=0}^\infty$ and $\{Q_m\}_{m=0}^\infty$ as follows:
\begin{align*}
P_0 &= 1 \\
P_1(a_1) &= a_1 \\
P_m(a_1,\ldots,a_m) &= P_{m-1}(a_1,\ldots,a_{m-1})a_m + P_{m-2}(a_1,\ldots,a_{m-2}) \\
Q_0 &= 1 \\
Q_1(a_1) &= -a_1 \\
Q_m(a_1,\ldots,a_m) &= -Q_{m-1}(a_2,\ldots,a_m)a_1 + Q_{m-2}(a_3,\ldots,a_m).
\end{align*}
Thus $P_2(a_1,a_2)=1+a_1a_2$, $Q_2(a_1,a_2) = 1+a_2a_1$, $P_3(a_1,a_2,a_3) = a_1+a_3+a_1a_2a_3$, $Q_3(a_1,a_2,a_3) = -a_1-a_3-a_3a_2a_1$, and so forth.

The Chekanov--Eliashberg DGA $(\A(\Lambda_m),\partial)$ for $\Lambda_m$ is then generated by $b_1,b_2$ in degree $1$ and $a_1,\ldots,a_m,t_1^{\pm 1},t_2^{\pm 1}$ in degree $0$, with differential given by:
\begin{align*}
\partial(b_1) &= t_1^{-1} + P_m(a_1,\ldots,a_m) \\
\partial(b_2) &= t_2 + Q_m(a_1,\ldots,a_m)
\end{align*}
and the differential of all degree $0$ generators is $0$.

\begin{rem}
If $m$ is even and $\Lambda_m$ is a $2$-component link, the grading on $\A(\Lambda_m)$ is not well-defined, but rather depends on a choice of Maslov potential on $\Lambda_m$: the general grading is $|b_1|=|b_2|=1$ and $|a_i| = (-1)^{i}\ell$ for some fixed $\ell\in\Z$. In this case, we choose a Maslov potential such that $\ell=0$. This corresponds to a choice of Maslov potential on the front projection from Figure~\ref{fig:2mknot} that is binary, cf.\ the sheaf computation in Section~\ref{ssec:sheaf-obj}.
\end{rem}

An object of the representation category $\Rep_n(\Lambda_m,\coeffs)$ is determined by where it sends the degree $0$ generators of $(\A(\Lambda_m),\partial)$: that is, it is given by a choice of matrices $A_1,\ldots,A_m,T_1,T_2 \in \Mat_n(\coeffs)$ such that
\begin{align*}
P_m(A_1,\ldots,A_m) &= - T_1^{-1} \\
Q_m(A_1,\ldots,A_m) &= -T_2.
\end{align*}
We now have the following result, generalizing Sylvester's determinant identity (which is the $m=2$ case).

\begin{prop}
Let $A_1,\ldots,A_m \in \Mat_n(\coeffs)$. Then 
\label{prop:sylvester}
$$\det(P_m(A_1,\ldots,A_m)) = (-1)^{mn}\det(Q_m(A_1,\ldots,A_m)).$$
\end{prop}

\begin{proof}
By the recurrence relation for $P_m$, we have the following matrix identities, where the matrices are presented in block form with each entry being an $n\times n$ matrix:
\begin{align}
\begin{pmatrix} P_{m-1}(A_1,\ldots,A_{m-1}) & P_m(A_1,\ldots,A_m) \end{pmatrix}
&=
\begin{pmatrix} 0 & 1 \end{pmatrix}
\begin{pmatrix} 0 & 1 \\ 1 & A_1 \end{pmatrix} \cdots
\begin{pmatrix} 0 & 1 \\ 1 & A_m \end{pmatrix} \label{eq:8} \\
\begin{pmatrix} Q_{m}(A_1,\ldots,A_{m}) & Q_{m-1}(A_2,\ldots,A_m) \end{pmatrix}
&=
\begin{pmatrix} 1 & 0 \end{pmatrix}
\begin{pmatrix} -A_m & 1 \\ 1 & 0 \end{pmatrix}  \cdots
\begin{pmatrix} -A_1 & 1 \\ 1 & 0 \end{pmatrix}. \label{eq:8a}
\end{align}
Multiplying \eqref{eq:8a} on the right by $\left( \begin{smallmatrix} 0 & 1 \\ 1 & A_1 \end{smallmatrix} \right) \cdots 
\left( \begin{smallmatrix} 0 & 1 \\ 1 & A_m \end{smallmatrix} \right)$
and combining with \eqref{eq:8} yields:
\begin{multline*}
\begin{pmatrix} 1 & 0 \\
P_{m-1}(A_1,\ldots,A_{m-1}) & P_m(A_1,\ldots,A_m) 
\end{pmatrix} \\
=
\begin{pmatrix}
Q_{m}(A_1,\ldots,A_{m}) & Q_{m-1}(A_2,\ldots,A_m) \\
0 & 1 \end{pmatrix}
\begin{pmatrix} 0 & 1 \\ 1 & A_1 \end{pmatrix} \cdots
\begin{pmatrix} 0 & 1 \\ 1 & A_m \end{pmatrix}.
\end{multline*}
Now take the determinant of both sides.
\end{proof}

By Proposition~\ref{prop:sylvester}, the objects of $\Rep_n(\Lambda_m,\coeffs)$ are in correspondence with
\begin{equation}
\{(A_1,\ldots,A_m)\in \Mat_n(\coeffs)\,|\,P_m(A_1,\ldots,A_m) \text{ is invertible}\};\label{eq:9}
\end{equation}
then $T_1$ and $T_2$ are determined by
$T_1 = -P_m(A_1,\ldots,A_m)^{-1}$ and $T_2 = -Q_m(A_1,\ldots,A_m)$.

We now turn to the morphisms and $A_\infty$ operations in $\Rep_n(\Lambda_m,\coeffs)$. Let $\rho = (A_1,\ldots,A_m)$ and $\rho'=(A_1',\ldots,A_m')$ be two objects in $\Rep_n(\Lambda_m,\coeffs)$. The graded hom spaces between $\rho$ and $\rho'$ are the free $\Mat_n(\coeffs)$-modules with generators given as follows:

\begin{align*}
\Hom^2(\rho,\rho') &= \langle b_1^\vee, b_2^\vee \rangle \\
\Hom^1(\rho,\rho') &= \langle a_1^\vee,\ldots,a_m^\vee,x_1^\vee,x_2^\vee \rangle \\
\Hom^0(\rho,\rho') &= \langle y_1^\vee,y_2^\vee \rangle.
\end{align*}

To calculate the $\mu_1$ maps, we appeal to \cite[Proposition~4.14]{NRSSZ}, which gives a combinatorial formula for the DGA of the $k$-copy of $\Lambda$ in terms of the DGA of $\Lambda$, and which for convenience we recall here. 

\begin{prop}[{\cite[Proposition 4.14]{NRSSZ}}]
Let the DGA $(\A,\partial)$ of $\Lambda$ be generated by Reeb chords $c_1,\ldots,c_p$ along with $t_1^{\pm 1},\ldots,t_q^{\pm 1}$ corresponding to base points $\ast_1,\ldots,\ast_q$. 
\label{prop:copydga}
The DGA $(\A^k,\partial^k)$ for the $k$-copy of $\Lambda$ is given as follows. As an algebra, $\A^k$ is generated by:
\begin{itemize}
\item
$\{c^{ij}_\ell\}_{1\leq i,j\leq k,~1\leq\ell\leq p}$, with $|c^{ij}_\ell|=|c_\ell|$;
\item
$\{(t_\ell^i)^{\pm 1}\}_{1\leq i\leq k,~1\leq\ell\leq q}$, with $|(t_\ell^i)^{\pm 1}| = 0$;
\item
$\{x^{ij}_\ell,y^{ij}_\ell\}_{1\leq i<j\leq k,~1\leq\ell\leq q}$, with $|x^{ij}_\ell|=0$ and $|y^{ij}_\ell|=-1$.
\end{itemize}
To state the differential, define $k\times k$ matrices: $C_1,\ldots,C_p$ with $(C_\ell)_{ij} = c^{ij}_\ell$,
and $X_1,\ldots,X_q,Y_1,\ldots,Y_q,\Delta_1,\ldots,\Delta_q$ with
\begin{align*}
(X_\ell)_{ij} &= \begin{cases} 1 & i=j \\
x_\ell^{ij} & i<j \\
0 & i>j \end{cases}
&
(Y_\ell)_{ij} &= \begin{cases} y_\ell^{ij} & i<j \\
0 & i \geq j \end{cases} &
(\Delta_\ell)_{ij} &= \begin{cases} t_\ell^i & i=j \\
0 & i\neq j.
\end{cases}
\end{align*}
Then $\partial^k$ is given entry-by-entry by:
\begin{align*}
\partial^k(C_\ell) &= \Phi(\partial(c_\ell))+Y_{r(c_\ell)}C_\ell-(-1)^{|c_\ell|}C_\ell Y_{c(c_\ell)} \\
\partial^k(X_\ell) &= \Delta_\ell^{-1} Y_{r(t_\ell)} \Delta_\ell X_\ell - X_\ell Y_{c(t_\ell)} \\
\partial^k(Y_\ell) &= Y_\ell^2,
\end{align*}
where $\Phi :\thinspace \A \to \Mat_k(\A^k)$ is the ring homomorphism given by
$\Phi(c_\ell) = C_\ell$, $\Phi(t_\ell) = \Delta_\ell X_\ell$, $\Phi(t_\ell^{-1}) = X_\ell^{-1} \Delta_\ell^{-1}$, and the link grading $r(\cdot), c(\cdot)$ was defined in Definition~\ref{def:link-grading}.
\end{prop}

In our case, we have link grading $c(t_1)=1$, $c(t_2)=2$, $(r(t_1),r(t_2))=\begin{cases} (2,1) & m \text{ odd} \\ (1,2) & m \text{ even} \end{cases}$, $(r(a_j),c(a_j)) = \begin{cases} (1,2) & j \text{ odd} \\ (2,1) & j \text{ even} \end{cases}$, and
$(r(b_j),c(b_j)) = \begin{cases} (1,2) & m \text{ odd} \\ (j,j) & m \text{ even} \end{cases}$.
This leads to the following formulas for the differential $\partial^2$ in the $2$-copy of $\Lambda$:
\begingroup
\allowdisplaybreaks
\begin{align*}
\partial^2(b_1^{12}) &= -x_1^{12}(t_1^2)^{-1}+y_{r(b_1)}^{12}b_1^{22}+b_1^{11}y_{c(b_1)}^{12}
+\sum_{j=1}^{m} P_{j-1}(a_1^{11},\ldots,a_{j-1}^{11})a_j^{12}P_{m-j}(a_{j+1}^{22},\ldots,a_{m}^{22}) \\
\partial^2(b_2^{12}) &= t_2^1x_2^{12}+y_{r(b_2)}^{12}b_2^{22}+b_2^{11}y_{c(b_2)}^{12}
-\sum_{j=1}^{m} Q_{m-j}(a_{j+1}^{11},\ldots,a_{m}^{11})a_j^{12}Q_{j-1}(a_1^{22},\ldots,a_{j-1}^{22}) \\
\partial^2(a_j^{12}) &= y_{r(a_j)}^{12}a_j^{22}-a_j^{11}y_{c(a_j)}^{12} \\
\partial^2(x_1^{12}) &= (t_1^1)^{-1}y_{r(t_1)}^{12}t_1^2 - y_1^{12} \\
\partial^2(x_2^{12}) &= (t_2^1)^{-1}y_{r(t_2)}^{12}t_2^2 - y_2^{12} \\
\partial^2(y_1^{12}) &= \partial^2(y_2^{12}) = 0.
\end{align*}
\endgroup
Augmenting by substituting $a_j^{11}=A_j$, $a_j^{22}=A_j'$,  $b_j^{11} = b_j^{22} = 0$, $t_j^1 = T_j$, $t_j^2 = T_j'$ and dualizing gives the following formulas for $\mu_1$:
\begingroup
\allowdisplaybreaks
\begin{align*}
\mu_1(z_1b_1^\vee) &= \mu_1(z_2b_2^\vee) = 0 \\
\mu_1(w_ja_j^\vee) &= P_{j-1}(A_1,\ldots,A_{j-1})w_j P_{n-j}(A_{j+1}',\ldots,A_n') b_1^\vee \\*
& \qquad 
- Q_{n-j}(A_{j+1},\ldots,A_n) w_j Q_{j-1}(A_1',\ldots,A_{j-1}') b_2^\vee \\
\mu_1(v_1x_1^\vee) &= -v_1(T_1')^{-1}b_1^\vee \\
\mu_1(v_2x_2^\vee) &= T_2v_2b_2^\vee \\
\mu_1(u_1y_1^\vee) &= \sum_{j \text{ odd}} u_1 A_j' a_j^\vee - \sum_{j \text{ even}} A_j u_1 a_j^\vee - u_1 x_1^\vee 
+ \begin{cases} T_2^{-1} u_1 T_2' x_2^\vee & m \text{ odd} \\ T_1^{-1} u_1 T_1' x_1^\vee & m \text{ even} \end{cases} \\
\mu_1(u_2y_2^\vee) &= -\sum_{j \text{ odd}} A_j u_2 a_j^\vee + \sum_{j \text{ even}} u_2 A_j' a_j^\vee - u_2 x_2^\vee
+ \begin{cases} T_1^{-1} u_2 T_1' x_1^\vee & m \text{ odd} \\ T_2^{-1} u_2 T_2' x_2^\vee & m \text{ even}. \end{cases}
\end{align*}
\endgroup

From this we can calculate the cohomology of $\Hom(\rho,\rho')$ with respect to $\mu_1$.

\begin{prop}\label{propReptorus}
Let $\rho = (A_1,\ldots,A_m)$ and $\rho'=(A_1',\ldots,A_m')$ be objects in $\Rep_n(\Lambda_m,\coeffs)$. Then we have:
\[
H^0\Hom(\rho,\rho') \cong \{(u_1,u_2) \in (\Mat_n(\coeffs))^2 \,|\,
u_1A_j'=A_ju_2 \text{ for $j$ odd and } A_ju_1 = u_2A_j' \text{ for $j$ even}\};
\]
\begin{align*}
&H^1\Hom(\rho,\rho') \cong (\Mat_n(\coeffs))^m \,/ \\
& \quad \{(u_1A_1'-A_1u_2,u_2A_2'-A_2u_1,u_1A_3'-A_3u_2,\ldots,u_1A_m'-A_mu_2)\,|\,u_1,u_2\in\Mat_n(\coeffs)\}
\end{align*}
if $m$ is odd, and
\begin{align*}
&H^1\Hom(\rho,\rho') \cong (\Mat_n(\coeffs))^m \,/ \\
& \quad \{(u_1A_1'-A_1u_2,u_2A_2'-A_2u_1,u_1A_3'-A_3u_2,\ldots,u_2A_m'-A_mu_1)\,|\,u_1,u_2\in\Mat_n(\coeffs)\}
\end{align*}
if $m$ is even; and $H^2\Hom(\rho,\rho') = 0$.
\label{prp:cohrep}
\end{prop}

\begin{proof}
As $\mu_1(-vT_1'x_1^\vee)=vb_1^\vee $ and $\mu_1(T_2^{-1}vx_2^\vee)=vb_2^\vee$, the statement for $H^2$ is clear. 
For $H^1\Hom(\rho,\rho')$, note that any degree $1$ cycle can be written as $\sum_j w_j a_j^\vee + v_1x_1^\vee+v_2x_2^\vee$, where $w_1,\ldots,w_m$ are unconstrained and $v_1,v_2$ are determined by $w_1,\ldots,w_m$; to get the desired isomorphism, project away the $x_1^\vee$ and $x_2^\vee$ terms.
Finally, for $H^0\Hom(\rho,\rho')$, note that by the formula for $\mu_1$, the statement that $\mu_1(u_1y_1^\vee+u_2y_2^\vee)=0$ is equivalent to the conditions $u_1A_j'=A_ju_2$ etc.\ given in the statement of the proposition, along with the conditions
\begin{align*}
u_1 &= T_1^{-1}u_2T_1' & u_2 &= T_2^{-1}u_1T_2' && \text{$m$ odd} \\
u_1 &= T_1^{-1}u_1T_1' & u_2 &= T_2^{-1}u_2T_2' && \text{$m$ even},
\end{align*}
but these extra conditions are superfluous due to the equalities $T_1^{-1} = -P_m(A_1,\ldots,A_m)$, $(T_1')^{-1} = -P_m(A_1',\ldots,A_m')$, $T_2=-Q_m(A_1,\ldots,A_m)$, $T_2'=-Q_m(A_1',\ldots,A_m')$ and Lemma~\ref{lem:u-commutes-with-pq} below.
\end{proof}

\begin{lemma} \label{lem:u-commutes-with-pq}
Suppose for some $m \geq 1$ that $A_1,\dots,A_m,A'_1,\dots,A'_m,u_1,u_2$ are elements of $\Mat_n(\coeffs)$ satisfying
\[ u_1A'_j = A_ju_2\text{ for all odd }j \qquad\text{and}\qquad A_ju_1 = u_2A'_j\text{ for all even }j. \]
Then we have
\begin{align*}
u_1P_m(A'_1,\dots,A'_m) &= \begin{cases} P_m(A_1,\dots,A_m)u_2, & m\text{ odd} \\ P_m(A_1,\dots,A_m)u_1, & m\text{ even}, \end{cases} \\
Q_m(A_1,\dots,A_m)u_2 &= \begin{cases} u_1Q_m(A'_1,\dots,A'_m), & m\text{ odd} \\ u_2Q_m(A'_1,\dots,A'_m), &m\text{ even}. \end{cases}
\end{align*}
\end{lemma}

\begin{proof}
We prove this by induction on $m$, the cases $m=1,2$ being straightforward once we recall that $P_1(a_1)=a_1$, $P_2(a_1,a_2) = 1+a_1a_2$, $Q_1(a_1) = -a_1$, and $Q_2(a_1,a_2) = 1+a_2a_1$.  We let $(i,j)$ be $(1,2)$ if $m$ is odd and $(2,1)$ if $m$ is even; in either case we have $A_mu_j = u_iA'_m$.

For the first equation, multiplying the defining recurrence for $P_m$ on the left by $u_1$ gives
\begin{align*}
u_1P_m(A'_1,\dots,A'_m) &= (u_1P_{m-1}(A'_1,\dots,A'_{m-1}))A'_m + u_1P_{m-2}(A'_1,\dots,A'_{m-2}) \\
&= (P_{m-1}(A_1,\dots,A_{m-1})u_i)A'_m + P_{m-2}(A_1,\dots,A_m)u_j \\
&= P_{m-1}(A_1,\dots,A_{m-1})(A_mu_j) + P_{m-2}(A_1,\dots,A_m)u_j \\
&= P_m(A_1,\dots,A_m)u_j
\end{align*}
as desired.  For the second equation, we use the recurrence
\[ Q_m(A_1,\dots,A_m) = -A_mQ_{m-1}(A_1,\dots,A_{m-1}) + Q_{m-2}(A_1,\dots,A_{m-2}), \]
which is easily seen to be equivalent to the one which defines $Q_m$; multiplying on the right by $u_2$ then gives $Q_m(A_1,\dots,A_m)u_2 = u_iQ_m(A'_1,\dots,A'_m)$ by an identical argument.
\end{proof}

\begin{rem}
A summary of Proposition~\ref{propReptorus} and its proof is as follows. Since $\mu_1(d_1x_1^\vee) = -d_1(T_1')^{-1}b_1^\vee$ and $\mu_1(d_2x_2^\vee) = T_2d_2b_2^\vee$ and $T_1',T_2$ are invertible, to compute cohomology with respect to $\mu_1$ we can cancel all multiples of $x_1^\vee$ with all multiples of $b_1^\vee$, and similarly $x_2^\vee$ with $b_2^\vee$. The resulting quotient complex is supported only in degrees $0$ and $1$ and is given by the map
\begin{align*}
(\Mat_n(\coeffs))^2 &\to (\Mat_n(\coeffs))^m \\
(u_1,u_2) & \mapsto (u_1A_1'-A_1u_2,u_2A_2'-A_2u_1,u_1A_3'-A_3u_2,\ldots).
\end{align*}
Then $H^0\Hom(\rho,\rho')$ and $H^1\Hom(\rho,\rho')$ are the kernel and cokernel of this map, respectively.
\end{rem}

To conclude the computation of the cohomology category $H^*\Rep_n(\Lambda,\coeffs)$, we need to compute the composition maps $\mu_2$. For degree reasons, the only possible nonzero $\mu_2$ maps are
\begin{align*}
\mu_2 &:\thinspace H^0\Hom(\rho',\rho'') \otimes H^0\Hom(\rho,\rho') \to H^0\Hom(\rho,\rho'') \\
\mu_2 &:\thinspace H^0\Hom(\rho',\rho'') \otimes H^1\Hom(\rho,\rho') \to H^1\Hom(\rho,\rho'') \\
\mu_2 &:\thinspace H^1\Hom(\rho',\rho'') \otimes H^0\Hom(\rho,\rho') \to H^1\Hom(\rho,\rho'').
\end{align*}
Let $\rho=(A_1,\ldots,A_m)$, $\rho'=(A_1',\ldots,A_m')$, $\rho''=(A_1'',\ldots,A_m'')$ be objects in $\Rep_n(\Lambda,\coeffs)$. From \cite[Proposition~4.14]{NRSSZ}, we read off the parts of the $3$-copy differential $\partial^3$ that are relevant for calculating $\mu_2$:
\begin{align*}
\partial^3(a_j^{13}) &= y_{r(a_j)}^{12}a_j^{23}-a_j^{12}y_{c(a_j)}^{23}+ \cdots \\
\partial^3(x_j^{13}) &= (t_j^1)^{-1}y_{r(t_j)}^{12}t_j^2x_j^{23}-x_j^{12}y_{c(t_j)}^{23} + \cdots \\
\partial^3(y_j^{13}) &= y_j^{12}y_j^{23}.
\end{align*}

Dualizing these differentials gives the following. For $H^0 \otimes H^0$, we have:
\[
\mu_2(u_1'y_1^\vee+u_2'y_2^\vee,u_1y_1^\vee+u_2y_2^\vee) = -u_1u_1'y_1^\vee-u_2u_2'y_2^\vee.
\]
In terms of the isomorphism for $H^0$ from Proposition~\ref{prp:cohrep}, we have:
\begin{equation}
\mu_2((u_1',u_2'),(u_1,u_2)) = -(u_1u_1',u_2u_2').\label{eq:7}
\end{equation}
For $H^0 \otimes H^1$, we have:
\[
\mu_2(u_1'y_1^\vee+u_2'y_2^\vee,\sum_j w_ja_j^\vee+v_1x_1^\vee+v_2x_2^\vee) = 
-\sum_{j \text{ odd}} w_ju_2'a_j^\vee-\sum_{j \text{ even}} w_ju_1'a_j^\vee-v_1u_1'x_1^\vee-v_2u_2'x_2^\vee.
\]
In terms of the isomorphisms for $H^0$ and $H^1$ from Proposition~\ref{prp:cohrep}, we have:
\begin{equation}
\mu_2((u_1',u_2'),(w_1,\ldots,w_m)) = -(w_1u_2',w_2u_1',w_3u_2',\ldots).\label{eq:5}
\end{equation}

Similarly, for $H^1 \otimes H^0$, we find
\begin{equation}
\mu_2((w_1',\ldots,w_m'),(u_1,u_2)) = -(u_1w_1',u_2w_2',u_1w_3',\ldots).\label{eq:3}
\end{equation}

\section{The category \texorpdfstring{$\Sh(\Lambda,\coeffs)$}{Sh(Lambda,k)}}
\label{sec:category-shlambda}

In this Section we recall some of the theory of microsupport of constructible sheaves from \cite{KashiSchapi}, its application to symplectic and contact geometry as in \cite{GKS}, and the construction of the category $\Sh(\Lambda,\coeffs)$ from \cite{STZ}. We do not intend to be exhaustive, recalling only those facts which are necessary in order to have a working definition of microsupport and compute the category for our particular Legendrian $(2,m)$ torus knots.

\subsection{Dg version of the derived category}
\label{sec:dg-version-derived}

We describe here the construction of the dg-quotient of categories of complexes of sheaves as described in \cite{Drinfeld}. Letting $X$ be a topological space, we denote by $\widetilde{\Sh}(X,\coeffs)$ the dg-category of complexes of sheaves of $\coeffs$-modules on $X$ with constructible homology. Let $\widetilde{A}$ be any full subcategory of $\widetilde{\Sh}(X,\coeffs)$ stable under degree shift and whose objects contain all acyclic complexes. We denote by $A$ the quotient category by acyclic complexes as described in \cite[Section 3]{Drinfeld}. Briefly, the objects of $A$ are those of $\widetilde{A}$, and if $\cF$ is acyclic then $\hom_A(\cF,\cF)$ is equipped with one additional morphism $\varepsilon_F$ of degree $-1$ with differential $d\varepsilon_F=\mathrm{id}_F$.

\begin{rem}\label{rem:qi}
This extra morphism has the following effect. Let $f\in\hom(\cF,\cF')$ be a closed morphism which induces an isomorphism in the cohomological category (i.e.\ a quasi-isomorphism). Then the complex of sheaves $\cF_f:=\operatorname{cone}(f)$ is acyclic. We denote by $i_\cF,i_{\cF'},p_\cF,p_{\cF'}$ the canonical inclusions into and projections out of $\cF_f$. Note that $i_{\cF'}$ and $p_\cF$ are closed morphisms, whereas $di_\cF=i_{\cF'}\circ f$ and $dp_{\cF'}=f\circ p_\cF$. Thus in $A$ the map $g:=p_\cF\circ\varepsilon\circ i_{\cF'}$ is a closed map such that $fg+I_{\cF'}=d(p_{\cF'}\circ\varepsilon\circ i_{\cF'})$ and $gf+I_{\cF}=d(p_\cF\circ\varepsilon\circ i_{\cF})$, which means that $f$ is invertible up to homotopy. In particular, $[f]$ is invertible in the homology category $H_*(A)$.
\end{rem}
Indeed the category $H^0(\Sh(X,\coeffs))$ coincides with the derived category of constructible sheaves of $\coeffs$-modules on $X$, and we have $H^i(\hom(\cF,\cF'))=\Ext^i(\cF,\cF')$.

\subsection{Combinatorial description of constructible sheaves}
\label{sec:comb-descr-constr}
 Let $\mathcal{S}$ be a stratification of the plane $\mathbb{R}^2$, viewed as a category where there is a unique map from $a$ to $b$ iff $a\subset \overline{b}$. For a given stratum $a\in\mathcal{S}$, we let the star of $a$, denoted $s(a)$, be the union of all strata whose closure contains $a$. Here we will only consider stratifications which are \textit{regular cell complexes}, i.e.\ all strata and their stars are contractible in $S^2=\mathbb{R}^2\cup{\infty}$.

We denote by $\widetilde{\Fun}(\mathcal{S},\coeffs)$ the category of functors from $\mathcal{S}$ into chain complexes of $\coeffs$-modules (with acyclic stalk on the non compact strata). As in Section~\ref{sec:dg-version-derived}, given any full subcategory $\widetilde{A}$ whose objects contain all functors valued in acyclic complexes, we denote by $A$ its quotient by those functors.

An object $\cF\in \Sh(\R^2,\coeffs)$ whose homology sheaves are constructible with respect to the stratification $\mathcal{S}$ leads to an object $\Gamma_{\mathcal{S}}(\cF)$ defined by $\Gamma_{\mathcal{S}}(\cF)(a)=\cF(s(a))$.  This extends to a functor from $\Sh_{\mathcal{S}}(X,\coeffs)$ to $\Fun(\mathcal{S},\coeffs)$. It follows from \cite[Lemma 2.3.3]{Nadler_brane} that for regular cell complexes, this map is a quasi-equivalence.

\subsection{Microsupport}
\label{sec:microsupport}

In order to avoid notational confusion involving the sheaf cohomology of complexes of sheaves, given a sheaf $\cF$ of (possibly graded) $\coeffs$-vector spaces, we denote its $i$-th sheaf cohomology group by by $R^i\Gamma(\cF)$, the total group being $R\Gamma(\cF)=\bigoplus_iR^i\Gamma(\cF)$.

Let $\cF\in Ob(\Sh(X,\coeffs))$.  For $(x,p)\in T^*X$, we denote by $C_{x,p}$ the set of functions $f:X\rightarrow\R$ such that $f(x)=0$ and $df_x=p$. A point $(x,p)$ is \emph{characteristic} for $\cF$ if 
there is some $f \in C_{x,p}$ and a sequence of positive numbers $\varepsilon \to 0$ for which the restriction-induced map
\[ R\Gamma_{B(x,\varepsilon)}(\cF) \to R\Gamma_{\{f<0\} \cap B(x,\varepsilon)}(\cF) \]
is not an isomorphism. We denote by $SS(\mathcal{F})$ the closure of the set of characteristic points for $\mathcal{F}$. We denote by  $SS_\bullet(\mathcal{F})$ the set of points $(x,p)\in SS(\mathcal{F})$ such that $p\not= 0$. It follows from results in \cite{KashiSchapi} that if $SS_\bullet(\mathcal{F})$ is a submanifold of $T^*X$, then it is a conical Lagrangian for the standard symplectic structure. Its projectivization is a Legendrian submanifold of $S(T^*X)$.

Given a Legendrian submanifold $\Lambda$ of $S(T^*X)$, we denote by $\Sh(\Lambda,\coeffs)$ the full subcategory of $\Sh(X,\coeffs)$ generated by objects $\cF$ whose homology sheaves are constructible with respect to its front projection $\Pi(\Lambda) \subset X$ and such that $SS_\bullet(\mathcal{F})$ has projectivization inside $\Lambda$; if $X$ is a product $Y \times \R$, with $z$ the coordinate on $\R$, then we will also insist that $\cF$ has acyclic stalk for $z \ll 0$. It follows from \cite[Theorem 3.7]{GKS} and \cite[Theorem 4.10]{STZ} that $\Sh(\Lambda,\coeffs)$ is invariant under Legendrian isotopies of $\Lambda$.

We wish to study invariants of Legendrians in $\R^3$, which does not have the form $S(T^*X)$.  However, a 1-jet space $J^1(Y)$ naturally embeds in the subspace of $S(T^*(Y\times \R))$ consisting of ``graphical'' (i.e., non-vertical) contact elements, so that we can talk about sheaves on $Y\times \R$ with microsupport in a Legendrian in $J^1(Y)$.  In the case $Y=\R$ we thus get an invariant $\Sh(\Lambda,\coeffs)$ of $\Lambda \subset J^1(\R) = \R^3$, as desired.  This is the perspective used throughout \cite{STZ}, and the embedding $i: \R^3 \hookrightarrow S(T^*\R^2)$ is explicitly described in \cite[\S 2.1.1]{STZ}.  Its image is equal to the set of ``downward'' covectors, and it identifies the front ($xz$) projection of $\Lambda \subset \R^3$ with the projection of $i(\Lambda) \subset S(T^*\R^2)$ to the base $\R^2$.  Thus $\Sh(\Lambda,\coeffs)$ consists of complexes of sheaves on $\R^2$ whose homology is constructible with respect to the front projection $\Pi(\Lambda) \subset \R^2$, and \cite{STZ} gives a very concrete interpretation of the singular support condition, which we now describe.

Assume that $\Lambda \subset \R^3$ is generic so that its front $\Pi(\Lambda)$ has only cusps and double points as singularities. Then $\Pi(\Lambda)$ induces a stratification of $\R^2$ where the top-dimensional strata are the regions of the complement of the front, the $0$-dimensional strata are cusps and double points of the front projection, and the remainder of the front forms the $1$-dimensional strata. We denote by $\mathcal{S}_\Lambda$ this stratification. When this stratification is a regular cell complex, the condition for $\cF$ to have microsupport in $\Lambda$ translates to the following conditions for $\Gamma_{\mathcal{S}_\Lambda}(\cF)$:
\begin{enumerate}
\item \textbf{Arcs:} Near an arc $a$ we have a map $f$ going to the upper $2$-dimensional stratum $U$ and a map $g$ going to the lower $2$-dimensional stratum $D$; we require $g$ to be a quasi-isomorphism.
  \begin{figure}[ht!]
\[
\xymatrix{
\cF(U) \\ \cF(s(a)) \ar[u]^{f} \ar[d]_{\cong} \\ \cF(D)
}
\]
    \caption{The microsupport condition near an arc $a$.}
    \label{fig:muarc}
  \end{figure}
\item \textbf{Cusps:} Near a cusp $c$ we have an upper arc $u$ and a lower arc $d$, and regions $I$ and $O$ which are ``inside'' the cusp and ``outside'' the cusp respectively. We require the maps from $c$ to $u$, from $b$ to $O$, and from $a$ to $I$ to be quasi-isomorphisms as shown in Figure~\ref{fig:mucusp}.
  \begin{figure}[ht!]
\[
\xymatrix{
&&&& \cF(s(u)) \ar@/_/[dllll] 
\ar[d]_{\cong} \\
\cF(O) && \cF(c) \ar[ll]_(0.4){\cong} \ar[urr] 
\ar[rr]
\ar[drr]^{\cong} &&\cF(I) \\
&&&& \cF(s(d)) \ar@/^/[ullll]^{\cong} \ar[u]
}
\] 
    \caption{The microsupport condition near a cusp.}
  \label{fig:mucusp}
  \end{figure}
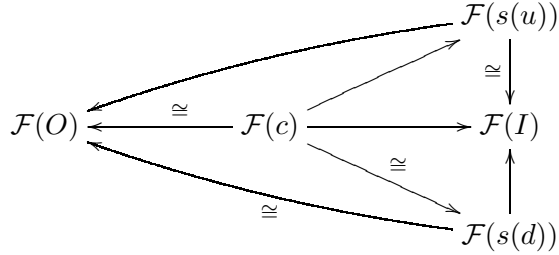
\item \textbf{Crossings:}
Near a crossing $x$ we have four $2$-dimensional strata, labeled $N$, $S$, $E$, and $W$; and four arcs $nw$, $ne$, $sw$, and $se$. We require all downward maps to be quasi-isomorphisms, as in Figure~\ref{fig:mucross}.
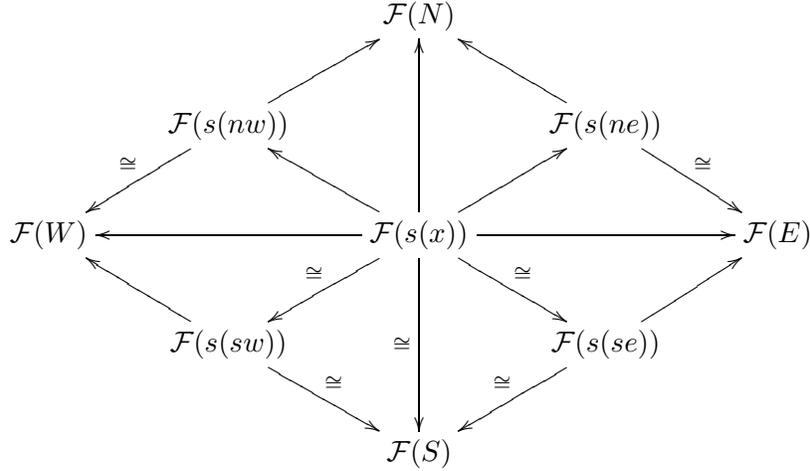
\begin{figure}[ht!]
\[
\xymatrix{
&& \cF(N) && \\
& \cF(s(nw)) \ar[ur] 
\ar[dl]_{\cong} && \cF(s(ne)) \ar[ul] 
\ar[dr]^{\cong} & \\
\cF(W) && \cF(s(x)) \ar[ll] 
\ar[ul] 
\ar[uu] 
\ar[ur] 
\ar[rr] 
\ar[dl]_{\cong} \ar[dd]_{\cong} \ar[dr]^{\cong} && \cF(E) \\
& \cF(s(sw)) \ar[ul] 
\ar[dr]^{\cong} && \cF(s(se)) \ar[dl]_{\cong} \ar[ur] 
& \\
&& \cF(S) &&
}
\]
  \caption{The microsupport condition near a crossing $c$.}
  \label{fig:mucross}
\end{figure}
We also require the total complex $$\mathrm{Tot}\big(\cF(s(x))\rightarrow \cF(s(nw))\oplus \cF(s(ne))\rightarrow \cF(N)\big)$$ to be acyclic.
\end{enumerate}

\subsection{Microlocal rank}
\label{sec:microlocal-rank}

Fix a Maslov potential $\mu$ for $\Lambda$, viewed as an integer $\mu(a)$ assigned to each arc $a$ of the stratification. Following \cite[Section 5.1]{STZ}, an object $\cF$ of $\Sh(\Lambda,\coeffs)$ is said to have \textit{microlocal rank} $n$ if for any arc $a$ of the stratification, if $f$ is the chain map of Figure~\ref{fig:muarc} then
\[ H^i(\operatorname{Cone}(f)) \cong \begin{cases} \coeffs^n & i = -\mu(a) \\ 0 & i \neq -\mu(a). \end{cases} \]
We note for future reference that this condition is satisfied for $\operatorname{Cone}(0 \to \coeffs^n)$ for $\mu(a) = 0$, and for $\operatorname{Cone}(\coeffs^n \to 0)$ for $\mu(a)=1$, where $\coeffs^n$ is in degree $0$ in each case.

We denote by $\Sh_n(\Lambda,\coeffs)$ the full subcategory of $\Sh(\Lambda,\coeffs)$ which consists of objects of microlocal rank $n$.  (In \cite{STZ} this category is labeled $\mathcal{C}_n(\Lambda)$.)  Then the equivalences of categories provided by \cite{GKS,STZ} in proving the invariance of $\Sh(\Lambda,\coeffs)$ carry sheaves of microlocal rank $n$ to sheaves of microlocal rank $n$, so $\Sh_n(\Lambda,\coeffs)$ is a Legendrian isotopy invariant of $\Lambda$.

\section{The cohomology sheaf category for \texorpdfstring{$\Lambda_m$}{Lambda m}}
\label{sec:sheaf-comp}

In this section, we compute the objects and morphisms for the cohomology category $H^*\Sh_n(\Lambda_m,\coeffs)$ of the sheaf category $\Sh_n(\Lambda_m,\coeffs)$ for the Legendrian knot $\Lambda_m$. The version of $\Lambda_m$ that we will use here is the rainbow closure of the $2$-braid $\sigma_1^m \in B_2$, as shown in Figure~\ref{fig:2mknot}, equipped with the binary Maslov potential given by $0$ on the two lower strands and $1$ on the two upper strands. This is not identical to the version of $\Lambda_m$ from Section~\ref{sec:repr-categ-2}, which we used to compute the representation category $\Rep_n(\Lambda_m,\coeffs)$; however, they are Legendrian isotopic, and both $\Rep_n$ and $\Sh_n$ are invariant (up to $A_\infty$ equivalence) under Legendrian isotopy. The advantage of considering the rainbow closure front for the sheaf computation is that, as explained in \cite{STZ}, this front has a binary Maslov potential and all objects are legible.

The objects of $H^*\Sh_n(\Lambda_m,\coeffs)$, which are the same as the objects of $\Sh_n(\Lambda_m,\coeffs)$, are presented in Section~\ref{ssec:sheaf-obj}; as in \cite{STZ}, these can be depicted as a collection of $\coeffs$-vector spaces with certain linear maps between them. The spaces of morphisms in $H^*\Sh_n(\Lambda_m,\coeffs)$ are given by $\Ext$ groups between objects. For Legendrian links such as $\Lambda_m$ with a binary Maslov potential, the only possible nonzero $\Ext$ groups are  $\Ext^0$, $\Ext^1$, and $\Ext^2$: as explained in \cite[Section 7.2.1]{STZ}, in this setting all objects correspond to $\coeffs[x,y]$-modules, and $\coeffs[x,y]$ admits a bimodule resolution of length $3$. We calculate $\Ext^0$ and $\Ext^1$ in Sections~\ref{ssec:ext0} and \ref{ssec:ext1}, respectively, and then we describe their composition in Section~\ref{sec:compositions-sh-}. 
The final $\Ext$ group, $\Ext^2$, vanishes, but the argument for this is more delicate and we leave it to Section~\ref{sec:vanishing-ext_2}.

\subsection{Objects in \texorpdfstring{$\Sh_n(\Lambda_m,\coeffs)$}{Sh n(Lambda,k)}}
\label{ssec:sheaf-obj}

Here we describe the objects in the category $\Sh_n(\Lambda_m,\coeffs)$ where $\Lambda_m$ is equipped with the obvious binary Maslov potential where the upper and lower strands at any cusp have potentials 1 and 0 respectively.

First, we note that any sheaves $\cF$ in $\Ob(\Sh_n(\Lambda_m,\coeffs))$ must have homology concentrated in degree $0$, as shown in \cite[Proposition 5.17]{STZ}. This implies that the map $$\left(\cdots\rightarrow C_{-1}\rightarrow \ker d_0\rightarrow 0\rightarrow0\rightarrow\cdots\right)\rightarrow\left(\cdots\rightarrow C_{-1}\rightarrow C_0\rightarrow C_1\rightarrow \cdots\right)$$ is a quasi-isomorphism, and so is $$\left(\cdots\rightarrow C_{-1}\rightarrow \ker d_0\rightarrow 0\rightarrow 0\rightarrow\cdots\right)\rightarrow\left(\cdots\rightarrow0\rightarrow H_0\rightarrow 0\rightarrow 0\rightarrow\cdots\right).$$

It follows from Remark \ref{rem:qi} that any object $\cF$ is homotopy equivalent to its $0$-th homology sheaf and thus we can restrict to objects which are actual sheaves. 

The sheaf is zero over the unbounded 2-dimensional stratum, and thus from the cusp condition (Figure \ref{fig:mucusp}) we get that the sheaf is also zero over the star of the two ``outside'' cusps and the star of the strand below them. Since the sheaf has microlocal rank $n$, the region $I$ just inside these cusps has $\cF(I)=V$ for some vector space $V$ of dimension $n$. Note that the topmost strand $u$ of the front is the upward strand for both cusps, and by the microsupport condition, we have an isomorphism $\cF(s(U)) \stackrel{\cong}{\to} \cF(I)$. We can use this isomorphism to change $\cF(s(U))$ to an isomorphic vector space, with the result that the map $\cF(s(U)) \to \cF(I)$ is the identity.

We next consider the inside cusps of $\Lambda_m$, which we label $c_l$ (left) and $c_r$ (right), and which have a common upper strand $u$ and a common inner $2$-dimensional stratum $I$ (different from $I$ in the previous paragraph). From the microsupport condition at the cusp (see Figure~\ref{fig:mucusp}), for each of $c=c_l$ and $c=c_r$ we have an isomorphism $\cF(s(c)) \stackrel{\cong}{\to} \cF(s(d))$, and by replacing $\cF(s(d))$ by an isomorphic vector space we can arrange that this isomorphism is the identity map. The microlocal rank condition implies that we can identify the vector space $\cF(I)$ with $V \oplus V$; we replace $\cF(s(u))$ by an isomorphic vector space so that the isomorphism $\cF(s(u)) \to \cF(I) = V \oplus V$ is the identity. Finally, for both $c_l$ and $c_r$ we have an isomorphism $\cF(s(d)) \to \cF(O)$, where $O$ is the same region in both cases; we change $\cF(O)$ so that this isomorphism is the identity for $c_l$, but we do not assume that the same is true for $c_r$ (we will use the extra freedom to define the map $\phi_{m+1}$ below). The sheaf around the cusps $c_l$ and $c_r$ is now given as in Figure~\ref{fig:cusp2m}.

  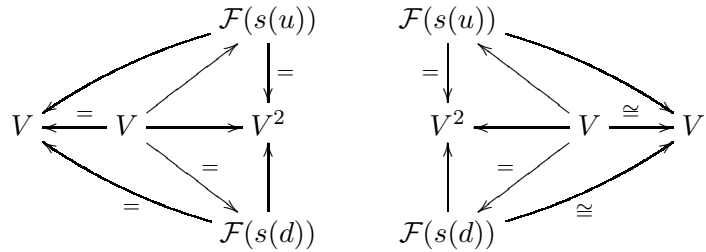
\begin{figure}[ht!]
\[
\xymatrix{
&& \cF(s(u)) \ar@/_/[dll] \ar[d]^{=} &
\cF(s(u)) \ar@/^/[drr] \ar[d]_{=} && \\
V & V \ar[l]_(0.4){=} \ar[ur] \ar[r] \ar[dr]^{=} & V^2 &
V^2 & V \ar[r]^(0.4){\cong} \ar[ul] \ar[l] \ar[dl]_{=} & V \\
&& \cF(s(d)) \ar@/^/[ull]^{=} \ar[u] &
\cF(s(d)) \ar@/_/[urr]_{\cong} \ar[u]
}
\] 
    \caption{The microsupport condition near the inside cusps $c_l$ (left) and $c_r$ (right).
}
  \label{fig:cusp2m}
  \end{figure}

We now turn our attention to crossings. Note that at all crossings, the stalk of $\cF$ along the south 2-dimensional stratum is zero. In addition, there are $m-1$ arcs in the front of $\Lambda_m$ that form the $ne$ arc for one crossing and the $nw$ arc for the crossing to its right; by changing the vector spaces associated to each of these arcs by an isomorphism, we can arrange that for all crossings, the maps $\cF(s(nw)) \to \cF(W)$ and $\cF(s(ne)) \to \cF(E)$ are the identity. The sheaf around each crossing is now given as in Figure~\ref{fig:cross2m}, where the bottom half of the sheaf as in Figure~\ref{fig:mucross} is zero and has been omitted. The acylicity condition at the crossing is then equivalent to the condition that $f_1 \oplus f_2$ is an isomorphism.

\begin{figure}[ht!]
\[
\xymatrix{
&& V^2 && \\
& \cF(s(nw)) \ar[ur]^{f_1} \ar[dl]_{=} && \cF(s(ne)) \ar[ul]_{f_2} \ar[dr]^{=} & \\
V && 0 \ar[ll]_{0} \ar[ul]_(0.4){0} \ar[uu]^{0} \ar[ur]^(0.4){0} \ar[rr]^{0} && V
}
\]
  \caption{The sheaf near a crossing of $\Lambda_m$.}
  \label{fig:cross2m}
\end{figure}
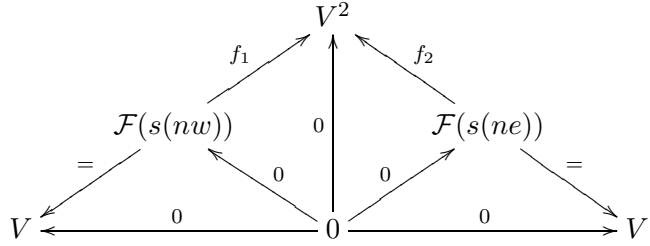

Altogether, the data needed to describe a sheaf with microsupport on $\Lambda_m$ are:
\begin{itemize}
\item For each inner Maslov-$0$ one dimensional stratum $a$ with $U$ the two dimensional stratum above it, a map $\cF(s(a))\rightarrow \cF(U)$.
\item Associated to the ``inside'' Maslov-$1$ one dimensional stratum, a map from $V^2$ to $V$.
\end{itemize}
These data must satisfy the following:
\begin{itemize}
\item If two one-dimensional strata meet at a crossing, the direct sum of the two associated maps is an isomorphism.
\item At the inner left cusp, the composition of the map associated to the lower strand with the one associated to the upper strand is the identity. 
\item At the inner right cusp, the composition of the map associated to the lower strand with the one associated to the upper strand is an isomorphism.
\end{itemize}

All this discussion implies that an object of the sheaf category $\Sh_n(\Lambda_m,\coeffs)$ corresponds to a $n$-dimensional vector space $V$ and a choice of maps $\phi_1,\ldots,\phi_{m+1} :\thinspace V \to V^2$ and $\psi :\thinspace V^2 \to V$ such that the following conditions hold:
\begin{itemize}
\item
$\psi\circ \phi_1=\Id$ and $\psi \circ \phi_{m+1}$ is an isomorphism;
\item
for $1\leq i\leq m$, the images of $\phi_i$ and $\phi_{i+1}$ are complementary in $V^2$: in particular, $\phi_i$ is injective for each $i$.
\end{itemize}
See Figure~\ref{fig:sheaf-object}. In Section~\ref{sec:example-hopf-link} we will demonstrate this construction in detail for the Hopf link $\Lambda_2$.

\begin{figure}
\[
\xymatrix{
&V \\
& V^2 \ar[u]^{\psi} \\
V \ar[ur]^{\phi_1} \ar@{.>}@/^1pc/[uur]^{=} & V \ar[u]_{\phi_2} & \cdots & V \ar[ull]_{\phi_{m+1}} \ar@{.>}@/_1pc/[uull]_{\cong}
}
\] 
\caption{An object of $\Sh_n(\Lambda_m,\coeffs)$.}\label{fig:sheaf-object}
\end{figure}

We now reformulate these conditions.
Since $\psi \circ \phi_1 = \Id$, we can choose a basis of $V^2$ so that $\phi_1 = \left( \begin{matrix} 0 \\ 1 \end{matrix} \right)$ and $\psi = \left( \begin{matrix} 0 & 1 \end{matrix} \right)$.

Now the images of $\phi_1$ and $\phi_2$ are complementary, which means that we can choose a basis for the domain of $\phi_2$ for which $\phi_2 = \left( \begin{matrix} 1 \\ A_1 \end{matrix} \right)$ for some linear $A_1 :\thinspace V \to V$. Similarly, for $3 \leq j \leq m+1$, the image of $\phi_{j-1}$ is complementary to the image of $\phi_{j-2}$ as well as to the image of $\phi_{j}$, and so we can inductively choose bases for $\phi_j$ and linear maps $A_{j-1} :\thinspace V \to V$ such that $\phi_j = \phi_{j-2} + \phi_{j-1} \circ A_{j-1}$.

Thus we obtain a sequence of endomorphisms of $V$, $A_1,\ldots,A_m$, such that if we construct the following chain of maps:
\begin{equation}
\xymatrix{
V \oplus V \ar[rr]^{\left( \begin{smallmatrix} 0 & 1 \\ 1 & A_m \end{smallmatrix} \right)} &&
\cdots \ar[rr]^{\left( \begin{smallmatrix} 0 & 1 \\ 1 & A_2 \end{smallmatrix} \right)} &&
V \oplus V  \ar[rr]^{\left( \begin{smallmatrix} 0 & 1 \\ 1 & A_1 \end{smallmatrix} \right)} &&
V^2
}
\label{eq:obj-chain}
\end{equation}
then for each $k=1,\ldots,m$, the sum $\phi_k \oplus \phi_{k+1}$ is the composition of the final $k$ arrows in this chain.

It follows from \eqref{eq:8} that the matrix for $\phi_k$ is $\left( \begin{matrix} P_{k-1}(A_2,\ldots,A_k) \\ P_k(A_1,\ldots,A_k) \end{matrix} \right)$. The inner right cusp now dictates that $\psi \circ \phi_{m+1}$ is an isomorphism on $V$; since $\psi = \begin{pmatrix} 0 \\ 1 \end{pmatrix}$, this simply states that
$P_m(A_1,\ldots,A_m)$ is invertible. We conclude the following:

\begin{prop}\label{PropObjSH}
The objects in $\Sh_n(\Lambda_m,\coeffs)$ are in correspondence with choices of $A_1,\ldots,A_m \in \End V$ for which $P_m(A_1,\ldots,A_m) \in \End V$ is invertible. Any such $(A_1,\ldots,A_m)$ determines an object as in Figure~\ref{fig:sheaf-object} by the conditions that $\psi = \left( \begin{smallmatrix} 0 \\ 1 \end{smallmatrix} \right)$ and the final $k$ arrows in \eqref{eq:obj-chain} compose to give $\phi_k \oplus \phi_{k+1}$.
\end{prop}

\subsection{Example: the Hopf link}
\label{sec:example-hopf-link}

Here we illustrate the identification of objects in $\Sh_n(\Lambda_m,\coeffs)$ with the configuration shown in Figure~\ref{fig:sheaf-object}, in the case $m=2$ (the Hopf link). This link is shown in Figure~\ref{fig:hopf}, where we have labeled the strata as determined by the front projection: the $U_i$ are 2-dimensional, the $a_i$ and $b_i$ are $1$-dimensional, and the $x_i$ and $y_i$ are $0$ dimensional.

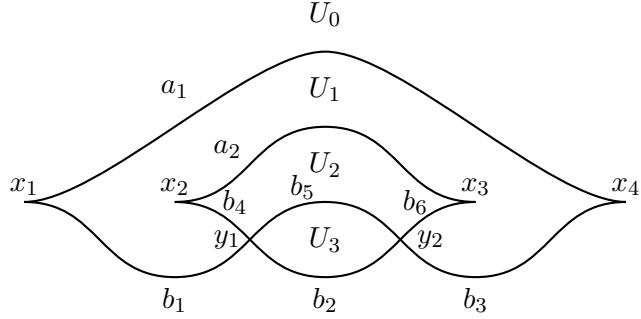
\begin{figure}[ht!]
  \centering
\begin{tikzpicture}
    \draw [thick] (0,1) .. controls +(1,0) and +(-1,0)
    .. (2,0) ..  controls +(1,0) and +(-1,0) .. (4,1); \draw [thick] (4,1) .. controls +(-1,0) and +(1,0) ..  (2,2) .. controls +(-1,0) and +(1,0) ..  (0,1) ;
\draw [thick] (-2,1) .. controls +(1,0) and +(-1,0) .. (0,0) ..  controls
    +(1,0) and +(-1,0) .. (2,1).. controls +(1,0) and +(-1,0) .. (4,0)..  controls
    +(1,0) and +(-1,0) .. (6,1).. controls +(-1,0) and +(1,0) ..  (2,3) .. controls +(-1,0) and +(1,0) ..  (-2,1) ;
\node at (-2,1.2) {$x_1$};
\node at (0,1.2) {$x_2$};
\node at (4,1.2) {$x_3$};
\node at (6,1.2) {$x_4$};
\node at (0.7,0.5) {$y_1$};
\node at (3.4,0.5) {$y_2$};
\node at (2,3.5) {$U_0$};
\node at (2,2.5) {$U_1$};
\node at (2,1.5) {$U_2$};
\node at (2,0.5) {$U_3$};
\node at (0,2.5) {$a_1$};
\node at (0.7,1.7) {$a_2$};
\node at (0,-0.3) {$b_1$};
\node at (2,-0.3) {$b_2$};
\node at (4,-0.3) {$b_3$};
\node at (0.8,1) {$b_4$};
\node at (1.7,1.2) {$b_5$};
\node at (3.2,1) {$b_6$};
  \end{tikzpicture}
    \caption{The front projection of the Hopf link.}
  \label{fig:hopf}
\end{figure}

The category $\mathcal{S}$ corresponding to this stratification as in Section \ref{sec:comb-descr-constr} is the transitive closure of the left hand diagram of Figure~\ref{fig:hopfsheaf}, in which we omit the identity self-morphisms for clarity. 

Applying the functor $\Gamma_{\mathcal{S}}$ to an object $\cF$ in the sheaf category leads to the right hand diagram in Figure~\ref{fig:hopfsheaf}. Here by applying automorphisms of the blue vector spaces, we can turn isomorphisms into identity maps as shown. (There is one more isomorphism between $V$ and $V$ that we could also turn into the identity, but as discussed in Section~\ref{ssec:sheaf-obj}, it is more convenient to leave this as an arbitrary isomorphism for the purposes of identifying objects in the sheaf category with matrices $A_1,A_2$.)
The information in the right hand diagram is then encoded by the red maps $\phi_1,\phi_2,\phi_3,\psi$; that is, precisely the picture in Figure~\ref{fig:sheaf-object} for $m=2$.

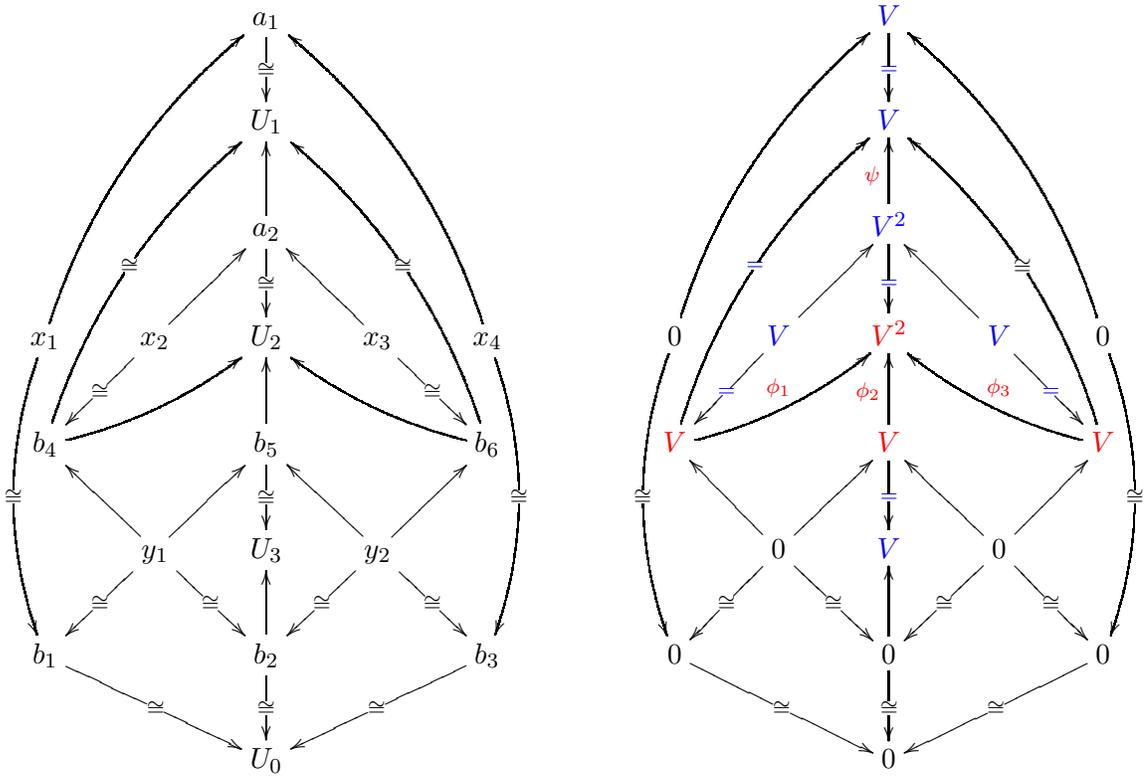
\begin{figure}[ht!]
\[
\xymatrix{
&& a_1 \ar[d]|{\cong} && &&
&& {\color{blue} V} \ar@[blue][d]|{\color{blue} =} && \\
&& U_1 && &&
&& {\color{blue} V} && \\
&& a_2 \ar[u] \ar[d]|{\cong} && &&
&& {\color{blue} V^2} \ar@[red][u]^{\color{red} \psi} \ar@[blue][d]|{\color{blue} =} &&\\
x_1 \ar@/^1pc/[uuurr] \ar@/_1pc/[ddd]|\cong & x_2 \ar[ur] \ar[dl]|{\cong}& U_2 & x_3 \ar[ul] \ar|{\cong}[dr] & x_4 \ar@/_1pc/[uuull] \ar@/^1pc/[ddd]|{\cong} && 
0 \ar@/^1pc/[uuurr] \ar@/_1pc/[ddd]|\cong & {\color{blue} V} \ar[ur] \ar@[blue][dl]|{\color{blue} =}& {\color{red} V^2} & {\color{blue} V} \ar[ul] \ar@[blue]|{\color{blue} =}[dr] & 0 \ar@/_1pc/[uuull] \ar@/^1pc/[ddd]|{\cong} \\
b_4 \ar@/^1pc/[uuurr]|{\cong} \ar@/_/[urr] && b_5 \ar[u] \ar[d]|{\cong} && b_6 \ar@/_1pc/[uuull]|{\cong} \ar@/^/[ull] &&
{\color{red} V} \ar@[blue]@/^1pc/[uuurr]|{\color{blue} =} \ar@[red]@/_/[urr]^{\color{red} \phi_1} && {\color{red} V} \ar@[red][u]^{\color{red} \phi_2} \ar@[blue][d]|{\color{blue} =} && {\color{red} V} \ar@/_1pc/[uuull]|{\cong} \ar@[red]@/^/[ull]_{\color{red} \phi_3}  \\
& y_1 \ar[ul] \ar[ur] \ar[dl]|{\cong} \ar[dr]|{\cong} & U_3 & y_2 \ar[ul] \ar[ur] \ar[dl]|{\cong} \ar[dr]|{\cong} & &&
& 0 \ar[ul] \ar[ur] \ar[dl]|{\cong} \ar[dr]|{\cong} & {\color{blue} V} & 0 \ar[ul] \ar[ur] \ar[dl]|{\cong} \ar[dr]|{\cong} &\\
b_1 \ar[drr]|{\cong} && b_2 \ar[d]|{\cong} \ar[u] && b_3 \ar[dll]|{\cong} &&
0 \ar[drr]|{\cong} && 0 \ar[d]|{\cong} \ar[u] && 0 \ar[dll]|{\cong}  \\
&& U_0 && &&
&& 0 &&
}
\]
\caption{Left: the stratification $\mathcal{S}$ of the front of the Hopf link as a category; right: the result of applying the functor $\Gamma_{\mathcal{S}}$ to an object in the sheaf category. There is also a (zero) map from $a_1$ to $U_0$, omitted for readability.}
\label{fig:hopfsheaf}
\end{figure}

\subsection{\texorpdfstring{$\Ext^0$}{Ext0} in the sheaf category}
\label{ssec:ext0}
We next turn to $\Ext^0$ between two objects in $\Sh_n(\Lambda_m,\coeffs)$. Given two objects, a morphism between them in $\Ext^0$ is determined by a morphism between the strata in the stratification determined by the front of $\Lambda_m$. In terms of the diagram from Figure~\ref{fig:sheaf-object}, a morphism is a choice of $u_0,u_1,\ldots,u_{m+1} :\thinspace V \to V$ and $v :\thinspace V^2 \to V^2$ such that the following diagram commutes:
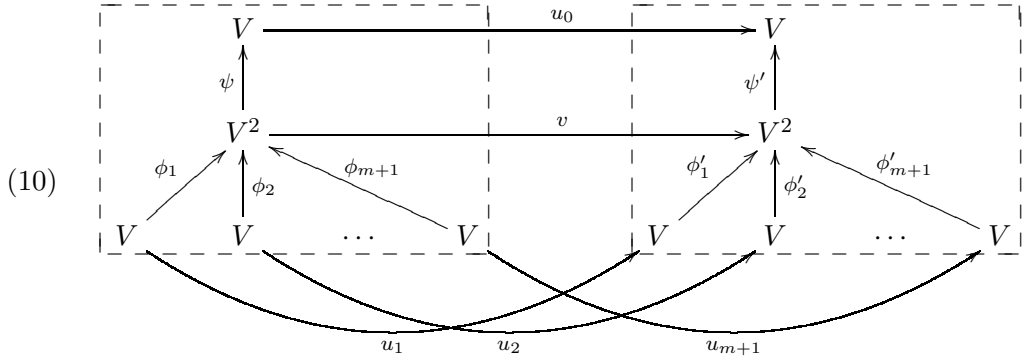
\begin{equation}
\begin{gathered}
\def\sheafbox{\save [].[ddrrr]*[F--:<3pt>]\frm{}\restore\phantom{V^2}}%
\xymatrix{
\sheafbox&V \ar[rrrrr]^(.6){u_0} &&&&\sheafbox& V\\
& V^2 \ar[u]^{\psi} \ar[rrrrr]^(.6){v}&&&&& V^2 \ar[u]^{\psi'} \\
V \ar[ur]^{\phi_1} \ar@/_3pc/[rrrrr]_{u_1} & V \ar[u]_{\phi_2} \ar@/_3pc/[rrrrr]_{u_2} & \cdots & V \ar[ull]_{\phi_{m+1}} \ar@/_3pc/[rrrrr]_{u_{m+1}} &&
V \ar[ur]^{\phi_1'} & V \ar[u]_{\phi_2'} & \cdots & V \ar[ull]_{\phi_{m+1}'} 
}
\end{gathered}
\label{eq:mor-chain}
\end{equation}

Now describe two objects in $\Sh_n(\Lambda_m,\coeffs)$ by $(A_1,\ldots,A_m)$ and $(A_1',\ldots,A_m')$ as in Proposition~\ref{PropObjSH}. Then a morphism from $(A_1,\ldots,A_m)$ to $(A_1',\ldots,A_m')$
is a choice of $u_0,u_1,\ldots,u_{m+1},v$ such that $u_0 \circ \psi = \psi' \circ v$ and the following diagram commutes:
\[
\xymatrix{
V \oplus V \ar[rr]^{\left( \begin{smallmatrix} 0 & 1 \\ 1 & A_m \end{smallmatrix} \right)} 
\ar[dd]_{\left( \begin{smallmatrix} u_m & 0 \\ 0 & u_{m+1} \end{smallmatrix} \right)}
&&
\cdots \ar[r] & V \oplus V \ar[rr]^{\left( \begin{smallmatrix} 0 & 1 \\ 1 & A_2 \end{smallmatrix} \right)} 
\ar[dd]_{\left( \begin{smallmatrix} u_2 & 0 \\ 0 & u_{3} \end{smallmatrix} \right)}&&
V \oplus V  \ar[rr]^{\left( \begin{smallmatrix} 0 & 1 \\ 1 & A_1 \end{smallmatrix} \right)} 
\ar[dd]_{\left( \begin{smallmatrix} u_1 & 0 \\ 0 & u_{2} \end{smallmatrix} \right)}
&&
V^2 \ar[dd]^v \\ \\
V \oplus V \ar[rr]^{\left( \begin{smallmatrix} 0 & 1 \\ 1 & A_m' \end{smallmatrix} \right)} &&
\cdots \ar[r] & V \oplus V \ar[rr]^{\left( \begin{smallmatrix} 0 & 1 \\ 1 & A_2' \end{smallmatrix} \right)} &&
V \oplus V  \ar[rr]^{\left( \begin{smallmatrix} 0 & 1 \\ 1 & A_1' \end{smallmatrix} \right)} &&
V^2.
}
\]

The commutativity of the $k$-th square from the right (for $2 \leq k \leq m$) gives the pair of equations $u_{k+1}=u_{k-1}$ and $A_k'u_{k+1} = u_kA_k$. 
The commutativity of the rightmost square gives $v = \left( \begin{smallmatrix} u_2 & 0 \\ A_1'u_2-u_1A_1 & u_1 \end{smallmatrix} \right)$, whence the relation $u_0 \circ \psi = \psi' \circ v$ combined with $\psi = \psi' = \left( \begin{matrix} 0 & 1 \end{matrix} \right)$ implies $u_0=u_1$ and $A_1'u_2-u_1A_1 = 0$.
Since $\Ext^0$ is given by morphisms between objects in $\Sh_n(\Lambda_m,\coeffs)$, we conclude the following result.

\begin{prop}
Let $\cF = (A_1,\ldots,A_m)$ and $\cF' = (A_1',\ldots,A_m')$ denote objects in $\Sh_n(\Lambda_m,\coeffs)$. Then we have:
\label{prp:ext0}
\begin{align*}
\Ext^0(\cF,\cF') \cong \{(u_1,u_2) \in (\End V)^2 &: A_k'u_1 = u_2A_k \text{ for all even $k$ and } \\
& \quad A_k'u_2 = u_1A_k \text{ for all odd $k$}\}.
\end{align*}
Explicitly, given $u_1,u_2$, we can construct an element of $\Ext^0$ as depicted in \eqref{eq:mor-chain}, as follows: define $u_0=u_1$, $u_k=u_1$ for $k \geq 1$ odd, $u_k = u_2$ for $k \geq 2$ even, and $v = \left( \begin{smallmatrix} u_2 & 0 \\ 0 & u_1 \end{smallmatrix} \right)$.
\end{prop}

\subsection{\texorpdfstring{$\Ext^1$}{Ext1} in the sheaf category}
\label{ssec:ext1}
We now turn our attention to $\Ext^1$. We use the description of $\Ext^1(\cF,\cF')$ as extensions $\cG$ of $\cF$ by $\cF'$ (see \cite[Theorem~III.2.4]{HiltonStamm} or \cite[Exercise~6.1]{Hartshorne}), i.e.\ as short exact sequences \[
0 \rightarrow \cF' \rightarrow \cG \rightarrow \cF \rightarrow 0.
\]
 Let $\cF = (A_1,\ldots,A_m)$ and $\cF' = (A_1',\ldots,A_m')$ be objects in $\Sh_n(\Lambda_m,\coeffs)$. Let $\cG$ denote an extension of $\cF$ by $\cF'$. It follows from the triangle inequality \cite[Proposition~5.1.3]{KashiSchapi} that $SS_\bullet(\cG)\subset SS_\bullet(\cF)\cup SS_\bullet(\cF')$. This implies that $SS_\bullet(\cG)$ has projectivization inside $\Lambda$. Furthermore, the microlocal rank of $\cG$ is $2n$, so we can represent $\cG$ by a diagram
\begin{equation}
\begin{gathered}
\xymatrix{
&V^2 \\
& V^4 \ar[u]^{\Psi} \\
V^2 \ar[ur]^{\Phi_1}  & V^2 \ar[u]_{\Phi_2} & \cdots & V^2. \ar[ull]_{\Phi_{m+1}} 
}
\end{gathered}
\label{eq:ext1}
\end{equation}
The extension
\[
0 \rightarrow \cF' \rightarrow \cG \rightarrow \cF \rightarrow 0
\]
is given componentwise by a collection of short exact sequences $0 \rightarrow V \rightarrow V^2 \rightarrow V \rightarrow 0$ and one short exact sequence $0 \rightarrow V^2 \rightarrow V^4 \rightarrow V^2 \rightarrow 0$. We can choose bases for the components $V^2,V^4$ of $\cG$ such that all maps from $\cF'$ to $\cG$ are given by $\left( \begin{matrix} 1 \\ 0 \end{matrix} \right)$ (for $V \rightarrow V^2$) and $\left( \begin{smallmatrix} 1 & 0 \\ 0 & 0 \\ 0 & 1 \\ 0 & 0 \end{smallmatrix} \right)$ (for $V^2 \rightarrow V^4$), while all maps from $\cG$ to $\cF$ are given by $\left( \begin{matrix} 0 & 1 \end{matrix} \right)$ (for $V^2 \rightarrow V$) and $\left( \begin{smallmatrix} 0 & 1 & 0 & 0 \\ 0 & 0 & 0 & 1 \end{smallmatrix} \right)$ (for $V^4 \rightarrow V^2$).

As before, we represent $\cF$ by a sequence of maps
\[
\xymatrix{
V \oplus V \ar[rr]^{\left( \begin{smallmatrix} 0 & 1 \\ 1 & A_m \end{smallmatrix} \right)} &&
\cdots \ar[rr]^{\left( \begin{smallmatrix} 0 & 1 \\ 1 & A_2 \end{smallmatrix} \right)} &&
V \oplus V  \ar[rr]^{\left( \begin{smallmatrix} 0 & 1 \\ 1 & A_1 \end{smallmatrix} \right)} &&
V^2
}
\]
and similarly for $\cF'$. (In each case the top map from $V^2$ to $V$ is given by
$\psi = \begin{pmatrix}0&1\end{pmatrix}$, and the $A_i$ determine $\phi_1,\ldots,\phi_{m+1}$.)
We can analogously represent $\cG$ by $\Psi :\thinspace V^4 \to V^2$ along with a sequence of maps
\begin{equation}
\xymatrix{
V^2 \oplus V^2 \ar[rr]^{\Omega_m} &&
\cdots \ar[rr]^{\Omega_2} &&
V^2 \oplus V^2  \ar[rr]^{\Omega_1} &&
V^4
}
\label{eq:Omega}
\end{equation}
so that for $k=1,\ldots,m$, $\Phi_k \oplus \Phi_{k+1}$ is the composition of the final $k$ arrows. Note in particular that for each $k \geq 2$, $\Omega_k :\thinspace V^2 \oplus V^2 \to V^2 \oplus V^2$ sends the second $V^2$ factor to the first $V^2$ factor by the identity map; that is, the matrix for $\Omega_k$ is of the form $\left( \begin{smallmatrix} 0 & 0 & \ast & \ast \\ 0 & 0 & \ast & \ast \\ 1 & 0 & \ast & \ast \\ 0 & 1 & \ast & \ast \end{smallmatrix} \right)$.

Now let $\cG$ be an extension of $\cF$ by $\cF'$. The top maps in $\cF',\cG,\cF$ give the commutative diagram
\[
\xymatrix{
V \ar[rr]^{\left( \begin{smallmatrix} 1 \\ 0 \end{smallmatrix} \right)} && V^2 \ar[rr]^{\left( \begin{smallmatrix} 0 & 1 \end{smallmatrix} \right)} && V \\
V^2 \ar[u]^{\left( \begin{smallmatrix} 0 & 1 \end{smallmatrix} \right)} \ar[rr]^{\left( \begin{smallmatrix} 1 & 0 \\ 0 & 0 \\ 0 & 1 \\ 0 & 0 \end{smallmatrix} \right)} &&
V^4 \ar[u]^{\Psi} \ar[rr]^{\left( \begin{smallmatrix} 0 & 1 & 0 & 0 \\ 0 & 0 & 0 & 1 \end{smallmatrix} \right)} && V^2 \ar[u]_{\left( \begin{smallmatrix} 0 & 1 \end{smallmatrix} \right)},
}
\]
whence $\Psi = \left( \begin{matrix} 0 & u_0 & 1 & v_0 \\ 0 & 0 & 0 & 1 \end{matrix} \right)$ for some $u_0,v_0\in\End V$. The commutativity involving the rest of the maps $\Phi_1,\ldots,\Phi_{m+1}$ can be expressed as the following commutative diagram:
\[
\xymatrix{
V \oplus V \ar[rr]^{\left( \begin{smallmatrix} 0 & 1 \\ 1 & A_m' \end{smallmatrix} \right)} \ar[d] &&
\cdots \ar[rr]^{\left( \begin{smallmatrix} 0 & 1 \\ 1 & A_2 \end{smallmatrix} \right)} &&
V \oplus V  \ar[rr]^{\left( \begin{smallmatrix} 0 & 1 \\ 1 & A_1' \end{smallmatrix} \right)}  \ar[d]&&
V^2  \ar[d] \\
V^2 \oplus V^2 \ar[rr]^{\Omega_m} \ar[d] &&
\cdots \ar[rr]^{\Omega_2} &&
V^2 \oplus V^2  \ar[rr]^{\Omega_1} \ar[d] &&
V^4 \ar[d] \\
V \oplus V \ar[rr]^{\left( \begin{smallmatrix} 0 & 1 \\ 1 & A_m \end{smallmatrix} \right)} &&
\cdots \ar[rr]^{\left( \begin{smallmatrix} 0 & 1 \\ 1 & A_2' \end{smallmatrix} \right)} &&
V \oplus V  \ar[rr]^{\left( \begin{smallmatrix} 0 & 1 \\ 1 & A_1 \end{smallmatrix} \right)} &&
V^2,
}
\]
where all the vertical maps between the first and second rows are $\left( \begin{smallmatrix} 1 & 0 \\ 0 & 0 \\ 0 & 1 \\ 0 & 0 \end{smallmatrix} \right)$ and all the vertical maps between the second and third rows are $\left( \begin{smallmatrix} 0 & 1 & 0 & 0 \\ 0 & 0 & 0 & 1 \end{smallmatrix} \right)$.

Commutativity of the $k$-th set of boxes from the right in the above diagram is now easily seen to imply that:
\[
\Omega_k = \left( \begin{matrix} 0 & \ast & 1 & \ast \\
0 & 0 & 0 & 1 \\
1 & \ast & A_k' & \ast \\
0 & 1 & 0 & A_k
\end{matrix}
\right).
\]

Along with the restriction on the form of $\Omega_k$ for $k \geq 2$, we conclude that extensions of $\cF$ by $\cF'$ are in correspondence with maps $\Psi,\Omega_1,\ldots,\Omega_m$ where
\[
\begin{gathered}
\Psi =  \left( \begin{matrix} 0 & v_0 & 1 & w_0 \\ 0 & 0 & 0 & 1 \end{matrix} \right) \\
\Omega_1 = \left( \begin{matrix}  0 & x & 1 & v_1 \\
0 & 0 & 0 & 1 \\
1 & y & A_1' & w_1 \\
0 & 1 & 0 & A_1
\end{matrix} \right)
\qquad
\Omega_k = \left( \begin{matrix} 0 & 0 & 1 & v_k \\
0 & 0 & 0 & 1 \\
1 & 0 & A_k' & w_k \\
0 & 1 & 0 & A_k
\end{matrix}
\right),~2\leq k \leq m
\end{gathered}
\]
for some $x,y,v_0,\ldots,v_m,w_0,\ldots,w_m \in \End V$.

To compute $\Ext^1(\cF,\cF')$, we need to quotient the space of extensions by equivalence. Suppose we have two extensions $\cG$ and $\cG'$ of $\cF$ by $\cF'$. As above, represent these by matrices $\Omega_k$ and $\Omega_k'$ respectively, determined by $x,y,v_0,\ldots,v_m,w_0,\ldots,w_m$ and 
$x',y',v_0',\ldots,v_m',w_0',\ldots,w_m'$ respectively. We wish to determine when $\cG$ and $\cG'$ are isomorphic; this happens precisely when there is some isomorphism $\cG \to \cG'$ such that the following diagram commutes:
\begin{equation}
\begin{gathered}
\xymatrix{
0 \ar[r] &\cF' \ar@{=}[d] \ar[r] & \cG \ar[r] \ar[d] & \cF \ar@{=}[d] \ar[r] & 0 \\
0 \ar[r] &\cF' \ar[r] & \cG' \ar[r] & \cF \ar[r] & 0. \\
}
\end{gathered}
\label{eq:ext1-comm}
\end{equation}
The map from $\cG$ to $\cG'$ can be expressed in components as follows:
\begin{equation}
\def\sheafbox{\save [].[ddrrr]*[F--:<3pt>]\frm{}\restore\phantom{V^2}}%
\begin{gathered}
\xymatrix{
\sheafbox&V^2 \ar[rrrr]^(.625){\Upsilon_0} &&&\sheafbox& V^2\\
& V^4 \ar[u]^{\Psi} \ar[rrrr]^(.625){\Theta}&&&& V^4 \ar[u]^{\Psi'} \\
V^2 \ar[ur]^{\Phi_1} \ar@/_3pc/[rrrr]_{\Upsilon_1} & V^2 \ar[u]_{\Phi_2} \ar@/_3pc/[rrrr]_{\Upsilon_2} & \cdots & V^2 \ar[ull]_{\Phi_{m+1}} \ar@/_3pc/[rrrr]_{\Upsilon_{m+1}} &
V^2 \ar[ur]^{\Phi_1'} & V^2 \ar[u]_{\Phi_2'} & \cdots & V^2 \ar[ull]_{\Phi_{m+1}'} 
}
\end{gathered}
\label{eq:ext1also}
\end{equation}
The commutativity of \eqref{eq:ext1-comm} implies that we have
\[
\Theta = \left( \begin{matrix}
1 & z_1 & 0 & z_2 \\
0 & 1 & 0 & 0 \\
0 & z_3 & 1 & z_4 \\
0 & 0 & 0 & 1 \end{matrix} \right)
\qquad
\Upsilon_k = \left( \begin{matrix} 1 & u_k \\ 0 & 1 \end{matrix} \right), ~ k \geq 0
\]
for some $z_1,z_2,z_3,z_4,u_k \in \End V$. The commutativity of \eqref{eq:ext1also} implies that $\Upsilon_0\Psi = \Psi'\Theta$ and
 the following diagram commutes:
\[
\xymatrix{
V^2 \oplus V^2 \ar[rr]^{\Omega_m} \ar[d]_{\Upsilon_m \oplus \Upsilon_{m+1}} &&
\cdots \ar[rr]^{\Omega_2} &&
V^2 \oplus V^2  \ar[rr]^{\Omega_1} \ar[d]_{\Upsilon_1 \oplus \Upsilon_2} &&
V^4 \ar[d]_\Theta \\
V^2 \oplus V^2 \ar[rr]^{\Omega_m'} &&
\cdots \ar[rr]^{\Omega_2'} &&
V^2 \oplus V^2  \ar[rr]^{\Omega_1'} &&
V^4.
}
\]
Writing out the commutativity of this diagram and using the above formulas for $\Theta$, $\Upsilon_k$, $\Omega_k$, and $\Omega_k'$, we get the following equations:
\begin{equation}
\begin{aligned}
x' &= x+z_2 & \qquad v_1' &= v_1-u_2+z_1+z_2 A_1 \\
y' &= y-u_1+z_4 & w_1' &= w_1+z_3-A_1 u_2+z_4 A_1 \\
v_0' &= v_0-z_3 & v_k' &= v_k + u_{k-1} - u_{k+1} \\
w_0' &= w_0+u_0-z_4 & w_k' &= w_k + u_k A_k - A_k' u_{k+1}
\end{aligned}
\label{eq:ext1-equiv}
\end{equation}
(here $k \geq 2$).

We conclude that $\cG = (x,y,v_k,w_k)$ and $\cG'=(x',y',v_k',w_k')$ are isomorphic if and only if there exist $z_1,z_2,z_3,z_4,u_k$ so that \eqref{eq:ext1-equiv} holds. Now an inspection of \eqref{eq:ext1-equiv} shows that every extension $\cG$ is isomorphic to one where $x=y=w_0=0$ and $v_k=0$ for all $k \geq 0$; just choose $z_j,u_k$ appropriately. 

Thus for the purposes of calculating $\Ext^1(\cF,\cF')$, we may restrict ourselves to extensions where $x=y=w_0=v_k=0$, which are determined by $(w_1,\ldots,w_m)\in(\End V)^m$. 
Now the extensions $(w_1,\ldots,w_m)$ and $(w_1',\ldots,w_m')$ are isomorphic if and only if there are $z_j,u_k$ such that \eqref{eq:ext1-equiv} hold
(where $x=y=w_0=v_k=x'=y'=w_0'=v_k'=0$). In this setting $v_k' = v_k + u_{k-1} - u_{k+1}$ gives $u_{k+1}=u_{k-1}$ for all $2\leq k\leq m$, so that $u_k$ is either $u_1$ or $u_2$ depending on the parity of $k$. The remaining equations in \eqref{eq:ext1-equiv} become
$z_4=u_0=u_1$, $z_1=u_2$, $z_2=z_3=0$, and $w_k' = w_k + u_k A_k - A_k' u_{k+1}$ for $1 \leq k \leq m$. We conclude the following result.

\begin{prop}
Let $\cF = (A_1,\ldots,A_m)$ and $\cF' = (A_1',\ldots,A_m')$ denote objects in $\Sh_n(\Lambda_m,\coeffs)$. Then we have
\begin{align*}
&\Ext^1(\cF,\cF') \cong 
(\End V)^m \, / \\
& \quad \{(u_1A_1-A_1'u_2,u_2A_2-A_2'u_1,u_1A_3-A_3'u_2,\ldots,u_1A_m-A_m'u_2)\,|\,u_1,u_2\in\End V\}
\end{align*}
if $m$ is odd, and
\begin{align*}
&\Ext^1(\cF,\cF') \cong 
(\End V)^m \, / \\
&\quad \{(u_1A_1-A_1'u_2,u_2A_2-A_2'u_1,u_1A_3-A_3'u_2,\ldots,u_2A_m-A_m'u_1)\,|\,u_1,u_2\in\End V\}
\end{align*}
if $m$ is even.

Explicitly, given $(w_1,\ldots,w_m) \in (\End V)^m$, we can construct an element of $\Ext^1$ as depicted in \eqref{eq:ext1}, by the conditions that $\Psi = \left( \begin{smallmatrix} 0 & 0 & 1 & 0\\ 0 & 0 & 0 & 1\end{smallmatrix} \right)$ and $\Phi_1,\ldots,\Phi_{m+1}$ are given as follows. Define $\Omega_1,\ldots,\Omega_m$ by $\Omega_j = \left( \begin{smallmatrix} 0 & 0 & 1 & 0 \\ 0 & 0 & 0 & 1 \\ 1 & 0 & A_j' & w_j \\ 0 & 1 & 0 & A_j \end{smallmatrix} \right)$ for $1\leq j\leq m$; these then determine $\Phi_1,\ldots,\Phi_{m+1}$ by the condition that the final $k$ arrows in \eqref{eq:Omega} compose to give $\Phi_k \oplus \Phi_{k+1}$.
\label{prp:ext1}
\end{prop}

\subsection{Compositions in \texorpdfstring{$\Sh_n(\Lambda_m,\coeffs)$}{Sh n(Lambda m,k)} }
\label{sec:compositions-sh-}
We describe now the compositions in the category $\Sh_n(\Lambda_m,\coeffs)$. In Section \ref{sec:vanishing-ext_2}, we will show that all $\Ext^i$ groups vanish for $i>1$, so we only have to give compositions of the form $u'\circ u$, $u'\circ e$, and $e'\circ u$ for $u \in \Ext^0(\cF,\cF')$, $u' \in \Ext^0(\cF',\cF'')$, $e \in \Ext^1(\cF,\cF')$, and $e' \in \Ext^1(\cF',\cF'')$.

In the first case this is obvious: the composition in $\Ext^0$ of morphisms
\[ \cF \xrightarrow{u} \cF' \xrightarrow{u'} \cF'' \]
given by $u=(u_1,u_2)$ and $u'=(u'_1,u'_2)$ is $u'\circ u = (u'_1u_1,u'_2u_2)$.

The only non-trivial compositions left involve degree $0$ with degree $1$. Let $\mathcal{F},\mathcal{F}'$ and $\mathcal{F}''$ be three objects of $Sh_n(\Lambda_m,\coeffs)$ and let $u:\mathcal{F}\rightarrow \mathcal{F}'$ be a morphism and $e':\mathcal{F}''\xrightarrow{i}\mathcal{G}\xrightarrow{p}\mathcal{F}'$ an extension.  Then, following \cite[\S III.1]{HiltonStamm}, the extension $e'\circ u:\mathcal{F}''\rightarrow u^*\mathcal{G}\rightarrow\mathcal{F}$ is given by the top row of the diagram
$$
\xymatrix{
\cF'' \ar@{=}[d] \ar@{-->}[r] \ar@/^1pc/[rr]^{0}& u^*\mathcal{G}\ar[r] \ar[d] & \mathcal{F}\ar[d]^-u\\
\cF'' \ar[r]^{i} & \mathcal{G}\ar[r]^{p} &\mathcal{F}'
,}
$$
in which the right square is a pullback square whose universal property gives us the dotted arrow.  (Concretely, the map $u^*\cG \to \cG$ restricts to an isomorphism
\[ \ker(u^*\cG \to \cF) \xrightarrow{\sim} \ker(\cG \xrightarrow{p} \cF') = \img(\cF'' \xrightarrow{i} \cG), \]
which we use to lift the map $i:\cF'' \to \cG$ to $u^*\cG$.)

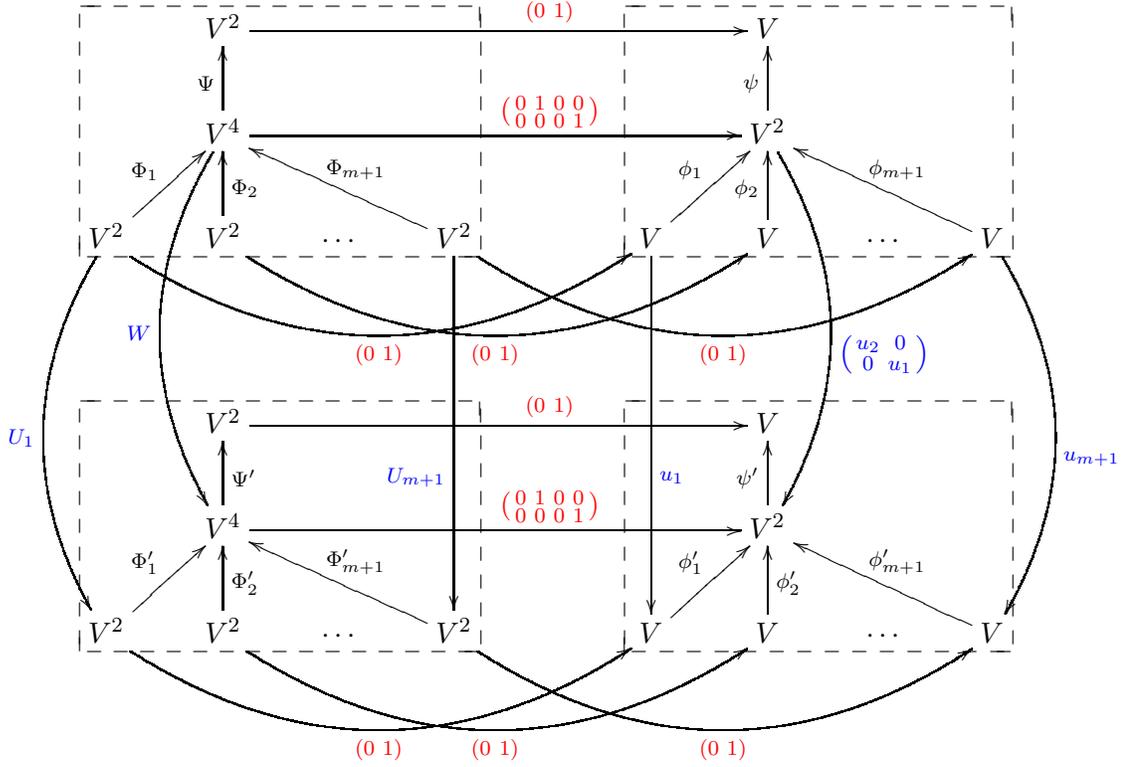
\begin{figure}
\def\sheafbox{\save [].[ddrrr]*[F--:<3pt>]\frm{}\restore\phantom{V^2}}%
\[
\xymatrix{
\sheafbox&V^2 \ar@[red][rrrrr]^(.6){\color{red} \begin{smallmatrix}
    (0 &1)
  \end{smallmatrix}} &&&&\sheafbox& V\\
& V^4 \ar@[blue]@/_2pc/[dddd]_{\color{blue} W}\ar[u]^{\Psi} \ar@[red][rrrrr]^(.6){\color{red} \left(\begin{smallmatrix}
    0 &1&0&0\\
0&0&0&1
  \end{smallmatrix}\right)}&&&&& V^2 \ar@[blue]@/^2pc/[dddd]^(.55){\color{blue} 
  \left(\begin{smallmatrix}
 u_2 & 0 \\
0 & u_1
 \end{smallmatrix}\right)
}\ar[u]^{\psi} \\
V^2 \ar@[blue]@/_2pc/[dddd]_{\color{blue} U_1}\ar[ur]^{\Phi_1} \ar@[red]@/_3pc/[rrrrr]_{\color{red} \begin{smallmatrix}
    (0 &1)
  \end{smallmatrix}} & V^2 \ar[u]_{\Phi_2} \ar@[red]@/_3pc/[rrrrr]_{\color{red} \begin{smallmatrix}
    (0 &1)
  \end{smallmatrix}} & \cdots & V^2 \ar[ull]_{\Phi_{m+1}} \ar@[red]@/_3pc/[rrrrr]_{\color{red} \begin{smallmatrix}
    (0 &1)
  \end{smallmatrix}} 
\ar@[blue][dddd]_(.6){\color{blue} U_{m+1}}
&&
V \ar@[blue][dddd]^(.6){\color{blue} u_1}\ar[ur]^{\phi_1} & V \ar[u]^{\phi_2} & \cdots & V \ar@[blue]@/^2pc/[dddd]^(.55){\color{blue} u_{m+1}}\ar[ull]_{\phi_{m+1}}\\ 
&&&&&&\\
\sheafbox&V^2 \ar@[red][rrrrr]^(.6){\color{red} \begin{smallmatrix}
    (0 &1)
  \end{smallmatrix}} &&&&\sheafbox& V\\
& V^4 \ar[u]_{\Psi'} \ar@[red][rrrrr]^(.6){\color{red} \left(\begin{smallmatrix}
    0 &1&0&0\\
0&0&0&1
  \end{smallmatrix}\right)}&&&&& V^2 \ar[u]^{\psi'} \\
V^2 \ar[ur]^{\Phi_1'} \ar@[red]@/_3pc/[rrrrr]_{\color{red} \begin{smallmatrix}
    (0 &1)
  \end{smallmatrix}} & V^2 \ar[u]_{\Phi_2'} \ar@[red]@/_3pc/[rrrrr]_{\color{red} \begin{smallmatrix}
    (0 &1)
  \end{smallmatrix}} & \cdots & V^2 \ar[ull]_{\Phi_{m+1}'} \ar@[red]@/_3pc/[rrrrr]_{\color{red} \begin{smallmatrix}
    (0 &1)
  \end{smallmatrix}} &&
V \ar[ur]^{\phi_1'} & V \ar[u]_{\phi_2'} & \cdots & V \ar[ull]_{\phi_{m+1}'}}
\]
\caption{The pullback diagram for the morphisms $p: \cG \to \cF'$ and $u:\cF \to \cF'$.}
\label{fig:complicated-diagram}
\end{figure}

Now $u \in \Ext^0(\cF,\cF')$ is determined by $(u_1,u_2) \in (\End V)^2$ as in Proposition~\ref{prp:ext0}, and $e' \in \Ext^1(\cF',\cF'')$ is determined by $(w_1',\ldots,w_m') \in (\End V)^m$ as in Proposition~\ref{prp:ext1}.
To calculate $e' \circ u$, it suffices to calculate the corresponding data $(w_1,\ldots,w_m)$ for $e'\circ u$ in terms of $(u_1,u_2)$ and $(w_1',\ldots,w_m')$.

The pullback $u^*\mathcal{G}$ is now given by the diagram in Figure~\ref{fig:complicated-diagram}, where all squares commute.
From the definition of pullback, we get that the map $U_i$ is given by the matrix
$\left(\begin{smallmatrix}
1 & 0 \\
0 & u_i  
\end{smallmatrix}\right)$ and 
$W$ by 
$\left(\begin{smallmatrix}
  1 & 0 & 0 & 0\\
0 & u_2 & 0 & 0 \\
0 & 0 & 1 & 0\\
0 & 0& 0 & u_1 
\end{smallmatrix}\right)$.
As in Section~\ref{ssec:ext1}, from the relation $ W\circ \Phi_k=\Phi'_k\circ U_k$ we obtain that $W\circ (\Phi_k\oplus \Phi_{k+1})=(\Phi'_k\oplus \Phi'_{k+1})\circ (U_k\oplus U_{k+1})$, which implies the commutativity of each square in the following diagram: 

\begin{equation}
\begin{gathered}
\xymatrix{
V^2 \oplus V^2 \ar[rr]^{\Omega_m} \ar[d]^{U_{m}\oplus U_{m+1}} &&
\cdots \ar[rr]^{\Omega_2} &&
V^2 \oplus V^2  \ar[rr]^{\Omega_1}  \ar[d]^{U_{1}\oplus U_{2}}&&
V^4  \ar[d]^{W} \\
V^2 \oplus V^2 \ar[rr]^{\Omega'_m}  &&
\cdots \ar[rr]^{\Omega'_2} &&
V^2 \oplus V^2  \ar[rr]^{\Omega'_1} &&
V^4 . 
}
\end{gathered}
\label{eq:pullback}
\end{equation}
Here, as in Proposition~\ref{prp:ext1}, we have
$\Omega_k= \left( \begin{smallmatrix} 0 & 0 & 1 & 0 \\
0 & 0 & 0 & 1 \\
1 & 0 & A''_k & w_k \\
0 & 1 & 0 & A_k
\end{smallmatrix} \right)$
and 
$\Omega'_k= \left( \begin{smallmatrix} 0 & 0 & 1 & 0 \\
0 & 0 & 0 & 1 \\
1 & 0 & A''_k & w'_k \\
0 & 1 & 0 & A'_k
\end{smallmatrix} \right)$.

Plugging the formulas for $\Omega_k,\Omega_k',U_k,W$ into the commutative squares in \eqref{eq:pullback} yields $w_k'u_{k+1} = w_k$ for all $k$. It follows that the composition of the extension $e' = (w_1',\ldots,w_m')$ with $u = (u_1,u_2)$ is $e'\circ u = (w_1'u_2,w_2'u_1,w_3'u_2,\ldots)$.

Similar computations for the pushforward of extensions imply that if $u'=(u'_1,u'_2):\cF'\rightarrow \cF''$ is an element of $\Ext^0(\cF',\cF'')$ and $e=(w_1,\cdots,w_m)\in \Ext^1(\cF,\cF')$, then $u'\circ e=(u'_1w_1,u'_2w_2,u'_1w_3,\ldots)$.

\section{Vanishing of \texorpdfstring{$\Ext^2$}{Ext2}}
\label{sec:vanishing-ext_2}

Our goal in this section is to prove the following.

\begin{prop} \label{prop:ext2-vanishes}
Let $\Lambda$ be a rainbow braid closure and $\coeffs$ a field, and fix integers $r,s \geq 1$.  Given $\cF \in \Sh_r(\Lambda,\coeffs)$ and $\cG \in \Sh_s(\Lambda,\coeffs)$, we have $\Ext^i(\cF,\cG)=0$ for all $i \neq 0,1$.
\end{prop}

\begin{proof}
Since $\Lambda$ carries a binary Maslov potential, such $\cF$ and $\cG$ are equivalent to their zeroth cohomology sheaves (see the discussion in Section \ref{ssec:sheaf-obj}), so without loss of generality we can assume that $\cF$ and $\cG$ are honest sheaves of $\kk$-modules.

We will first prove Proposition~\ref{prop:h2shom-vanishes}, asserting that $H^i(\R^2; \sHom(\cF,\cG)) = 0$ for all $i \geq 2$, in Section~\ref{ssec:cech-h2}.  In Section~\ref{ssec:rhom-equals-hom}, we will then prove Proposition~\ref{prop:rhom-equals-hom}, asserting that $R\sHom(\cF,\cG) \simeq \sHom(\cF,\cG)$.  Combining these facts, we see that $\Ext^\bullet(\cF,\cG)$ is quasi-isomorphic to
\[ R\Gamma(\R^2, R\sHom(\cF,\cG)) \simeq R\Gamma(\R^2, \sHom(\cF,\cG)), \]
which is the sheaf cohomology of $\sHom(\cF,\cG)$ and thus vanishes in degrees $i<0$ and $i\geq 2$, as desired.
\end{proof}

\subsection{A computation of \texorpdfstring{\v{C}ech}{Cech} cohomology} \label{ssec:cech-h2}
\def\hexsize{1.5}
\def\dotsize{\hexsize*2pt}
\def\reddot(#1){ \fill [color=red] (#1) circle[radius=\dotsize]; }
\def\bluedot(#1){ \fill [color=blue] (#1) +(-\dotsize,-\dotsize) rectangle +(\dotsize,\dotsize); }
\tikzset{hexagon/.style={shape=regular polygon, regular polygon sides=6,
		minimum size=\hexsize cm, draw, inner sep=0}}
\newcommand{\placehex}[2]{
  \ifodd#1
    \node[hexagon] (h#1;#2) at ({(#1*\hexsize/2+#1*\hexsize/4)},{(#2)*\hexsize*sin(60)}) {};
  \else
    \node[hexagon] (h#1;#2) at ({(#1*\hexsize/2+#1*\hexsize/4)},{(#2+1/2)*\hexsize*sin(60)}) {};
  \fi
}
\def\drawcrossing(#1)(#2)(#3)(#4){ \draw (#1)--(#2) (#3) -- ($(#3)!.4!(#4)$) (#4) -- ($(#4)!.4!(#3)$); }

In this subsection we will prove that
\[ H^i(\R^2; \sHom(\cF,\cG)) = 0 
\text{ for all } i \geq 2, \]
where $\cF \in \Sh_r(\Lambda,\coeffs)$ and $\cG \in \Sh_s(\Lambda,\coeffs)$ are sheaves of $\coeffs$-modules with microsupport on a rainbow braid closure $\Lambda$.  Letting $\cH = \sHom(\cF,\cG)$, we will compute the \v{C}ech cohomology of $\cH$, using the fact that 
\[ \lim_{\substack{\longrightarrow \\ \cU}} \check{H}^*(\cU; \cH) \cong H^*(\R^2;\cH) \]
as $\cU$ ranges over open coverings of $\R^2$ \cite[Th\'{e}or\`{e}me~II.5.10.1]{Godement}.  It will therefore suffice to show that every $\cU$ admits a refinement $\cV$ such that $\check{H}^i(\cV;\cH) = 0$ for all $i \geq 2$.

We explain our construction of $\cV$ from a given open cover $\cU$ of $\R^2$.  We tile the plane by hexagons of side length $\epsilon > 0$, arranging by a small isotopy that the front projection of $\Lambda$ only intersects these hexagons transversely along the midpoints of edges.  We let $D \subset \R^2$ be a union of finitely many tiles $D_1,\dots,D_n$ which is topologically a disk and which contains an open neighborhood of the front.  By taking the side length $\epsilon$ sufficiently small and perturbing $\Lambda$ slightly, we can also arrange that
\begin{itemize}
\item every tile $D_i \subset D$ is contained in some open set $U_i \in \cU$;
\item every tile contains at most one crossing or cusp;
\item the two horizontal edges of a hexagon do not intersect the front;
\item if a tile does not contain a crossing, then its boundary intersects the front at most twice, and it does so along consecutive edges if and only if it contains a cusp.
\end{itemize}
In other words, the tiles can only intersect $\Lambda$ in the configurations shown in Figure~\ref{fig:possible-tiles}, and they fit together to produce a front as in the left side of Figure~\ref{fig:tiling}.  For some positive $\delta \ll \epsilon$, each tile $D_i$ has a $\delta$-neighborhood $V_i \subset U_i$, so we let
\[ \cV = \{V_1,\dots,V_n\} \cup \left\{ U \cap (\R^2 \ssm D) \mid U \in U_i \right\}, \]
which is an open cover of the plane refining $\cU$.

\begin{figure}
\begin{tikzpicture}
\def\hexsize{1.25}
\foreach \i in {0,2,4,6,8,10,12} {\placehex{\i}{0};}
\begin{scope}[every path/.style = very thick]
  \draw (h0;0.side 3) -- (h0;0.side 6);
  \draw (h2;0.side 2) -- (h2;0.side 5);
  \begin{scope}[every path/.append style={rounded corners=10pt}]
    \draw (h4;0.side 3) -- (h4;0.center) -- (h4;0.side 5);
    \draw (h6;0.side 2) -- (h6;0.center) -- (h6;0.side 6);
  \end{scope}
  \draw (h8;0.side 5) -- (h8;0.center) -- (h8;0.side 6);
  \draw (h10;0.side 2) -- (h10;0.center) -- (h10;0.side 3);
  \drawcrossing(h12;0.side 2)(h12;0.side 5)(h12;0.side 3)(h12;0.side 6);
\end{scope}
\end{tikzpicture}
\caption{The possible ways in which the front for $\Lambda$ can intersect a hexagonal tile.}
\label{fig:possible-tiles}
\end{figure}
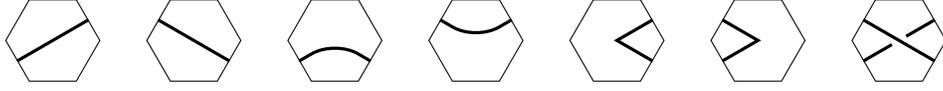

\begin{figure}
\begin{tikzpicture}
\def\smallhexexample{0/0, 1/0, 1/1, 2/-1, 2/0, 2/1, 3/0, 3/1, 4/0}
\begin{scope}[hexagon/.append style=dashed]
  \foreach \i/\j in \smallhexexample { \placehex{\i}{\j}; }
  \begin{scope}[every path/.style=very thick,rounded corners=10pt]
    \draw (h0;0.center) -- (h2;1.center) -- (h4;0.center);
    \draw (h0;0.center) -- (h2;-1.center) -- (h4;0.center);
  \end{scope}
\end{scope}
\begin{scope}[xshift=200]
  \foreach \i/\j in \smallhexexample { \placehex{\i}{\j}; }
  \begin{scope}[every path/.style=thin,dashed,rounded corners=10pt]
    \draw (h0;0.center) -- (h2;1.center) -- (h4;0.center);
    \draw (h0;0.center) -- (h2;-1.center) -- (h4;0.center);
  \end{scope}
  \begin{scope}
    \clip ($(h0;0.center)+(-0.25,0)$) -- ($(h2;1.center)+(0,0.25)$) -- ($(h4;0.center)+(0.25,0)$) -- ($(h2;-1.center)+(0,-0.25)$) -- cycle;
    \foreach \i/\j in \smallhexexample {
      \foreach \k in {1,...,6} { \reddot(h\i;\j.corner \k); }
      \foreach \k in {1,...,6} { \bluedot(h\i;\j.side \k); }
    }
  \end{scope}
\end{scope}
\end{tikzpicture}
\caption{Left, the hexagonal tiling used to produce the refinement $\cV$ of the open cover $\cU$.  Right, the blue squares and red dots contribute generators to $\check{C}^1(\cV;\cH)$ and $\check{C}^2(\cV;\cH)$ respectively.}
\label{fig:tiling}
\end{figure}
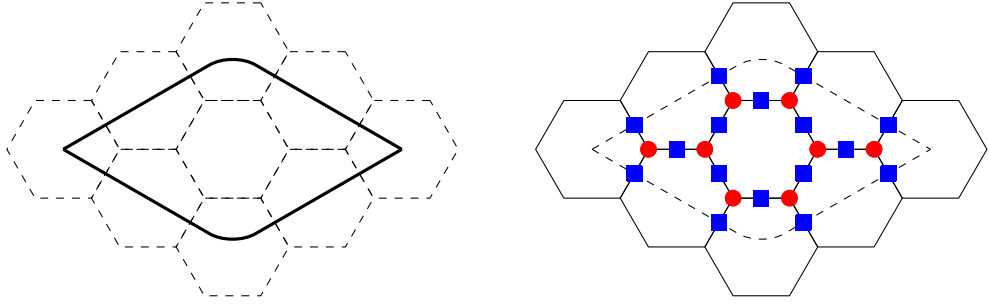

The \v{C}ech complex $\check{C}^*(\cV;\cH)$ now admits a simple description.  The sheaves $\cF$ and $\cG$ are zero outside $D$, and hence so is $\cH = \sHom(\cF,\cG)$.  Since arbitrary intersections of open sets which lie in $\R^2 \ssm D$ do not contribute to $\check{C}^*(\cV;\cH)$, we can discard them.  The remaining open sets are $\delta$-neighborhoods of the hexagonal tiles $D_1,\dots,D_n$, which intersect in pairs along a neighborhood of each edge $e_{ij} = D_i \cap D_j$ of the tiling and in triples along a neighborhood of each vertex $v_{ijk} = D_i \cap D_j \cap D_k$, as shown in the right side of Figure~\ref{fig:tiling}.  There are no nonempty $(k+1)$-fold intersections in $D$ for $k \geq 3$; hence $\check{C}^k(\cV;\cH)=0$ for all such $k$ and we have
\[ \check{H}^2(\cV;\cH) = \coker\big(d^1: \check{C}^1(\cV;\cH) \to \check{C}^2(\cV;\cH)\big). \]
Thus it suffices to show that $d^1$ is surjective.

The map $d^1$ can be described combinatorially as follows.  Each triple intersection
\[ V_{ijk} := V_i \cap V_j \cap V_k \ni v_{ijk} \]
lies in a two-dimensional stratum of the plane, so $\cF$ and $\cG$ and thus $\cH = \sHom(\cF,\cG)$ are locally constant on $V_{ijk}$, from which
\[ \cH(V_{ijk}) = \Hom(\cF(V_{ijk}), \cG(V_{ijk})). \]
A nonempty double intersection
\[ V_{ij} := V_i \cap V_j \]
along an edge $e_{ij}$ of the tiling contains two triple intersections $V_{ijk_1}$ and $V_{ijk_2}$ corresponding to either endpoint of $e_{ij}$.  (This is a slight abuse of notation if one of the endpoints lies on $\partial D$, but in this case $\cF$, $\cG$, and $\cH$ are zero on any triple intersection near this endpoint and we can ignore it.)  The restriction maps $\cH(V_{ij}) \to \cH(V_{ijk})$ fall into one of three cases:
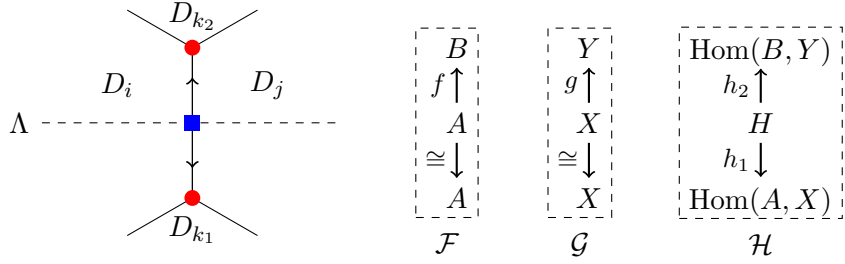
\begin{figure}
\begin{tikzpicture}
\tikzset{->-/.style={decoration={markings,mark=at position #1 with {\arrow[thick]{>}}},postaction={decorate}}}
\begin{scope}
  \draw (0,-1) -- (0,1);
  \draw[->-=0.6] (0,0) -- (0,1);
  \draw[->-=0.6] (0,0) -- (0,-1);
  \foreach \j in {-1,1} {
    \foreach \i in {30,150} { \draw (0,\j) -- +(\i*\j:1); }
    \reddot (0,\j);
  }
  \draw[thin,dashed] (-2,0) -- (2,0);
  \node at (-2,0) [left=1] {$\Lambda$};
  \node at (-1,0.5) {$D_i$};
  \node at (1,0.5) {$D_j$};
  \node at (0,1) [above=3] {$D_{k_2}$};
  \node at (0,-1) [below=3] {$D_{k_1}$};
  \bluedot (0,0);
\end{scope}
\begin{scope}[xshift=100,local bounding box=scopeF]
  \node (A) at (0,1) {$B$};
  \node (B) at (0,0) {$A$};
  \node (C) at (0,-1) {$A$};
  \draw [thick, ->] (B) -- (A) node[font=\small,midway,sloped,left,rotate=270] {$f$};
  \draw [thick, ->] (B) -- (C) node[font=\small,midway,sloped,left,rotate=90] {$\cong$};
\end{scope}
\begin{scope}[xshift=150,local bounding box=scopeG]
  \node (A) at (0,1) {$Y$};
  \node (B) at (0,0) {$X$};
  \node (C) at (0,-1) {$X$};
  \draw [thick, ->] (B) -- (A) node[font=\small,midway,sloped,left,rotate=270] {$g$};
  \draw [thick, ->] (B) -- (C) node[font=\small,midway,sloped,left,rotate=90] {$\cong$};
\end{scope}
\begin{scope}[xshift=215,local bounding box=scopeH]
  \node[text depth=0pt] (A) at (0,1) {$\Hom(B,Y)$};
  \node (B) at (0,0) {$H$};
  \node[text depth=0pt] (C) at (0,-1) {$\Hom(A,X)$};
  \draw [thick, ->] (B) -- (A) node[font=\small,midway,sloped,left,rotate=270] {$h_2$};
  \draw [thick, ->] (B) -- (C) node[font=\small,midway,sloped,left,rotate=90] {$h_1$};
\end{scope}
\foreach \s in {F,G,H} {
  \draw[dashed] (scope\s.north east) rectangle (scope\s.south west);
  \node at (scope\s.south) [below=2] {$\mathcal{\s}$};
}
\end{tikzpicture}
\caption{A portion of the \v{C}ech complex corresponding to the restriction maps $\cH(V_{ij}) \to \cH(V_{ijk_l})$, near an edge $e = D_i \cap D_j$ of the tiling which intersects $\Lambda$.  The diagrams at right represent $\cF$, $\cG$, and $\cH$ near $e$.}
\label{fig:cech-d1-local}
\end{figure}
\begin{itemize}
\item If $\Lambda$ does not pass through $e_{ij}$, then $\cF$, $\cG$, $\cH$ are all locally constant near $e_{ij}$ and the restriction maps 
\[ \cH(V_{ijk_1}) \leftarrow \cH(V_{ij}) \rightarrow \cH(V_{ijk_2}) \]
are both isomorphisms.
\item If a strand of $\Lambda$ with Maslov potential $0$ passes through $e_{ij}$, then the restriction maps $f,g$ depicted in Figure~\ref{fig:cech-d1-local} are injective.  In this case the local sections $H = \cH(V_{ij})$ are in bijection with pairs $(\varphi: A \to X, \psi: B\to Y)$ such that the diagram
\[ \xymatrix{ B\ar[r]^\psi & Y \\ A\ar[r]^\varphi\ar[u]^f & X\ar[u]_g } \]
commutes.  The restriction maps $h_1$ and $h_2$ send $(\varphi,\psi)$ to $\varphi$ and $\psi$ respectively.  Then $h_1$ is surjective, since every $\varphi$ can be lifted to some $\psi$; and $h_2$ is injective with image
\[ \img(h_2) = \{\psi: B \to Y \mid \img(\psi \circ f) \subset \img(g) \}, \]
since every $\psi$ coming from such a diagram determines the corresponding $\varphi$ uniquely.
\item If a strand of $\Lambda$ with Maslov potential $1$ passes through $e_{ij}$, then $f,g$ are surjective in Figure~\ref{fig:cech-d1-local}, so a similar argument says that $h_1$ is injective with image
\[ \img(h_1) = \{ \varphi: A \to X \mid \ker(f) \subset \ker(g\circ\varphi) \} \]
and $h_2$ is surjective.
\end{itemize}

In Figure~\ref{fig:cech-big-unknot}, we illustrate $d^1: \check{C}^1 \to \check{C}^2$ for a particular Legendrian unknot $\Lambda$ (drawn as a dotted curve) and tiling of the plane.  Here each blue square along an edge $e_{ij}$ contributes $\cH(V_{ij})$ to $\check{C}^1(\cV;\cH)$, each red dot at a vertex $v_{ijk}$ contributes $\cH(V_{ijk})$ to $\check{C}^2(\cV;\cH)$, and each line segment connecting a blue square to a red dot gives a component $\cH(V_{ij}) \to \cH(V_{ijk})$ of the differential $d^1$. 

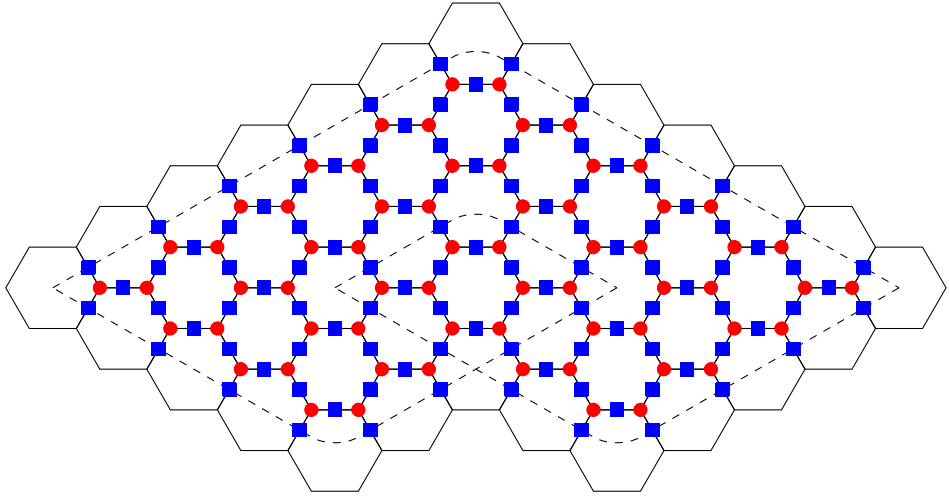
\begin{figure}
\begin{tikzpicture}
\def\hexsize{1.25}
\def\largehexexample{0/0, 1/0, 1/1, 2/-1, 2/0, 2/1, 3/-1, 3/0, 3/1, 3/2, 4/-2, 4/-1, 4/0, 4/1, 4/2, 5/-1, 5/0, 5/1, 5/2, 5/3, 6/-1, 6/0, 6/1, 6/2, 6/3, 7/-1, 7/0, 7/1, 7/2, 7/3, 8/-2, 8/-1, 8/0, 8/1, 8/2, 9/-1, 9/0, 9/1, 9/2, 10/-1, 10/0, 10/1, 11/0, 11/1, 12/0}
\begin{scope}
  \foreach \i/\j in \largehexexample { \placehex{\i}{\j}; }
  \begin{scope}[every path/.style=thin,dashed,rounded corners=10pt]
    \draw (h0;0.center) -- (h6;3.center) -- (h12;0.center);
    \draw (h12;0.center) -- (h8;-2.center) -- (h4;0.center);
    \draw (h4;0.center) -- (h6;1.center) -- (h8;0.center);
    \draw (h8;0.center) -- (h4;-2.center) -- (h0;0.center);
  \end{scope}
  \begin{scope}
    \clip ($(h0;0.center)+(-0.25,0)$) -- ($(h6;3.center)+(0,0.25)$) -- ($(h12;0.center)+(0.25,0)$) -- ($(h8;-2.center)+(0,-0.25)$) -- ($(h6;-1.center)+(0,-0.25)$) -- ($(h4;-2.center)+(0,-0.25)$) -- cycle;
    \foreach \i/\j in \largehexexample {
      \foreach \k in {1,...,6} { \reddot(h\i;\j.corner \k); }
      \foreach \k in {1,...,6} { \bluedot(h\i;\j.side \k); }
    }
  \end{scope}
\end{scope}
\end{tikzpicture}
\caption{The \v{C}ech complex in degrees 1 and 2 for a Legendrian knot $\Lambda$.}
\label{fig:cech-big-unknot}
\end{figure}

We now prove that $d^1$ is surjective, and hence that $\check{H}^2(\cV;\cH) = 0$, by playing the following game.  We treat the blue squares and red dots in Figure~\ref{fig:cech-big-unknot} as vertices of a bipartite graph $\Gamma$, whose edges are given by the half-edges of hexagons which connect them.  We call a blue square a \emph{leaf} if it is adjacent to exactly one red dot.  We produce a new graph $\Gamma'$ by removing one of the following:
\begin{itemize}
\item a collection of leaves, or
\item some collection of leaves on edges $\{e_{i_lj_l}\}$ \emph{and} all of the red dots at vertices $\{v_{i_lj_lk_l}\}$ adjacent to them, but only if the corresponding restriction map
\[ \bigoplus_l \cH(V_{i_lj_l}) \to \bigoplus_l \cH(V_{i_lj_lk_l}) \]
is surjective.  (The collection of vertices should be read as a set rather than a list, since some leaves may be adjacent to the same vertex.)
\end{itemize}
To any such graph $\Gamma$, we can associate a map
\[ d_\Gamma: C^1(\Gamma) \to C^2(\Gamma) \]
by letting $C^1(\Gamma)$ be the direct sum of the $\cH(V_{ij})$ corresponding to blue squares along edges $e_{ij}$, letting $C^2(\Gamma)$ be the direct sum of the $\cH(V_{ijk})$ over vertices $v_{ijk}$ with red dots, and taking $d_\Gamma$ to be the sum of all of the corresponding restriction maps $\cH(V_{ij}) \to \cH(V_{ijk})$.  If we have not yet removed any squares or dots, then $d_\Gamma$ coincides with the original differential $d^1: \check{C}^1(\cV;\cH) \to \check{C}^2(\cV;\cH)$.

\begin{lemma} \label{lem:remove-pegs}
If $d_{\Gamma'}$ is surjective, then so is $d_{\Gamma}$.
\end{lemma}

\begin{proof}
Supposing that we have only removed a leaf corresponding to the edge $e_{ij}$ and adjacent to the vertex $e_{ijk}$, we can write
\[ C^1(\Gamma) = \cH(V_{ij}) \oplus C^1(\Gamma'), \qquad d_\Gamma(s_{ij},s') = \rho_{ij,ijk}(s_{ij}) + d_{\Gamma'}(s') \]
where $\rho_{ij,ijk}: \cH(V_{ij}) \to \cH(V_{ijk})$ is the corresponding restriction map.  Assuming that $d_{\Gamma'} = d_\Gamma(0,\cdot)$ is surjective, it follows immediately that $d_\Gamma$ is as well.  We can repeat this argument to remove as many leaves as we like.

Now suppose instead that we have removed some leaves along edges $\{e_{i_lj_l}\}$ and all of the adjacent red dots, which are at vertices $\{v_{i_lj_lk_l}\}$.  Letting $A^1 = \bigoplus_l \cH(V_{i_lj_l})$ and $A^2 = \bigoplus_l \cH(V_{i_lj_lk_l})$, we can then write
\[ d_{\Gamma} = \twomatrix{f}{g}{0}{d_{\Gamma'}}: A^1 \oplus C^1(\Gamma') \to A^2 \oplus C^2(\Gamma'), \]
in which the component $A^1 \to C^2(\Gamma')$ is zero because none of the removed blue squares are adjacent to red dots which remain in $\Gamma'$.  The map $f: A^1 \to A^2$ is assumed surjective, so if $d_{\Gamma'}$ is also surjective then  
$d_\Gamma$ is as well.
\end{proof}

If we can remove vertices from $\Gamma$ until no red dots remain, the corresponding $d_{\Gamma'}$ will clearly be surjective since its codomain is zero, and it will follow from Lemma~\ref{lem:remove-pegs} that $d^1: \check{C}^1(\cV;\cH) \to \check{C}^2(\cV;\cH)$ is surjective.  This will establish that $\check{H}^2(\cV;\cH) = 0$, as desired.

\begin{lemma} \label{lem:cusp-surjective}
Let $D_0$ be a hexagonal tile containing a cusp of $\Lambda$, and $D_1$ and $D_2$ the adjacent tiles such that $\Lambda$ intersects the edges $e_{01}$ and $e_{02}$.  Then the component
\[ \cH(V_{01}) \oplus \cH(V_{02}) \to \cH(V_{012}) \]
of the differential $d^1: \check{C}^1(\cV;\cH) \to \check{C}^2(\cV;\cH)$ is surjective.
\end{lemma}

\begin{proof}
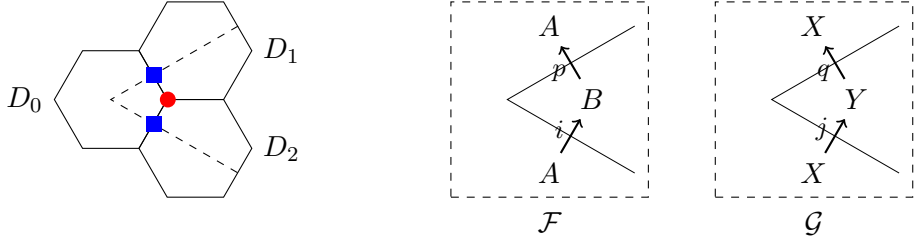
\begin{figure}
\begin{tikzpicture}
\def\cusptiles{0/0,1/0,1/1}
\begin{scope}
  \foreach \i/\j in \cusptiles { \placehex{\i}{\j}; }
  \draw[thin,dashed] (h1;1.side 6) -- (h0;0.center) -- (h1;0.side 5);
  \reddot(h0;0.corner 6);
  \bluedot(h0;0.side 5);
  \bluedot(h0;0.side 6);
  \node at (h0;0.corner 3) [left=0.25] {$D_0$};
  \node at (h1;1.corner 6) [right=0.25] {$D_1$};
  \node at (h1;0.corner 6) [right=0.25] {$D_2$};
\end{scope}
\begin{scope}[xshift=150, local bounding box=scopeF, hexagon/.append style={draw=none}]
  \foreach \i/\j in \cusptiles { \placehex{\i}{\j}; }
  \draw[very thin] (h1;1.side 6) -- (h0;0.center) -- (h1;0.side 5);
  \node (A) at (h1;0.side 3) {$A$};
  \node (B) at (h1;0.side 1) {$B$};
  \node (C) at (h1;1.side 2) {$A$};
  \begin{scope}[every node/.style={font=\small,left}]
    \draw [thick, ->] (A) -- node[pos=0.7] {$i$} (B);
    \draw [thick, ->] (B) -- node[pos=0.2] {$p$} (C);
  \end{scope}
\end{scope}
\begin{scope}[xshift=250, local bounding box=scopeG, hexagon/.append style={draw=none}]
  \foreach \i/\j in \cusptiles { \placehex{\i}{\j}; }
  \draw[very thin] (h1;1.side 6) -- (h0;0.center) -- (h1;0.side 5);
  \node (A) at (h1;0.side 3) {$X$};
  \node (B) at (h1;0.side 1) {$Y$};
  \node (C) at (h1;1.side 2) {$X$};
  \begin{scope}[every node/.style={font=\small,left}]
    \draw [thick, ->] (A) -- node[pos=0.7] {$j$} (B);
    \draw [thick, ->] (B) -- node[pos=0.2] {$q$} (C);
  \end{scope}
\end{scope}
\foreach \s in {F,G} {
  \draw[dashed] (scope\s.north east) rectangle (scope\s.south west);
  \node at (scope\s.south) [below=2] {$\mathcal{\s}$};
}
\end{tikzpicture}
\caption{A portion of the \v{C}ech complex near a cusp of $\Lambda$.}
\label{fig:cusp-tiles}
\end{figure}

Suppose that the sheaves $\cF$ and $\cG$ are represented by diagrams
\[ A \xrightarrow{i} B \xrightarrow{p} A \qquad\mathrm{and}\qquad X \xrightarrow{j} Y \xrightarrow{q} X \]
near the cusp, as illustrated in Figure~\ref{fig:cusp-tiles}.  Then we have seen that
\begin{align*}
\img\big(\cH(V_{02}) \to \cH(V_{012})\big) &= \{ \psi: B \to Y \mid \img(\psi\circ i) \subset \img(j) \} \\
\img\big(\cH(V_{01}) \to \cH(V_{012})\big) &= \{ \varphi: B \to Y \mid \ker(p) \subset \ker(q\circ\varphi) \}.
\end{align*}
We therefore wish to show that every $f: B \to Y$ can be written as a sum $f = \psi + \varphi$, where $\psi$ and $\varphi$ are as above.

Since $p\circ i: A \to A$ is an isomorphism and likewise for $q\circ j: X \to X$, we have direct sum decompositions
\[ B \cong A \oplus K, \qquad Y \cong X \oplus L \]
for which $i(a) = (a,0)$ and $p(a,k) = \alpha(a)$, and $j(x) = (x,0)$ and $q(x,l) = \beta(x)$ for some automorphisms $\alpha: A \to A$ and $\beta: X \to X$.  Viewing $\psi$ and $\varphi$ as maps $A\oplus K \to X \oplus L$, the conditions $\img(\psi\circ i) \subset \img(j)$ and $\ker(p) \subset \ker(q\circ\varphi)$ are equivalent to $\psi$ and $\varphi$ having block forms
\[ \psi = \twomatrix{*}{*}{0}{*}, \qquad \varphi=\twomatrix{*}{0}{*}{*}, \]
so it is clear that any $f: B \to Y$ can be written as such a sum, as desired.
\end{proof}

\newcommand{\rotatey}{\protect\rotatebox[origin=c]{90}{{\sf Y}}}
\begin{prop} \label{prop:remove-y}
Suppose that $\Gamma$ contains a subgraph $\rotatey$ (resembling the letter ``{\sf Y}'', rotated 90 degrees counterclockwise) of the following form:
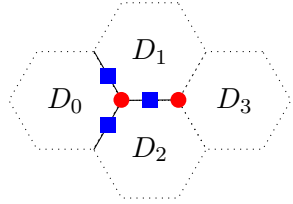
\begin{center}
\begin{tikzpicture}[hexagon/.append style=dotted]
\foreach \i/\j in {0/0,1/0,1/1,2/0} { \placehex{\i}{\j}; }
\draw (h0;0.corner 6) edge (h1;0.corner 1) edge (h0;0.corner 1) edge (h0;0.corner 5);
\foreach \r in {h1;0.corner 1, h1;0.corner 2} { \reddot (\r); }
\foreach \b in {h0;0.side 5, h0;0.side 6, h1;0.side 1} { \bluedot (\b); }
\node at (h0;0.center) {$D_0$};
\node at (h1;1.center) {$D_1$};
\node at (h1;0.center) {$D_2$};
\node at (h2;0.center) {$D_3$};
\end{tikzpicture}
\end{center}
\noindent consisting of the three blue squares along three edges $e_{01}$, $e_{02}$, and $e_{12}$, and the two red dots at the endpoints $v_{012}$ and $v_{123}$ of the horizontal edge $e_{12}$.  Suppose moreover that the blue squares on $e_{01}$ and $e_{02}$ are both leaves.  (We make no claims about whether there are vertices along the dotted edges.)  Let $\Gamma' = \Gamma \ssm \rotatey$ be the subgraph produced by removing these five vertices.  If $d_{\Gamma'}$ is surjective, then $d_{\Gamma}$ is surjective as well.
\end{prop}

\begin{proof}
Our goal will be to apply Lemma~\ref{lem:remove-pegs} to the three vertices on $e_{01}$, $e_{02}$, and $v_{012}$.  If we can do this, then the vertex on $e_{12}$ will become a leaf connected only to $v_{123}$, and the restriction map
\[ \rho_{12,123}: \cH(V_{12}) \to \cH(V_{123}) \]
is automatically an isomorphism since the front projection of $\Lambda$ does not pass through the horizontal edge $e_{12}$.  In this case, we can apply Lemma~\ref{lem:remove-pegs} again to remove the vertices on $e_{12}$ and $v_{123}$, leaving us with the subgraph $\Gamma'$.  It will thus suffice to show that the sum of restriction maps
\[ \rho_{01,012}+\rho_{02,012}: \cH(V_{01}) \oplus \cH(V_{02}) \to \cH(V_{012}) \]
is surjective.

If the front does not pass through $e_{01}$, then the restriction map $\rho_{01,012}: \cH(V_{01}) \to \cH(V_{012})$ is an isomorphism, and likewise if it avoids $e_{02}$ then $\rho_{02,012}: \cH(V_{02}) \to \cH(V_{012})$ is an isomorphism.  This establishes the desired surjectivity except when $\Lambda$ intersects both of the edges $e_{01}$ and $e_{02}$, in which case $D_0$ contains either a left cusp of $\Lambda$ or a crossing.

If $D_0$ contains a crossing, then since both strands have Maslov potential 0, we have argued that the downward restriction map $\rho_{01,012}: \cH(V_{01}) \to \cH(V_{012})$ is surjective.  Otherwise $D_0$ contains a cusp, and then Lemma~\ref{lem:cusp-surjective} says precisely that $\rho_{01,012}+\rho_{02,012}$ is surjective, as desired.
\end{proof}

\begin{prop} \label{prop:h2shom-vanishes}
Let $\Lambda$ be a rainbow braid closure, and let $\cF \in \Sh_r(\Lambda)$ and $\cG \in \Sh_s(\Lambda)$ be sheaves of $\kk$-modules.  Then
\[ H^i(\R^2; \sHom(\cF,\cG)) = 0 \]
for all $i \geq 2$.
\end{prop}

\begin{proof}
Let $\cH = \sHom(\cF,\cG)$.  As explained at the beginning of this subsection, we have
\[ H^*(\R^2; \sHom(\cF,\cG)) \cong \lim_{\substack{\longrightarrow \\ \cU}} \check{H}^*(\cU; \cH), \]
and so it suffices to exhibit for any open cover $\cU$ of $\R^2$ a refinement $\cV$ such that $\check{H}^i(\cV;\cH) = 0$ for $i \geq 2$.  We produce $\cV$ from a tiling of $\R^2$ by hexagons of sufficiently small side length $\epsilon > 0$, as described above.  We have already seen that $\check{C}^i(\cV;\cH)=0$ for $i \geq 3$, so we need only show that $d^1: \check{C}^1(\cV;\cH) \to \check{C}^2(\cV;\cH)$ is surjective.

The key observation we need is that in the region $D = \bigcup D_i$ containing the front, every horizontal edge belongs to a unique $\rotatey$, and all of these $\rotatey$ subgraphs are disjoint.  We sort the horizontal edges on the interior of $D$ from leftmost to rightmost, breaking ties arbitrarily, and let $\rotatey_1, \dots, \rotatey_n$ denote the corresponding $\rotatey$ subgraphs.  (Here we can ignore horizontal edges on the boundary of $D$, because they lie in the unbounded region of the complement of the front in $\R^2$, and hence $\cH$ vanishes along them.)  We will now apply Proposition~\ref{prop:remove-y} to remove $\rotatey_1, \dots, \rotatey_n$ in order.

Fix $j \leq n$ and suppose that we have removed $\rotatey_i$ for all $i < j$.  Then we claim that the two blue squares on the left end of $\rotatey_j$ must be leaves.  Indeed, labeling the tiles around $\rotatey_j$ as in Figure~\ref{fig:remove-y-inductively}, suppose that the blue square on $e_{13}$ is not a leaf.  Then the red dot at $v_{123}$ must not have been removed yet.  It belongs to some $\rotatey_i$ which must satisfy $i<j$, since $e_{12}$ lies strictly to the left of $e_{34}$, and which was therefore already removed.  But this $\rotatey_i$ includes the red dot at $v_{123}$, so we have a contradiction.  The blue square on $e_{13}$ is therefore a leaf, and applying the same argument at $v_{014}$ says that the blue square on $e_{14}$ is a leaf as well.  We now remove $\rotatey_j$ and continue by induction.

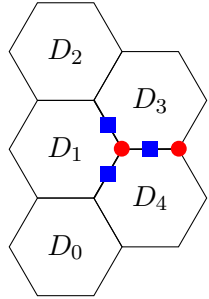
\begin{figure}
\begin{tikzpicture}
\foreach \i/\j in {0/-1,0/0,0/1,1/0,1/1} { \placehex{\i}{\j}; }
\draw (h0;0.corner 6) edge (h1;0.corner 1) edge (h0;0.corner 1) edge (h0;0.corner 5);
\foreach \r in {h1;0.corner 1, h1;0.corner 2} { \reddot (\r); }
\foreach \b in {h0;0.side 5, h0;0.side 6, h1;0.side 1} { \bluedot (\b); }
\node at (h0;-1.center) {$D_0$};
\node at (h0;0.center) {$D_1$};
\node at (h0;1.center) {$D_2$};
\node at (h1;0.center) {$D_4$};
\node at (h1;1.center) {$D_3$};
\end{tikzpicture}
\caption{A subgraph $\rotatey_j$ which we wish to remove from the graph $\Gamma$.}
\label{fig:remove-y-inductively}
\end{figure}

The end result of this induction is that all of the $\rotatey$ subgraphs have been removed.  Since every red dot lies on a unique horizontal edge, and was thus a vertex of some $\rotatey_i$, the resulting graph $\Gamma'$ has no red vertices, and hence the map $d_{\Gamma'}$ is surjective.  This implies by Proposition~\ref{prop:remove-y} that the original $d_\Gamma = d^1: \check{C}^1(\cV;\cH) \to \check{C}^2(\cV;\cH)$ is surjective, as desired.
\end{proof}

\subsection{Vanishing of \texorpdfstring{$R^q\sHom(\cF,\cG)$}{RHom(F,G)} for \texorpdfstring{$q \geq 1$}{q >= 1}} \label{ssec:rhom-equals-hom}

The main result of this subsection is the following.

\begin{prop} \label{prop:rhom-equals-hom}
Let $\Lambda$ be a rainbow braid closure, and let $\cF \in \Sh_r(\Lambda,\coeffs)$ and $\cG \in \Sh_s(\Lambda,\coeffs)$ be sheaves of $\coeffs$-modules for some field $\coeffs$.  Then
\[ R^q\sHom(\cF,\cG) = 0 \]
for all $q \geq 1$.  In other words, $R\sHom(\cF,\cG) \simeq \sHom(\cF,\cG)$.
\end{prop}

Proposition~\ref{prop:rhom-equals-hom} can be checked locally, since for any open $U \subset \R^2$ we have
\[ R\sHom(\cF,\cG)|_U \simeq R\sHom(\cF|_U,\cG|_U). \]
Indeed, if $i: U \to \R^2$ is the inclusion map, and $0 \to \cG \to \cI^\bullet$ is an injective resolution, then $i^*$ is exact and sends injective sheaves to injective sheaves \cite[Proposition~2.4.1]{KashiSchapi}, so $0 \to \cG|_U \to \cI|_U^\bullet$ is also an injective resolution of $\cG|_U$.  (We use $i^*$ both here and in the sequel for the inverse image functor, as opposed to the notation $i^{-1}$ of \cite{KashiSchapi}.)  Then we have
\[ R\sHom(\cF,\cG)|_U = \sHom(\cF,\cI^\bullet)|_U \simeq \sHom(\cF|_U,\cI|^\bullet_U) = R\sHom(\cF|_U,\cG|_U) \]
as claimed.

We thus prove Proposition~\ref{prop:rhom-equals-hom} in several steps.  We will first prove it for contractible open subsets of 2-dimensional strata in Lemma~\ref{lem:rhom-constant}, and then for small open balls around points on the interior of strands in Propositions~\ref{prop:rhom-strand-0} and \ref{prop:rhom-strand-1} (for Maslov potentials 0 and 1 respectively).  This reduces the theorem to studying neighborhoods of crossings and cusps, which we then do using a model computation for rainbow braid closures representing a Legendrian unknot.  This last computation relies in turn on the vanishing of $H^2(\R^2;\sHom(\cF,\cG))$ established in Proposition~\ref{prop:h2shom-vanishes}.

We begin with neighborhoods of points in 2-dimensional strata.

\begin{lemma} \label{lem:rhom-constant}
Let $\csheaf[V][M]$ be the constant sheaf on a connected manifold $M$ whose stalk is a finite-dimensional $\kk$-vector space $V$, and let $\cF$ be any sheaf of $\kk$-modules on $M$.  Then
\[ R\sHom(\csheaf[V][M], \cF) \simeq \sHom(\csheaf[V][M], \cF). \]
In particular, if $W$ is another $\kk$-vector space then $R\sHom(\csheaf[V][M],\csheaf[W][M])$ is the constant sheaf on $M$ with stalk $\Hom(V,W)$.
\end{lemma}

\begin{proof}
If $\dim(V)=1$ then $\sHom(\csheaf[V][M], \cdot) \simeq \sHom(\csheaf[\kk][M],\cdot)$ is the identity functor, so its higher derived functors vanish.  Otherwise we induct on $\dim(V)$: we pick a 1-dimensional subspace $L \subset V$, with quotient $X = V/L$, and suppose that we have proved the result for $\kk$-vector spaces of dimension less than $\dim(V)$.  We apply the contravariant functor $R\sHom(\cdot, \cF)$ to the short exact sequence $0 \to \csheaf[L][M] \to \csheaf[V][M] \to \csheaf[X][M] \to 0$
to get an exact sequence
\[ 0 \to R\sHom(\csheaf[X][M], \cF) \to R\sHom(\csheaf[V][M], \cF) \to R\sHom(\csheaf[L][M], \cF) \to 0. \]
For all $j>0$ we have $R^j\sHom(\csheaf[X][M],\cF) \cong 0$ and $R^j\sHom(\csheaf[L][M],\cF) \cong 0$ by hypothesis, so by exactness the same is true for $R^j\sHom(\csheaf[V][M],\cF)$ and we are done.
\end{proof}

The following technical lemmas will help us establish the vanishing result near points on 1-dimensional strata, corresponding to strands of $\Lambda$.

\begin{lemma} \label{lem:rpushforward-inclusion}
Let $N$ be a codimension-0 submanifold of a manifold $M$, with inclusion map $i: N \hookrightarrow M$, and suppose that $M\smallsetminus N$ is also a manifold of dimension $\dim(M)$.  Then $Ri_*\csheaf[V][N] \simeq i_*\csheaf[V][N]$ for any $\kk$-vector space $V$.
\end{lemma}

\begin{proof}
The higher direct images $R^ji_*\csheaf[V][N]$ are the sheaves associated to the presheaves $U \mapsto H^j(i^{-1}(U); \csheaf[V][N]) = H^j(U \cap N; \csheaf[V][N])$.  If $p \in \overline{N}$, we can take local coordinates $(x_1,\dots,x_n)$ about $p=(0,\dots,0)$ in which $N$ is either a neighborhood of $p$ or the set $\{x_1 > 0\}$ or its complement; then for any sufficiently small open ball $U$ about $p$, the intersection $U \cap N$ is contractible and hence $H^j(U \cap N;\csheaf[V][N]) = 0$ for $j>0$.  Otherwise if $p \not\in \overline{N}$ then every sufficiently small neighborhood of $p$ satisfies $U\cap N = \emptyset$ and the same conclusion holds.  Thus if $j>0$ then the sheaf $R^ji_*\csheaf[V][N]$ must vanish, since its stalks are all zero.
\end{proof}

\begin{rem}
The hypothesis $\dim(M\smallsetminus N) = \dim(M)$ is necessary in Lemma~\ref{lem:rpushforward-inclusion}: if $M=\R^n$ and $N=\R^n\smallsetminus\{0\}$, and if $U$ is an open ball about the origin, then $H^{n-1}(U \cap N; \csheaf[\kk][N]) \cong H^{n-1}(S^{n-1}; \csheaf[\kk][N]) \cong \kk$ and so $R^{n-1}i_*\csheaf[\kk][N]$ is a skyscraper sheaf with stalk $\kk$ at $0$.
\end{rem}

\begin{lemma} \label{lem:rhom-lower-shriek}
Let $i: O \hookrightarrow U$ be the inclusion of an open set into $U$ such that $U\smallsetminus O \subset U$ has codimension $0$.  If $\cG$ is a sheaf of $\kk$-modules on $U$ such that the restriction $\cG|_O = i^*\cG$ is a constant sheaf, then
\[ R\sHom(i_!\csheaf[V][O], \cG) \simeq \sHom(i_!\csheaf[V][O],\cG) \]
for any finite-dimensional $\kk$-vector space $V$.
\end{lemma}

{\allowdisplaybreaks
\begin{proof}
We will show more generally that given an open inclusion $i: O \hookrightarrow U$, we have
\begin{align} \label{eq:rhom-lower-shriek-adjunction}
R\sHom(i_!\cF, \cG) \simeq Ri_*R\sHom(\cF, i^*\cG)
\end{align}
for any sheaves $\cF$ on $O$ and $\cG$ on $U$.  Since $i$ is an open inclusion, the functor $i_!$ is exact \cite[Proposition~2.5.4]{KashiSchapi}, so that
\[ R\sHom(i_!\cF, \cG) \simeq R\sHom(Ri_!\cF, \cG) \simeq Ri_*R\sHom(\cF, i^!\cG) \]
where the last adjunction is \cite[Proposition~3.1.10]{KashiSchapi}.  Now we have $i^!(\cdot) \simeq i^* \circ R\Gamma_O(\cdot)$ and $\Gamma_O(\cdot) = i_* \circ i^*$ by \cite[Proposition~3.1.12]{KashiSchapi} and  \cite[Proposition~2.3.9]{KashiSchapi} respectively; since $i^*$ takes injectives to injectives and is exact, the latter implies that $R\Gamma_O = Ri_* \circ i^*$.  Combining all of this and observing that $\cF \simeq i^*i_* \cF$ gives
\begin{align*}
R\sHom(i_!\cF, \cG) &\simeq Ri_*R\sHom(i^*i_*\cF, i^* \circ R\Gamma_O(\cG)) \\
&\simeq R\sHom(i_*\cF, Ri_* \circ i^* \circ R\Gamma_O(\cG)) \\
&\simeq R\sHom(i_*\cF, R\Gamma_O \circ R\Gamma_O(\cG)) \\
&\simeq R\sHom(i_*\cF, R\Gamma_O(\cG)) \\
&\simeq R\sHom(i_*\cF, Ri_* \circ i^*\cG)
\end{align*}
where the first two rows are related by the adjunction \cite[Equation~(2.6.15)]{KashiSchapi}; and the third and fourth rows can be compared directly or by applying the adjunction \cite[Equation~(2.6.9)]{KashiSchapi} twice, noting that $((i_*\cF)_O)_O \simeq (i_*\cF)_O$ by \cite[Proposition~2.3.6]{KashiSchapi}, and then applying the same adjunction once more.  Another adjunction gives
\[ R\sHom(i_!\cF,\cG) \simeq Ri_* R\sHom(i^*i_*\cF, i^*\cG) \simeq Ri_* R\sHom(\cF, i^*\cG) \]
as claimed.  Now if $\cF = \csheaf[V][O]$ and $i^*\cG \simeq \csheaf[W][O]$ for some vector space $W$, then $R\sHom(\cF, i^*\cG)$ is the constant sheaf on $O$ with stalk $H=\Hom(V,W)$ by Lemma~\ref{lem:rhom-constant}, and then Lemma~\ref{lem:rpushforward-inclusion} says that $R\sHom(i_!\csheaf[V][O],\cG) \simeq Ri_*\csheaf[H][O] \simeq i_*\csheaf[H][O]$, as desired.
\end{proof}
}

\begin{prop} \label{prop:rhom-strand-0}
If $U$ is a sufficiently small neighborhood of a point on the interior of a strand of $\Lambda$ with Maslov potential 0, then $R\sHom(\cF|_U,\cG|_U) \simeq \sHom(\cF|_U,\cG|_U)$.
\end{prop}

\begin{proof}
We identify $U$ with the open unit disk in the $xz$-plane, consisting of three strata: the open sets $N = \{z > 0\}$ and $S = \{z < 0\}$, and the strand $\{z=0\}$, as shown in Figure~\ref{fig:strand-nbhd}.
\begin{figure}
\begin{tikzpicture}
\def\nbhdradius{1.5}
\draw[dashed] (0,0) circle (\nbhdradius);
\draw (-\nbhdradius,0) -- (\nbhdradius,0) node [anchor=west] {$\Lambda=\{z=0\}$};
\node at (0,\nbhdradius/2) {$N$};
\node at (0,-\nbhdradius/2) {$S$};
\end{tikzpicture}
\caption{A neighborhood $U$ of a point on a strand of $\Lambda$.}
\label{fig:strand-nbhd}
\end{figure}
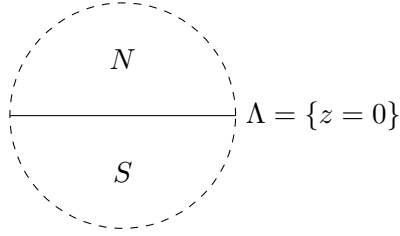
Then $\cF|_U$ and $\cG|_U$ are represented by quiver representations
\[ A \xleftarrow{\sim} A \xrightarrow{f} B \qquad \mathrm{and} \qquad V \xleftarrow{\sim} V \xrightarrow{g} W, \]
respectively, and $f$ and $g$ are both injective since the Maslov potential is 0.  Letting $C = \coker(f)$ and $D = \coker(g)$, we have exact sequences of sheaves on $U$ of the form
\begin{align}
0 \to \csheaf[A][U] \to \cF \to i_!\csheaf[C][N] \to 0 \label{eq:rhom-0-fseq} \\
0 \to \csheaf[V][U] \to \cG \to i_!\csheaf[D][N] \to 0, \label{eq:rhom-0-gseq}
\end{align}
where $i: N \hookrightarrow U$ is the inclusion, and applying $R\sHom(\cF,\cdot)$ to \eqref{eq:rhom-0-gseq} gives an exact sequence
\[ 0 \to R\sHom(\cF,\csheaf[V][U]) \to R\sHom(\cF,\cG) \to R\sHom(\cF,i_!\csheaf[D][N]) \to 0. \]

We first claim that $R\sHom(\cF,\csheaf[V][U])$ vanishes in positive degree: applying the contravariant functor $R\sHom(\cdot,\csheaf[V][U])$ to \eqref{eq:rhom-0-fseq} gives an exact sequence 
\[ 0 \to R\sHom(i_!\csheaf[C][N],\csheaf[V][U]) \to R\sHom(\cF,\csheaf[V][U]) \to R\sHom(\csheaf[A][U],\csheaf[V][U]) \to 0. \]
Applying Lemmas~\ref{lem:rhom-lower-shriek} and \ref{lem:rhom-constant} to the first and third terms respectively, we see that this is equivalent to
\[ 0 \to \sHom(i_!\csheaf[C][N],\csheaf[V][U]) \to R\sHom(\cF,\csheaf[V][U]) \to \sHom(\csheaf[A][U],\csheaf[V][U]) \to 0, \]
so that $R\sHom(\cF,\csheaf[V][U]) \simeq \sHom(\cF,\csheaf[V][U])$ as claimed.

We now wish to show that $R^j\sHom(\cF,i_!\csheaf[D][N]) = 0$ for $j>0$, so we apply $R\sHom(\cdot,i_!\csheaf[D][N])$ to \eqref{eq:rhom-0-fseq} to get
\[ 0 \to R\sHom(i_!\csheaf[C][N], i_!\csheaf[D][N]) \to R\sHom(\cF,i_!\csheaf[D][N]) \to R\sHom(\csheaf[A][U], i_!\csheaf[D][N]) \to 0. \]
We can apply Lemma~\ref{lem:rhom-lower-shriek} to $R\sHom(i_!\csheaf[C][N], i_!\csheaf[D][N])$, since $\left.i_!\csheaf[D][N]\right|_N$ is a constant sheaf, and $R\sHom(\csheaf[A][U],i_!\csheaf[D][N])$ vanishes in positive degrees by Lemma~\ref{lem:rhom-constant}, so the same is true of $R\sHom(\cF,i_!\csheaf[D][N])$ and we are done.
\end{proof}

\begin{prop} \label{prop:rhom-strand-1}
If $U$ is a sufficiently small neighborhood of a point on the interior of a strand of $\Lambda$ with Maslov potential 1, then $R\sHom(\cF|_U,\cG|_U) \simeq \sHom(\cF|_U,\cG|_U)$.
\end{prop}

\begin{proof}
Just as in Proposition~\ref{prop:rhom-strand-0}, we identify $U$ as the open unit disk with stratification $N \cup \{z=0\} \cup S$; then $\cF$ and $\cG$ correspond to quiver representations just as before, except that the Maslov potential is $1$ and so $f$ and $g$ are both surjective.  Let $K = \ker(f)\subset A$ and $L = \ker(g) \subset V$.  Then we have two short exact sequences of sheaves of the form
\begin{align}
0 \to i_*\csheaf[K][S] &\to \cF \to \csheaf[B][U] \to 0 \label{eq:rhom-1-fseq} \\
0 \to i_*\csheaf[L][S] &\to \cG \to \csheaf[W][U] \to 0, \label{eq:rhom-1-gseq}
\end{align}
where $i: S \hookrightarrow U$ is the inclusion map, and applying $R\sHom(\cF,\cdot)$ to \eqref{eq:rhom-1-gseq} gives
\[ 0 \to R\sHom(\cF, i_*\csheaf[L][S]) \to R\sHom(\cF, \cG) \to R\sHom(\cF,\csheaf[W][U]) \to 0. \]

We have $Ri_*\csheaf[L][S] \simeq i_*\csheaf[L][S]$ by Lemma~\ref{lem:rpushforward-inclusion}.    We also observe that $i^*\cF = \csheaf[A][S]$, so the adjunction $R\sHom(\cF,Ri_*\csheaf[L][S]) \simeq Ri_*R\sHom(i^*\cF,\csheaf[L][S])$ of \cite[Equation~(2.6.15)]{KashiSchapi} becomes
\[ R\sHom(\cF,i_*\csheaf[L][S]) \simeq Ri_*R\sHom(\csheaf[A][S], \csheaf[L][S]) \simeq Ri_*\csheaf[E][S] \simeq i_*\csheaf[E][S], \]
where $E = \Hom(A,L)$, by Lemma~\ref{lem:rhom-constant}.  We will show that $R\sHom(\cF,\csheaf[W][U])$ is supported in degree zero as well, which proves the proposition.

Applying $R\sHom(\cdot,\csheaf[W][U])$ to \eqref{eq:rhom-1-fseq}, we get an exact sequence
\[ 0 \to R\sHom(\csheaf[B][U], \csheaf[W][U]) \to R\sHom(\cF,\csheaf[W][U]) \to R\sHom(i_*\csheaf[K][S], \csheaf[W][U]) \to 0. \]
Again $R\sHom(\csheaf[B][U], \csheaf[W][U])$ vanishes in positive degree by Lemma~\ref{lem:rhom-constant}, so it suffices to show that $R\sHom(i_*\csheaf[K][S],\csheaf[W][U])$ does as well.  Letting $j: N \hookrightarrow U$ denote the inclusion map, we apply $R\sHom(\cdot,\csheaf[W][U])$ to the short exact sequence $0 \to j_!\csheaf[K][N] \to \csheaf[K][U] \to i_*\csheaf[K][S] \to 0$ to get another short exact sequence
\[ 0 \to R\sHom(i_*\csheaf[K][S],\csheaf[W][U]) \to \sHom(\csheaf[K][U],\csheaf[W][U]) \to \sHom(j_!\csheaf[K][N], \csheaf[W][U]) \to 0, \]
where in each of the last two terms we can replace $R\sHom$ with $\sHom$ by Lemmas~\ref{lem:rhom-constant} and \ref{lem:rhom-lower-shriek} respectively.  It follows that $R^q\sHom(i_*\csheaf[K][S],\csheaf[W][U]) = 0$ for $q\geq 2$, and that the sequence
\[ \sHom(\csheaf[K][U],\csheaf[W][U]) \to \sHom(j_!\csheaf[K][N],\csheaf[W][U]) \to R^1\sHom(i_*\csheaf[K][S],\csheaf[W][U]) \to 0 \]
is exact; but the first map in this last sequence is surjective, as can easily be seen on stalks, so $R^1\sHom(i_*\csheaf[K][S],\csheaf[W][U]) =0$ as well and we are done.
\end{proof}

Using these preliminary results, we can now prove Proposition~\ref{prop:rhom-equals-hom}. 

\begin{proof}[Proof of Proposition~\ref{prop:rhom-equals-hom}]
We have seen that the desired result can be checked locally, and Lemma~\ref{lem:rhom-constant} and Propositions~\ref{prop:rhom-strand-0} and \ref{prop:rhom-strand-1} collectively show that for all $q\geq 1$, the sheaf $R^q\sHom(\cF,\cG)$ is a direct sum of skyscraper sheaves supported on the crossings and cusps of $\Lambda$.  This implies that in the local-to-global Ext spectral sequence
\[ E_2^{p,q} = H^p(\R^2; R^q\sHom(\cF,\cG)) \ssabut \Ext^{p+q}(\cF,\cG) \]
of \cite[Th\'{e}or\`{e}me~II.7.3.3]{Godement}, all of the groups $E^{p,q}_2$ with $p,q\geq 1$ must vanish.  We have also shown that $E_2^{p,0} = H^p(\R^2;\sHom(\cF,\cG))$ vanishes for all $p\geq 2$, by Proposition~\ref{prop:h2shom-vanishes}.  Since the $E_2$ page is supported on the union of the line $p=0$ and the point $(p,q)=(1,0)$, the sequence collapses at this page.

If $\Lambda$ is the rainbow closure of an $n$-stranded braid, then the standard Legendrian unknot can be realized as the closure $\Lambda_0$ of the $n$-stranded braid $\sigma_1\sigma_2\dots\sigma_{n-1}$, where each $\sigma_i$ exchanges the $i$th and $(i+1)$st strands from the top.  If $U$ is a small neighborhood of a cusp of $\Lambda$, or a crossing of $\Lambda$ corresponding to the braid generator $\sigma_k$, then we let $U_0$ be the analogous neighborhood of the corresponding left or right cusp or $\sigma_k$-crossing of $\Lambda_0$.  It is easy to construct sheaves of $\kk$-modules $\cF_0 \in \Sh_r(\Lambda_0,\coeffs)$ and $\cG_0 \in \Sh_s(\Lambda_0,\coeffs)$ such that $\cF_0|_{U_0}$ and $\cG_0|_{U_0}$ are identical to $\cF|_U$ and $\cG|_U$ respectively; since $R\sHom(\cF,\cG)|_U \simeq R\sHom(\cF_0,\cG_0)|_{U_0}$, we will study the latter in order to determine the stalks of each $R^q\sHom(\cF,\cG)$ at the cusp or crossing in question.

Let $\Lambda_{\std}$ be the standard eye-shaped Legendrian unknot diagram, and note that $\Sh_r(\Lambda_0,\coeffs) \cong \Sh_r(\Lambda_{\std},\coeffs)$ and $\Sh_s(\Lambda_0,\coeffs) \cong \Sh_s(\Lambda_{\std},\coeffs)$ since $\Lambda_0$ and $\Lambda_{\std}$ are Legendrian isotopic; let $\cF_1 \in \Sh_r(\Lambda_{\std},\coeffs)$ and $\cG_1 \in \Sh_s(\Lambda_{\std},\coeffs)$ be sheaves corresponding to $\cF_0$ and $\cG_0$.  The functor $\Sh(\Lambda_\std,\coeffs) \to D(\kk{-}\mathrm{mod})$ sending a sheaf to its stalk inside the eye is an equivalence of categories \cite[Example~3.5]{STZ}, so we have $\Ext^\bullet(\cF_1,\cG_1) = \Ext^\bullet_{\kk}(\kk^r,\kk^s) = \Hom(\kk^r,\kk^s) \cong \kk^{rs}$.  We conclude that $\Ext^i(\cF_0,\cG_0) = 0$ for $i \geq 1$.

We now apply the above discussion of the Ext spectral sequence to the computation of $\Ext^i(\cF_0,\cG_0)$.  Since we know that these groups vanish for $i \geq 1$, examining the $E_2=E_\infty$ page tells us that 
\[ H^0(\R^2;R^q\sHom(\cF_0,\cG_0)) = E_2^{0,q} = 0 \]
for all $q \geq 1$.  But $R^q\sHom(\cF_0,\cG_0)$ is a sum of skyscraper sheaves for $q \geq 1$, so its zeroth cohomology is just the direct sum of its stalks at each cusp and crossing.  Thus the stalks of $R^q\sHom(\cF_0,\cG_0)$ all vanish for $q \geq 1$, and this implies as above that
\[ R^q\sHom(\cF|_U,\cG|_U) \simeq R^q\sHom(\cF_0|_{U_0},\cG_0|_{U_0}) = 0. \]
So we must have $R^q\sHom(\cF|_U,\cG|_U) = 0$ for all $q \geq 1$, and since this is true for a neighborhood $U$ of any cusp or crossing, we can conclude that $R^q\sHom(\cF,\cG) = 0$ for all $q \geq 1$, as desired.
\end{proof}

\section{Construction of the equivalence}
\label{sec:constr-equiv}
We are now ready to complete the proof of Theorem~\ref{thm:princ}.
\begin{proof}[Proof of Theorem \ref{thm:princ}]
We want to find an equivalence functor $$F: H^*(\Rep_n(\Lambda_m,\coeffs)) \rightarrow H^*(\Sh_n(\Lambda_m,\coeffs))$$
which is fully faithful and essentially surjective.

We first define $F$ on objects. To a representation described by $(A_1,\cdots, A_m)$ one associates the sheaf where $V=\coeffs^n$ and the maps are given by $(A_1^T,\cdots, A_m^T)$.  (Note that by equations~\eqref{eq:8} and~\eqref{eq:8a} we have
\[ P_m(A_1,\cdots,A_m)^T=P_m(A_m^T,\cdots,A_1^T)=(-1)^mQ_m(A_1,\cdots,A_m), \]
and thus Proposition~\ref{prop:sylvester} implies that this is indeed an object of $\Sh_n(\Lambda_m,\coeffs)$.) Equation~\eqref{eq:9} and Proposition~\ref{PropObjSH} imply that $F$ is essentially surjective on objects. 

Then to each morphism $(u_1,u_2): (A_1,\cdots,A_m) \to (A_1',\cdots,A_m')$, we associate the morphism $-(u_2^T,u_1^T)$. The first part of Proposition \ref{propReptorus} together with Proposition \ref{prp:ext0} implies that this is a well defined isomorphism on $\Hom^0$.

To a class $(w_1,\cdots,w_m)$ in $H^1\Hom(\rho,\rho')$, we associate the class of extensions given by $(w_1^T,\cdots,w_m^T)$.  The second part of Proposition~\ref{propReptorus} and Proposition~\ref{prp:ext1} show that this is a well defined isomorphism on $H^1\Hom$.  By Propositions~\ref{propReptorus} and~\ref{prop:ext2-vanishes}, all other $H^i\Hom$ groups ($i\neq 0,1$) in both categories are zero, so $F$ defines an isomorphism on all of $H^*\Hom$.

It now follows from the computation in Section~\ref{sec:compositions-sh-} and equations \eqref{eq:7}, \eqref{eq:5} and \eqref{eq:3} that the composition rule is preserved. Hence $F$ is a functor, and so it defines an equivalence of categories. 
\end{proof}

\bibliographystyle{alpha}
\bibliography{biblio}

\end{document}